\documentclass[11pt,a4paper,english,reqno,a4paper]{amsart}
\usepackage{amsmath,amssymb,amsthm, comment,graphicx}
\usepackage{tikz}
\usepackage{fancyhdr}
\usepackage{epsfig}
\usepackage{mathrsfs}
\newtheorem{theorem}{Theorem}[section]
\newtheorem{definition}{Definition}[section]
\newtheorem{lemma}{Lemma}[section]
\newtheorem{remark}{Remark}[section]
\newtheorem{proposition}{Proposition}[section]
\newtheorem{corollary}{Corollary}[section]

\numberwithin{equation}{section}

\newcommand{\R}{{\mathbb R}}

\begin{document}
\title[Convergence Rate of the Hypersonic Similarity for Steady Potential Flows]
{Convergence Rate of Hypersonic Similarity for Steady Potential Flows Over Two-Dimensional Lipschitz Wedge}

\author{Jie Kuang }
\address{ Innovation Academy for Precision Measurement Science and
Technology, Chinese Academy of Sciences, Wuhan 430071, China;
Wuhan Institute of Physics and Mathematics,
Chinese Academy of Sciences,  Wuhan 430071, China}
\email{\tt jkuang@wipm.ac.cn,\ jkuang@apm.ac.cn }

\author{Wei Xiang}
\address{Department of Mathematics
City University of Hong Kong
Kowloon, Hong Kong, China}
\email{\tt  weixiang@cityu.edu.hk}

\author{Yongqian Zhang}
\address{School of Mathematical Sciences, Fudan University, Shanghai 200433, China}
\email{\tt  yongqianz@fudan.edu.cn}


\date{}

\begin{abstract}
This paper is devoted to establishing the convergence rate of the hypersonic similarity for the inviscid steady irrotational Euler flow over a two-dimensional Lipschitz slender wedge in $BV\cap L^1$ space. The rate we established
is the same as the one predicted by Newtonian-Busemann law (see (3.29) in  \cite[Page 67]{anderson} for more details) 
as the incoming Mach number $\textrm{M}_{\infty}\rightarrow\infty$ for a fixed hypersonic similarity parameter $K$.
The hypersonic similarity, which is also called the Mach-number independence principle, is equivalent to the following Van Dyke's similarity theory: For a given hypersonic similarity parameter $K$,
when the Mach number of the flow is sufficiently large, the governing equations after the scaling are approximated by a simpler equation, that is called the hypersonic small-disturbance equation.
To achieve the convergence rate, we approximate the curved boundary by piecewisely straight lines and find a new Lipschitz continuous map $\mathcal{P}_{h}$ such that the trajectory can be obtained by piecing together the Riemann solutions near the approximated boundary. Next, we derive the $L^1$ difference estimates between the approximate solutions $U^{(\tau)}_{h,\nu}(x,\cdot)$
to the initial-boundary value problem for the scaled equations and the trajectories $\mathcal{P}_{h}(x,0)(U^{\nu}_{0})$ by piecing together all the Riemann solvers.
Then, by the uniqueness and the compactness of $\mathcal{P}_{h}$ and $U^{(\tau)}_{h,\nu}$,
we can further establish the $L^1$ estimates of order $\tau^2$ between the solutions to the initial-boundary value problem for the scaled equations and the solutions to the initial-boundary value problem for the hypersonic small-disturbance equations, if the total variations of the initial data and the tangential derivative of the boundary are sufficiently small.
Based on it, we can further establish a better convergence rate by considering the hypersonic flow past a two-dimensional Lipschitz slender wing and show that for the length of the wing with the effect scale order $O(\tau^{-1})$,
that is, the $L^1$ convergence rate between the two solutions is of order $O(\tau^{\frac{3}{2}})$
under the assumption that the initial perturbation 
has compact support.

\end{abstract}

\maketitle

\section{Introduction and Main result}\setcounter{equation}{0}
In this paper, as shown in Fig.\ref{fig1.1}, we will continue our recent work \cite{kuang-xiang-zhang}
to establish the convergence rate of the hypersonic similarity based on the problem of a uniformly hypersonic flow over a two-dimensional Lipschitz slender wedge governed by two-dimensional isentropic irrotational inviscid steady Euler equations:
\begin{eqnarray}\label{eq:1.1}
\left\{
\begin{array}{llll}
\partial_{x}(\rho u)+\partial_{y}(\rho v)=0,\\[5pt]
\partial_{x}v-\partial_{y}u=0,
\end{array}
\right.
\end{eqnarray}
together with the Bernoulli's law:
\begin{eqnarray}\label{eq:1.2}
\frac{1}{2}(u^2+v^2)+\frac{\rho^{\gamma-1}-1}{\gamma-1}=
\frac{1}{2}u_{\infty}^2+\frac{\rho_{\infty}^{\gamma-1}-1}{\gamma-1},
\end{eqnarray}
where $\rho$, $u$ and $v$ stand for the density, the horizontal and vertical velocities, respectively.
\vspace{5pt}
\begin{figure}[ht]
\begin{center}
\begin{tikzpicture}[scale=0.9]
\draw [thick](-2,1)--(2.7,1);
\draw [line width=0.06cm] (-2,1)to[out=20, in=180](2.7,1.8);
\draw [line width=0.06cm] (-2,1)to[out=-20, in=-180](2.7,0.2);
\draw [line width=0.02cm][red] (-2,1)to [out=30, in=180](2.5,2.4);
\draw [line width=0.02cm][red] (-2,1)to [out=-30, in=-180](2.5,-0.4);
\draw [thin][->](-4.0,1.8)--(-2.2,1.8);
\draw [thin][->](-4.0,1.3)--(-2.2,1.3);
\draw [thin][->](-4.0,0.8)--(-2.2,0.8);
\draw [thin][->](-4.0,0.3)--(-2.2,0.3);
\draw [thin](-1.5,0.9)--(-1.2,1.2);
\draw [thin](-1.3,0.85)--(-0.9,1.25);
\draw [thin](-1.1,0.80)--(-0.6,1.3);
\draw [thin](-0.9,0.75)--(-0.3,1.36);
\draw [thin](-0.6,0.7)--(0.2,1.5);
\draw [thin](-0.3,0.6)--(0.7,1.6);
\draw [thin](0,0.55)--(1.15,1.70);
\draw [thin](0.3,0.5)--(1.5,1.75);
\draw [thin](0.6,0.40)--(1.8,1.6);
\draw [thin](0.9,0.35)--(2.2,1.65);
\draw [thin](1.3,0.3)--(2.5,1.5);
\draw [thin](1.7,0.20)--(2.7,1.2);
\draw [line width=0.04cm][blue] (-1,1)to [out=-60, in=20](-1.1,0.7);
\node at (2.8,2.4) {$\mathcal{S}$};
\node at (2.8,-0.4) {$\mathcal{S}$};
\node at (-3.2,2.2) {$\textrm{M}_{\infty}\gg 1$};
\node at (-0.7,0.8) {$\theta$};
\node at (-2.0, -0.8) {$K=\textrm{M}_{\infty}\theta $ is a fixed constant};
\end{tikzpicture}
\end{center}
\caption{Hypersonic Potential Flows Past 2D Lipschitz Slender Wedge}\label{fig1.1}
\end{figure}
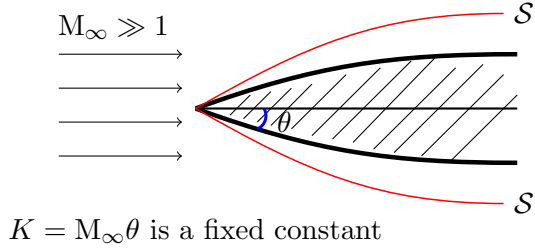

The Mach number $\textrm{M}$ is defined as
\begin{eqnarray}\label{eq-mach}
\textrm{M}=\frac{\sqrt{u^2+v^2}}{c},
\end{eqnarray}
where $c=\sqrt{\gamma \rho^{\gamma-1}}$ is the local sonic speed.

The flow is called hypersonic for $\textrm{M}>5$ (see Anderson \cite{anderson}), which is different from the supersonic flow (\cite{courant-friedrich}) and has many important applications in aerodynamics and engineering (see \cite{kuang-xiang-zhang}). In this paper, we consider the case that the Mach number $\textrm{M}_{\infty}$ of incoming flow is sufficiently large.
One of the useful observations to the hypersonic flow is the hypersonic similarity when the Mach number is sufficiently large: As shown in Fig. \ref{fig1.1}, let $\theta$ and $\textrm{M}_{\infty}$ be the wedge angle and Mach number of the incoming flow, respectively. The similarity parameter is defined as (see (127.3) in Landau-Lifschitz \cite[Page 482]{landau-lifschitz} for more details),
\begin{equation}\label{eq:1.4}
K=\textrm{M}_{\infty}\theta.
\end{equation}

Hypersonic similarity says that for any given fixed similarity parameter $K$, the flow structures are similar if $\textrm{M}_{\infty}$ is sufficiently large. Mathematically, it means that after scaling, the governed
equations for the same similarity parameter $K$ are approximated by the same hypersonic small-disturbance equations, which was first developed by Tsien in \cite{tsien} for the steady irrotational flow.
Due to the importance of this property from both the theoretical and experimental view (see \cite{anderson}), recently, we in \cite{kuang-xiang-zhang} justified Tsien's hypersonic similarity theory rigorously for the two-dimensional potential flow in the BV space but without a convergence rate. So in this paper, we will continue on this topic and find the convergence rate with respect to $\tau$.

Define
\begin{eqnarray}\label{eq-hypersonic}
a_{\infty}\doteq\tau \textrm{M}_{\infty}=\tau u_{\infty}\rho_{\infty}^{\frac{1-\gamma}{2}}.
\end{eqnarray}

Obviously, if $K$ is fixed, then $a_{\infty}$ is fixed too. So $a_{\infty}$
is also called the hypersonic similarity parameter (see Chapter 4 in \cite{anderson}).
Because the upper and lower half space domains can be treated similarly, let us only consider the lower half space domain, \emph{i.e.}, in the region that $x\geq0$ and $y\leq\tau b(x)$ with $b(x)<0$ in Fig.\ref{fig1.1}.
As done in \cite{kuang-xiang-zhang},
if $u_{\infty}$ is a sufficiently large number,
after defining the following scaling:
\begin{equation}\label{eq:1.6}
x=\bar{x},\quad y=\tau\bar{y}, \quad u=u_{\infty}(1+\tau^2\bar{u}) ,\quad v=u_{\infty}\tau \bar{v},\quad \rho=\rho_{\infty}\bar{\rho},
\end{equation}
we obtain by substituting \eqref{eq:1.6} into equations \eqref{eq:1.1} and \eqref{eq:1.2}, that
\begin{equation}\label{eq:1.7}
\begin{cases}
\partial_{\bar{x}}\big(\bar{\rho} (1+\tau^2\bar{u})\big)+\partial_{\bar{y}}(\bar{\rho} \bar{v})=0,\\[5pt]
\partial_{\bar{x}}\bar{v}-\partial_{\bar{y}}\bar{u}=0,\\[5pt]
\bar{u}+\displaystyle\frac{1}{2}(\bar{v}^2+\tau^2\bar{u}^2)+\frac{\bar{\rho}^{\gamma-1}-1}{(\gamma-1)a_{\infty}^2}=0.
\end{cases}
\end{equation}

Meanwhile, the corresponding fluid domain and
its boundary are given by (see Fig.\ref{fig1.2})
\begin{eqnarray*}
	\Omega=\{(\bar{x}, \bar{y}): \bar{x}>0,\ \bar{y}<b(\bar{x}) \},	
\end{eqnarray*}
and
\begin{eqnarray*}
	\Gamma=\{(\bar{x}, \bar{y}): \bar{x}>0,\ \bar{y}=b(\bar{x}) \}, \qquad \mathcal{I}=\{\bar{x}=0,\ \bar{y}\leq0\}.
\end{eqnarray*}

The initial condition is
\begin{eqnarray}\label{eq:1.8}
(\bar{\rho}, \bar{u}, \bar{v})=\big(\bar{\rho}_{0}, \bar{u}^{(\tau)}_{0},\bar{ v}_{0}\big)(\bar{y}),
&\quad\mbox{on}\quad\  \mathcal{I},
\end{eqnarray}
which satisfies equation $\eqref{eq:1.7}_{3}$. Along $\Gamma$, the boundary condition is
\begin{eqnarray}\label{eq:1.9}
\big((1+\tau^{2}\bar{u}), \bar{v}\big)\cdot \mathbf{n}=0, &\quad\mbox{on}\quad\  \Gamma,
\end{eqnarray}
where $\mathbf{n}=\mathbf{n}(\bar{x},b(\bar{x}))=\displaystyle {\frac{(b'(\bar{x}),-1)}{\sqrt{1+(b'(\bar{x}))^{2}}}}$ is the unit inner normal vector of $\Gamma$.

\vspace{5pt}
\begin{figure}[ht]
\begin{center}
\begin{tikzpicture}[scale=1.2]
\draw [line width=0.03cm][->] (-2,0)--(3.5,0);
\draw [line width=0.03cm][->] (-1,-2.7) --(-1,1);
\draw [line width=0.07cm](-1,0)to[out=-18, in=-185](3.2,-0.8);
\draw [line width=0.07cm](-1,0)--(-1,-2.7);
\draw [line width=0.02cm][red](-1,0)to[out=-27, in=-185](3.1,-1.5);
\draw [line width=0.04cm][blue](0.1,0)to [out=-60, in=20](0,-0.3);
\node at (3.4,-0.2) {$\bar{x}$};
\node at (-0.8,1) {$\bar{y}$};
\node at (-1.2,-0.2) {$O$};
\node at (3.5,-0.8) {$\Gamma$};
\node at (3.3,-1.5) {$\bar{\mathcal{S}}$};
\node at (-0.9,-2.9) {$\mathcal{I}$};
\node at (1.4,-0.25) {$\bar{\theta}=\arctan b'(\bar{x})$};
\node at (1.0,-1.5) {$\Omega$};
\node at (-1.9,-1.5) {$(\bar{\rho}_{0},\bar{u}_{0},\bar{ v}_{0})$};
\end{tikzpicture}
\end{center}
\caption{Hypersonic Similarity for 2D Steady Potential Flows}\label{fig1.2}
\end{figure}
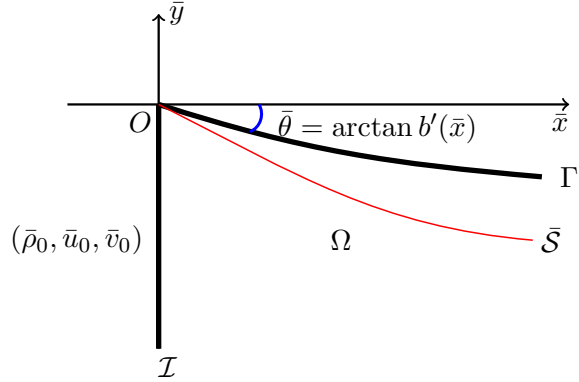

Mathematically, it follows from \eqref{eq:1.7} that, for a fixed parameter $a_{\infty}$, the hypersonic similarity means 
the structure of solutions of \eqref{eq:1.7}-\eqref{eq:1.9} should be similar and thus can be investigated by the hypersonic small-disturbance equations, derived via neglecting the terms involving $\tau^{2}$, 
\begin{equation}\label{eq:1.10}
\begin{cases}
\partial_{\bar{x}}\bar{\rho}+\partial_{\bar{y}}(\bar{\rho} \bar{v})=0, \\[5pt]
\partial_{\bar{x}}\bar{v}-\partial_{\bar{y}}\bar{u}=0, \\[5pt]
\bar{u}+\displaystyle \frac{1}{2}\bar{v}^2+\frac{\bar{\rho}^{\gamma-1}-1}{(\gamma-1)a_{\infty}^2}=0,
\end{cases}
\end{equation}
with initial data
\begin{eqnarray}\label{eq:1.10a}
(\bar{\rho}, \bar{u}, \bar{v})=\big(\bar{\rho}_{0}, \bar{u}_{0},\bar{ v}_{0}\big)(\bar{y}), &\quad\mbox{on}\quad\ \mathcal{I},
\end{eqnarray}
which satisfies the equation $\eqref{eq:1.10}_{3}$, and boundary condition that
\begin{eqnarray}\label{eq:1.11}
\bar{v}=b'(\bar{x}), &\quad\mbox{on}\quad\ \Gamma.
\end{eqnarray}

The hypersonic similarity is also called the Mach-number independence principle, which is equivalent to the Van Dyke's similarity theory(see \cite{dyke}).

Based on \cite{kuang-xiang-zhang}, the study of the two-dimensional steady irrotational hypersonic flow can be much simplified by studying of the simpler equations \eqref{eq:1.10}. 
Then, one followed-up question is whether we can find the 
convergence rate or not? 
It is useful to provide more accurate information in applications.

In this paper, we will answer this question by finding the 
convergence rate to the Van Dyke's similarity theory rigorously in $BV\cap L^1$ space.  The rate we established
is the same as the one predicted by Newtonian-Busemann law (see (3.29) in  \cite[Page 67]{anderson} for more details). First, it follows from the fact that the flow moves from the left to the right, \emph{i.e.}, $1+\tau^2\bar{u}>0$ and from \eqref{eq:1.7} that,
\begin{eqnarray}\label{eq:1.12}
\begin{split}
\bar{u}(\bar{\rho}, \bar{v}; \tau^{2})=\displaystyle\frac{1}{\tau^2}\Bigg(-1+\sqrt{1-\tau^2\Big(\bar{v}^2+\frac{2(\bar{\rho}^{\gamma-1}-1)}{(\gamma-1)a^{2}_{\infty}}\Big)}\Bigg).
\end{split}
\end{eqnarray}

Then, substituting \eqref{eq:1.12} into the first two equations of \eqref{eq:1.7}, we get
\begin{equation}\label{eq:1.13}
\begin{cases}
\partial_{\bar{x}}\big(\bar{\rho} (1+\tau^2\bar{u})\big)+\partial_{\bar{y}}(\bar{\rho} \bar{v})=0, &\quad\mbox{in}\quad \Omega, \\[5pt]
\partial_{\bar{x}}\bar{v}-\partial_{\bar{y}}\bar{u}=0, &\quad\mbox{in}\quad \Omega.
\end{cases}
\end{equation}

Similarly, it follows from \eqref{eq:1.10} that
\begin{equation}\label{eq:1.14}
\begin{cases}
\partial_{\bar{x}}\bar{\rho}+\partial_{\bar{y}}(\bar{\rho} \bar{v})=0,&\quad\mbox{in}\quad \Omega, \\[5pt]
\partial_{\bar{x}}\bar{v}+\partial_{\bar{y}}\displaystyle\bigg(\frac{1}{2}\bar{v}^2+\frac{\bar{\rho}^{\gamma-1}-1}{(\gamma-1)a_{\infty}^2}\bigg)=0,
&\quad\mbox{in}\quad \Omega,
\end{cases}
\end{equation}
where $(\bar{\rho},\bar{v})$ satisfies the initial condition \eqref{eq:1.10a} and the boundary condition \eqref{eq:1.11}.

To unify the notations in \eqref{eq:1.13} and \eqref{eq:1.14}, we rewrite $(\bar{\rho}, \bar{u}, \bar{v})$ as
$(\bar{\rho}^{(\tau)}, \bar{u}^{(\tau)}, \bar{v}^{(\tau)})$, with \eqref{eq:1.14} corresponding
to the case that $\tau=0$.
Let $U^{(\tau)}=(\bar{\rho}^{(\tau)}, \bar{v}^{(\tau)})^{\top}$ and
\begin{eqnarray}\label{eq:1.15}
G(U^{(\tau)}, \tau^{2})=\Big(\bar{\rho}^{(\tau)}\big(1+\tau^{2}\bar{u}^{(\tau)}\big), \bar{v}^{(\tau)}\Big)^{\top},\quad
F(U^{(\tau)}, \tau^{2})=\Big(\bar{\rho}^{(\tau)} \bar{v}^{(\tau)}, -\bar{u}^{(\tau)}\Big)^{\top}.
\end{eqnarray}

Then, \eqref{eq:1.13} and \eqref{eq:1.14} can be rewritten as
\begin{eqnarray}\label{eq:1.16}
\partial_{\bar{x}}G(U^{(\tau)}, \tau^{2})+\partial_{\bar{y}}F(U^{(\tau)}, \tau^{2})=0, &\quad \mbox{in}\quad \Omega,
\end{eqnarray}
with the initial condition
\begin{eqnarray}\label{eq:1.17}
U^{(\tau)}=U_{0}(y), &\quad \mbox{on}\quad\ \mathcal{I},
\end{eqnarray}
for $U_{0}=(\bar{\rho}_{0}, \bar{v}_{0})^{\top}$ and the boundary condition
\begin{eqnarray}\label{eq:1.18}
\Big((1+\tau^{2}\bar{u}^{(\tau)}),\bar{v}^{(\tau)}\Big)\cdot\mathbf{n}=0, &\quad \mbox{on}\quad\ \Gamma.
\end{eqnarray}

We finally remark that when $\tau=0$, $U=U^{(0)}=(\bar{\rho}, \bar{v})^{\top}$ and
\begin{eqnarray}\label{eq:1.19}
\begin{split}
G(U)=G(U,0)=U,\qquad F(U)=F(U,0)=\bigg(\bar{\rho} \bar{v},\ \displaystyle\frac{1}{2}\bar{v}^2+\frac{\bar{\rho}^{\gamma-1}-1}{(\gamma-1)a_{\infty}^2}\bigg)^{\top}.
\end{split}
\end{eqnarray}

Thus equation \eqref{eq:1.14} can be rewritten as the following hyperbolic conservation laws.
\begin{eqnarray}\label{eq:1.20}
\partial_{\bar{x}}U+\partial_{\bar{y}}F(U)=0, &\quad\mbox{in}\quad \Omega,
\end{eqnarray}
with the initial data
\begin{eqnarray}\label{eq:1.21}
U=U_{0}(y), &\quad \mbox{on}\quad\ \mathcal{I},
\end{eqnarray}
and the boundary condition \eqref{eq:1.11}.

A special case for the solutions of the initial-boundary value problems \eqref{eq:1.16}-\eqref{eq:1.18}, and \eqref{eq:1.20}-\eqref{eq:1.21} and \eqref{eq:1.11} are that
$b\equiv 0$ and $(\rho_{0}, u^{(\tau)}_0, v_0)(\bar{y})=\big(1, 0, 0\big)$, and $(\rho_{0}, u_0, v_0)(\bar{y})=\big(1, 0, 0\big)$.
In this case, the solutions are given by the same constant states
\begin{eqnarray}\label{eq-background}
\begin{split}
U^{(\tau)}=U=\underline{U}\doteq\big(1, 0\big)^{\top}, \quad \bar{u}^{(\tau)}=\bar{u}\equiv0, \quad   \ \forall (\bar{x}, \bar{y}) \in
\underline{\Omega},
\end{split}
\end{eqnarray}
where $\underline{\Omega}:=\{(\bar{x}, \bar{y}): \bar{x}>0, \bar{y}<0\}$. In this case, we call $\underline{U}$
as the \emph{background\ solution}.

Based on the background solution, we can introduce basic assumptions on the initial data $U_{0}(y)$ and boundary function $b(\bar{x})$
throughout the whole paper.\\

$\mathbf{(H1)}$ $U_{0}-\underline{U}\in (BV\cap L^{1})(\mathcal{I}; \mathbb{R}^{2})$ with $\inf_{\bar{y}\in \mathcal{I}}\bar{\rho}_0(\bar{y})>0$;\\

$\mathbf{(H2)}$ 
$b(\bar{x})\in 
Lip(\mathbb{R}_{+}; \mathbb{R})$ and $b'(\bar{x})\in (BV\cap L^{1})(\mathbb{R}_{+}; \mathbb{R})$ with
the properties that
\begin{equation}\label{eq:1.24x}
b(0)=0, \quad \ b'(0)\leq 0, \ \ \mbox{and}\ \ b(\bar{x})< 0, \ \ \mbox{for}\ \  \bar{x}>0,
\end{equation}
where $b'(\bar{x})$ represents the derivatives on the differential points of the boundary. Here, $Lip$ denotes the set of Lipschitz continuous functions.

Finally, before stating the main result of this paper, we introduce the definition of the entropy solutions of problem \eqref{eq:1.16}-\eqref{eq:1.18}.
\begin{definition}[Entropy solutions]\label{def:1.1}
	A weak solution $U^{(\tau)}\in \big(BV\cap L^{1}\big)((-\infty,b(x)); \mathbb{R}^2)$ of the initial-boundary value problem \eqref{eq:1.16}-\eqref{eq:1.18}
	is called an entropy solution, if for any convex entropy pair $(\mathcal{E},\mathcal{Q})$, that is, $\nabla \mathcal{Q}(G,\tau^2)=\nabla \mathcal{E}(G, \tau^{2})\nabla F(U(G),\tau^{2})$ and $\nabla^2\mathcal{E}(G,\tau^{2})\ge 0$, the following entropy inequality holds: For any $\phi \in C_0^{\infty}(\R^2)$ with $\phi\ge 0$,
	\begin{eqnarray}\label{eq:1.23}
	\begin{split}
	&\iint_{\Omega}\Big(\mathcal{E}(G,\tau^{2})\partial_{\bar{x}}\phi+ \mathcal{Q}(G, \tau^{2})\partial_{\bar{y}}\phi\Big)dxdy
	+ \int^{0}_{-\infty}\mathcal{E}(G_{0}, \tau^{2})\phi(0,y)dy\\[5pt]
	&\qquad+\int_{\Gamma}(\mathcal{E}(G,\tau^{2}),\mathcal{Q}(G,\tau^{2}))\cdot\mathbf{n}ds\geq 0,
	\end{split}
	\end{eqnarray}
	where $G_{0}=G(U_{0},\tau^{2})$ and $\mathbf{n}$ is the unit inner normal vector on boundary $\Gamma$.
\end{definition}

Then, the first main result for the wedge is stated as follows. 
\begin{theorem}\label{thm:1.1}
Assume $\mathbf{(H1)}$-$\mathbf{(H2)}$ hold. For a given fixed hypersonic similarity parameter $a_{\infty}$, let $U^{(\tau)}=(\bar{\rho}^{(\tau)}, \bar{v}^{(\tau)})^{\top}$ and $U=U^{(0)}=(\bar{\rho}, \bar{v})^{\top}$
be the global entropy solutions to the initial-boundary value problem \eqref{eq:1.16}-\eqref{eq:1.18} and
the initial-boundary value problem \eqref{eq:1.20}-\eqref{eq:1.21} and \eqref{eq:1.11}, respectively which are obtained by wave front tracking scheme.
There exist small parameters $\epsilon^{*}>0$ and $\tau^{*}>0$ depending only on $\underline{U}$ and $a_{\infty}$
such that for any $\tau\in(0,\tau^{*})$ and if
\begin{eqnarray}\label{eq:1.26}
T.V.\{U_{0}(\cdot); \mathcal{I}\}+|b'(0)|+T.V.\{b'(\cdot);\mathbb{R}_{+}\}<\epsilon,
\end{eqnarray}
for $\epsilon\in (0,\epsilon^*_{0})$, then it holds for $\bar{x}>0$ that
\begin{eqnarray}\label{eq:1.27}
\|U^{(\tau)}(\bar{x}, \cdot)-U(\bar{x}, \cdot)\|_{L^{1}((-\infty, b(\bar{x})))}
\leq C^*_1\bar{x} \tau^{2},
\end{eqnarray}
where constant $C^*_1>0$ depends on $T.V.\{U_{0}(\cdot); \mathcal{I}\}$, $|b'(0)|$, $T.V.\{b'(\cdot);\mathbb{R}_{+}\}$, $\underline{U}$, and $a_{\infty}$, but is independent on $\tau$ and $\bar{x}$.
Then, for $(\bar{\rho}^{(\tau)}, \bar{u}^{(\tau)}, \bar{v}^{(\tau)})$ being the solution
to problem \eqref{eq:1.7}-\eqref{eq:1.9}, and $(\bar{\rho}, \bar{u}, \bar{v})$ being the solution to problem \eqref{eq:1.10}-\eqref{eq:1.11}, there holds that
\begin{eqnarray}\label{eq:1.28}
\big\|\big(\bar{\rho}^{(\tau)}-\bar{\rho},\ \bar{u}^{(\tau)}-\bar{u},\ \bar{v}^{(\tau)}-\bar{v}\big)(\bar{x},\cdot)\big\|_{L^{1}((-\infty, b(\bar{x})))}
\leq C^*_2\bar{x} \tau^{2},
\end{eqnarray}
where constant $C^*_{2}>0$ depends on $T.V.\{U_{0}(\cdot); \mathcal{I}\}$, $|b'(0)|$, $T.V.\{b'(\cdot);\mathbb{R}_{+}\}$, $\underline{U}$ and $a_{\infty}$, but is independent on $\tau$ and $\bar{x}$.
\end{theorem}

Some remarks are given below.
\begin{remark}\label{rem:1.1}
The order of $\tau^2$ in the Theorem \ref{thm:1.1} for finite length of the wedge is consistent with the one predicted by 
Newtonian-Busemann law, who expected the error term is of order $\tau^2$ (see (3.29) in  \cite[Page 67]{anderson} for more details)
\end{remark}

\begin{remark}\label{rem:1.2}
When $\tau=0$, the convex entropy pair $(\mathcal{E}(G,\tau^{2}), \mathcal{Q}(G,\tau^{2}))$ can be taken of the form
\begin{eqnarray}\label{eq:1.29}
\mathcal{E}(U,0)=\frac{\rho v^{2}}{2}+\frac{\rho^{\gamma-1}-1}{a_{\infty }\gamma (\gamma-1)},\quad \ \
\mathcal{Q}(U,0)=v\mathcal{E}(U,0).
\end{eqnarray}
So the entropy solution $(\bar{\rho},\bar{v})$ of the initial-boundary value problem \eqref{eq:1.20}-\eqref{eq:1.21} and \eqref{eq:1.11}
satisfies the entropy inequality
\begin{eqnarray}\label{eq:1.30}
\partial_{\bar{x}}\mathcal{E}(U,0)+\partial_{\bar{y}}\mathcal{Q}(U,0)\leq 0,
\end{eqnarray}
in the distribution sense.
\end{remark}

\begin{remark}\label{rem:1.3}
Once the solution of problem \eqref{eq:1.16}-\eqref{eq:1.18} and solution of problem \eqref{eq:1.20}-\eqref{eq:1.21} and \eqref{eq:1.11} are obtained,
it is easy to obtain solution $(\bar{\rho}^{(\tau)},\bar{u}^{(\tau)},\bar{v}^{(\tau)})$ of problem \eqref{eq:1.7}-\eqref{eq:1.9} and solution $(\bar{\rho},\bar{u},\bar{v})$ of problem \eqref{eq:1.10}-\eqref{eq:1.11}
by solving $\bar{u}^{(\tau)}$ and $\bar{u}$ directly from equation \eqref{eq:1.12} and the third equation of \eqref{eq:1.10}, respectively.
\end{remark}

\par As an application of Theorem \ref{thm:1.1}, we can further establish a better convergence rate by considering the hypersonic flow past over a two-dimensional Lipschitz wing (see Fig. \ref{fig1.2'} below).
In the $(\bar{x},\bar{y})$-coordinates, let $b_{+}(\bar{x})$ and $b_{-}(\bar{x})$ be the functions of the upper and lower boundaries of the wing with
\begin{equation}\label{eq:1.31}
b_{+}(\bar{x})> 0, \ \ b_{-}(\bar{x})<0, \ \  \mbox{and} \ \  b_{\pm}(0)=b_{\pm}(\ell_{w})=0, \quad  \bar{x}\in [0,\ell_{w}],
\end{equation}
for some constant $\ell_{w}>0$, and the corresponding boundaries for the wing are defined by
\begin{eqnarray*}
\Gamma_{+}=\big\{(\bar{x}, \bar{y}): \bar{y}=b_{+}(\bar{x}), \bar{x}\in [0,\ell_{w}] \big\},\quad \Gamma_{-}=\big\{(\bar{x}, \bar{y}): \bar{y}=b_{-}(\bar{x}), \bar{x}\in [0,\ell_{w}] \big\}.
\end{eqnarray*}

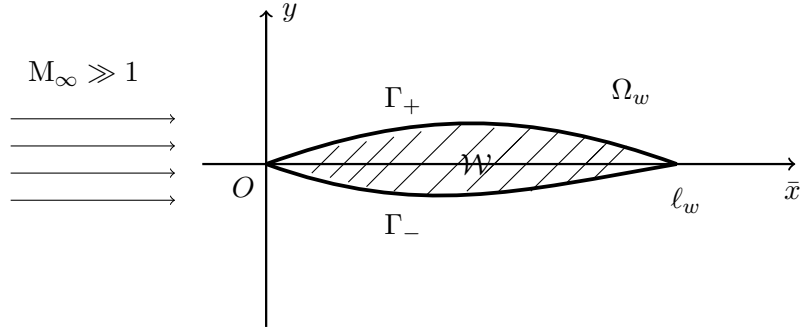
\begin{figure}[ht]
\begin{center}
\begin{tikzpicture}[scale=1.2]
\draw [line width=0.03cm][->](-2.7,1)--(3.8,1);
\draw [line width=0.03cm][->](-2,-0.8)--(-2,2.7);
\draw [line width=0.05cm] (-2,1)to[out=20, in=160](2.5,1.0);
\draw [line width=0.05cm] (-2,1)to[out=-20, in=-170](2.5,1.0);
\draw [thin][->](-4.8,1.5)--(-3.0,1.5);
\draw [thin][->](-4.8,1.2)--(-3.0,1.2);
\draw [thin][->](-4.8,0.9)--(-3.0,0.9);
\draw [thin][->](-4.8,0.6)--(-3.0,0.6);

\draw [thin](-1.5,0.9)--(-1.2,1.2);
\draw [thin](-1.3,0.85)--(-0.9,1.25);
\draw [thin](-1.1,0.80)--(-0.6,1.3);
\draw [thin](-0.9,0.75)--(-0.3,1.36);
\draw [thin](-0.6,0.7)--(0.15,1.45);
\draw [thin](-0.25,0.65)--(0.5,1.4);
\draw [thin](0.2,0.7)--(0.9,1.4);
\draw [thin](0.4,0.9)--(0.8,1.3);
\draw [thin](0.55,0.7)--(1.2,1.35);
\draw [thin](0.9,0.7)--(1.5,1.3);
\draw [thin](1.3,0.8)--(1.75,1.25);
\draw [thin](1.5, 1.0)--(1.7,1.2);
\draw [thin](1.6, 0.85)--(1.95, 1.20);
\node at (1.5, 1.5) {$$};
\node at (3.77,0.7) {$\bar{x}$};
\node at (-1.74,2.7) {$\bar{y}$};

\node at (2.0,1.8) {$\Omega_{w}$};
\node at (0.3,1.0) {$\mathcal{W}$};

\node at (-0.5,1.7) {$\Gamma_{+}$};
\node at (-0.5,0.3) {$\Gamma_{-}$};
\node at (2.6,0.6) {$\ell_{w}$};

\node at (-4.0,2.0) {$\textrm{M}_{\infty}\gg 1$};
\node at (-2.25,0.75) {$O$};
\end{tikzpicture}
\end{center}
\caption{Hypersonic Similarity for Potential Flows Past 2D Slender Wing}\label{fig1.2'}
\end{figure}

\par Then, for given functions $b_{\pm}(\bar{x})$, we can defined the domain which is formed by the wing and the fluids domain in following:
\begin{eqnarray*}
\mathcal{W}=\big\{(\bar{x}, \bar{y}): b_{-}(\bar{x})<\bar{y}< b_{+}(\bar{x}), \ \bar{x}\in[0, \ell_{w}] \big\},\quad  \Omega_{w}=(\mathbb{R}_{+}\times \mathbb{R})\setminus \mathcal{W}.
\end{eqnarray*}

The initial condition is
\begin{eqnarray}\label{eq:1.32}
(\bar{\rho}^{(\tau)}, \bar{u}^{(\tau)}, \bar{v}^{(\tau)})=(\bar{\rho}_{0}, \bar{u}^{(\tau)}_{0}, \bar{v}_{0})(\bar{y}), &\quad\mbox{on}\quad\ \mathbb{R},
\end{eqnarray}
where $\bar{\rho}_{0}$, $\bar{u}^{(\tau)}_{0}$ and $\bar{v}_{0}$ satisfy the equation $\eqref{eq:1.7}_{3}$ on $\mathbb{R}$.

On $\Gamma_{\pm}$, the boundary conditions are
\begin{eqnarray}\label{eq:1.33}
\big((1+\tau^{2}\bar{u}^{(\tau)}), \bar{v}^{(\tau)}\big)\cdot \mathbf{n}_{\pm}=0, &\quad\mbox{on}\quad\ \Gamma_{\pm},
\end{eqnarray}
where $\mathbf{n}_{\pm}=\mathbf{n}(\bar{x},b_{\pm}(\bar{x}))=\displaystyle\frac{(b'_{\pm}(\bar{x}),-1)}{\sqrt{1+(b'_{\pm}(\bar{x}))^{2}}}$ are the unit inner normal vectors of $\Gamma_{\pm}$.

Therefore, we can formulate the mathematical problem for the hypersonic flow moving over two-dimensional wing as the initial-boundary value problem \eqref{eq:1.32} and \eqref{eq:1.33}
to the following equations:
\begin{equation}\label{eq:1.34}
\begin{cases}
\partial_{\bar{x}}G(U^{(\tau)}, \tau^{2})+\partial_{\bar{y}}F(U^{(\tau)}, \tau^{2})=0, &\quad \mbox{in}\quad \Omega_{w},\\[5pt]
\bar{u}+\displaystyle\frac{1}{2}(\bar{v}^2+\tau^2\bar{u}^2)+\frac{\bar{\rho}^{\gamma-1}-1}{(\gamma-1)a_{\infty}^2}=0, &\quad \mbox{in}\quad \Omega_{w},
\end{cases}
\end{equation}
where $U^{(\tau)}=(\bar{\rho}^{(\tau)}, \bar{v}^{(\tau)})^{\top}$, and $G(U^{(\tau)}, \tau^{2})$, $F(U^{(\tau)}, \tau^{2})$ are defined by \eqref{eq:1.15}.

\par For $\tau=0$, the corresponding initial and boundary conditions are

\begin{eqnarray}\label{eq:1.35}
(\bar{\rho}, \bar{u}, \bar{v})=(\bar{\rho}_{0}, \bar{u}_{0}, \bar{v}_{0})(\bar{y}), &\quad\mbox{on}\quad \ \mathbb{R},
\end{eqnarray}
and
\begin{eqnarray}\label{eq:1.36}
\bar{v}=b'_{\pm}(\bar{x}), &\quad\mbox{on}\quad \ \Gamma_{\pm},
\end{eqnarray}
for $\bar{\rho}_{0}$, $\bar{u}_{0}$ and $\bar{v}_{0}$ satisfying $\eqref{eq:1.10}_{3}$ on $\mathbb{R}$.

The corresponding government equations for $\tau=0$ are given by
\begin{equation}\label{eq:1.37}
\begin{cases}
\partial_{\bar{x}}U+\partial_{\bar{y}}F(U)=0, &\quad \mbox{in}\quad \Omega_{w},\\[5pt]
\bar{u}+\displaystyle \frac{1}{2}\bar{v}^2+\frac{\bar{\rho}^{\gamma-1}-1}{(\gamma-1)a_{\infty}^2}=0, &\quad \mbox{in}\quad \Omega_{w},
\end{cases}
\end{equation}
for $U=(\bar{\rho}, \bar{v})^{\top}$, and $F(U)$ is defined by \eqref{eq:1.19}.

To unified the notations, let's denote the solutions to the initial-boundary value problem \eqref{eq:1.34} and \eqref{eq:1.32}-\eqref{eq:1.33} by $(\bar{\rho}^{(\tau)}_{w}, \bar{u}^{(\tau)}_{w}, \bar{v}^{(\tau)}_{w})$, and the solutions to the initial-boundary value problem
\eqref{eq:1.37} and \eqref{eq:1.35}-\eqref{eq:1.36} by $(\bar{\rho}_{w}, \bar{u}_{w}, \bar{v}_{w})$.

Set
\begin{eqnarray*}
U^{(\tau)}_{w}=(\bar{\rho}^{(\tau)}_{w},\ \bar{v}^{(\tau)}_{w})^{\top}, \qquad U_{w}=(\bar{\rho}_{w},\ \bar{v}_{w})^{\top},
\end{eqnarray*}
and
\begin{eqnarray*}
 U_{w,0}(\bar{y})=(\bar{\rho}_{0}, \bar{v}_{0})^{\top}(\bar{y}), \quad\ \mbox{for}\quad   \bar{y}\in \mathbb{R}.
\end{eqnarray*}

Then, $U^{(\tau)}_{w}$ and
$U_{w}$ are the solutions to the problem $\eqref{eq:1.34}_{1}$\eqref{eq:1.33} as well as the problem $\eqref{eq:1.37}_{1}$\eqref{eq:1.36}
with the same initial data $U_{w,0}(\bar{y})$.

When 
\begin{eqnarray*}
U_{w,0}(\bar{y})\equiv (1,0)^{\top},\quad\ \mbox{and}\quad\  b_{\pm}(\bar{x})\equiv 0,
\end{eqnarray*}
then the solutions to the initial-boundary value problem \eqref{eq:1.34} and \eqref{eq:1.32}-\eqref{eq:1.33} as well as the initial-boundary value problem
\eqref{eq:1.37} and \eqref{eq:1.35}-\eqref{eq:1.36} in this case are also the constant states, \emph{i.e.},
\begin{eqnarray}\label{eq:1.38}
\begin{split}
U^{(\tau)}_{w}=U_{w}=\underline{U}_{w}\doteq\big(1, 0\big)^{\top}, \quad\ \bar{u}^{(\tau)}_{w}=\bar{u}_{w}\equiv0, \quad\   \ \forall (\bar{x}, \bar{y}) \in
\underline{\Omega}_{w},
\end{split}
\end{eqnarray}
where $\underline{\Omega}_{w}=(\mathbb{R}_{+}\times \mathbb{R})\setminus \big\{(\bar{x}, \bar{y}): \bar{y}=0, \bar{x}\in [0,\ell_{w}]\big\}$.

Before given the second main result of the paper, we list some basic assumptions on the initial data $U_{w,0}(\bar{y})$ and boundary functions $b_{\pm}(\bar{x})$ in the following.\\

$\widetilde{\mathbf{(H1)}}$ The initial data $U_{w,0}(\bar{y})$ satisfies $U_{w,0}(\bar{y})-\underline{U}_{w}\in (BV\cap L^{1})(\mathbb{R}; \mathbb{R}^{2})$, $U_{w,0}(\bar{y})-\underline{U}$ has compact support, and $\inf_{\bar{y}\in \mathbb{R}}\bar{\rho}_0(\bar{y})>0$;\\
\smallskip

$\widetilde{\mathbf{(H2)}}$ The functions $b_{\pm}(\bar{x})$ defined by \eqref{eq:1.31} satisfy $b_{\pm}(\bar{x})\in Lip([0,\ell_{w}]; \mathbb{R})$ and $b'_{\pm}(\bar{x})\in (BV\cap L^{1})([0,\ell_{w}]; \mathbb{R})$ with
the properties that
\begin{equation}\label{eq:1.39}
\pm\ b'_{\pm}(0)\geq 0,\quad\  b_{-}(\bar{x})<0<b_{+}(\bar{x}), \quad\  \mbox{for}\quad  \bar{x}\in [0,\ell_{w}],
\end{equation}
where $Lip$ denotes the set of Lipschitz continuous functions.

Then, the second main result of the paper is stated as follows:
\begin{theorem}\label{thm:1.2}
Assume assumptions $\widetilde{\mathbf{(H1)}}$-$\widetilde{\mathbf{(H2)}}$ hold. For a given fixed hypersonic similarity parameter $a_{\infty}$, let $(\bar{\rho}^{(\tau)}_{w}, \bar{u}^{(\tau)}_{w}, \bar{v}^{(\tau)}_{w})$
be the solution to the initial-boundary value problem \eqref{eq:1.34} and \eqref{eq:1.32}-\eqref{eq:1.33}, and let $(\bar{\rho}_{w}, \bar{u}_{w}, \bar{v}_{w})$ be the solution to the initial-boundary value problem \eqref{eq:1.37} and \eqref{eq:1.35}-\eqref{eq:1.36} which can be obtained by the wave front tracking scheme. Then, there exist small parameters $\epsilon^{**}>0$ and $\tau^{**}>0$ depending only on $\underline{U}_{w}$ and $a_{\infty}$
such that for any $\tau\in(0,\tau^{**})$, if  
\begin{eqnarray}\label{eq:1.40}
T.V.\{U_{w,0}(\cdot); \mathbb{R}\}+|b'_{\pm}(0)|+T.V.\big\{b'_{\pm}(\cdot);[0,\ell_{w}]\big\}<\epsilon,
\end{eqnarray}
for $\epsilon\in (0,\epsilon^{**}_{0})$, then it holds that
\begin{eqnarray}\label{eq:1.41}
\sup_{\ell_{w}<\bar{x}<O(\tau^{-1})}\big\|\big(\bar{\rho}^{(\tau)}_{w}-\bar{\rho}_{w},\ \bar{u}^{(\tau)}_{w}-\bar{u}_{w},\ \bar{v}^{(\tau)}_{w}-\bar{v}_{w}\big)(\bar{x},\cdot)\big\|_{L^{1}(\mathbb{R}^{1})}
\leq C^*_3 \tau^{\frac{3}{2}}.
\end{eqnarray}
Here, the constant $C^*_{3}>0$ depends only on $T.V.\{U_{w,0}(\cdot); \mathbb{R}\}$, $|b'_{\pm}(0)|$, $T.V.\{b'_{\pm}(\cdot);\mathbb{R}_{+}\}$, $\underline{U}_{w}$ and $a_{\infty}$,
but is independent on $\tau$ and $\bar{x}$.
\end{theorem}

To show Theorem \ref{thm:1.1}, we will further develop the approaches given in \cite{chen-christoforou-zhang-1, chen-christoforou-zhang-2, chen-xiang-zhang, zhang-3} for the Cauchy problem, 
to the initial-boundary value problem with a curved boundary. As far as we know, there is no result on the comparison of two entropy solutions involving boundary.
One of the main difficulty is that the solutions of Riemann problem are more complicated near the curved boundary after scaling.
It leads to the standard Riemann semigroup does not exist in general.

To overcome this difficulty, we first approximate boundary $\bar{y}=b(\bar{x})$ by piecewise straight lines $\bar{y}=b_{h}(\bar{x})$, then follow the ideas in \cite{colombo-guerra} to establish a new Lipschitz continuous map $\mathcal{P}_{h}$, 
by piecing together the Riemann solutions near boundary $\bar{y}=b_{h}(\bar{x})$ for a short time
with piecewise constant initial data $U^{\nu}_{0}$.
Next, we compare the Riemann solvers between approximate solutions $U^{(\tau)}_{h,\nu}$ and $U_{h,\nu}$ case by case,
especially 
near the boundary, where one of the boundary conditions is the Neumann type while the other one is the Dirichlet type. Based on it,
we can establish the local $L^1$ difference estimate between the approximate solutions $U^{(\tau)}_{h,\nu}(\bar{x}+s,\cdot)$ and the trajectory $\mathcal{P}_{h}(\bar{x}+s,x)(U^{(\tau)}_{h,\nu}(\bar{x},\cdot))$ for $s>0$ sufficiently small. 
Then, we can further establish the global $L^1$ difference estimate between $U^{(\tau)}_{h,\nu}(\bar{x},\cdot)$
and $\mathcal{P}_{h}(\bar{x} ,0)(U^{\nu}_{0}(\cdot))$ by establishing the following error formula for the map $\mathcal{P}_{h}$ with approximate boundary $\bar{y}=b_{h}(\bar{x})$:
\begin{eqnarray*}
\begin{split}
&\big\|\mathcal{P}_{h}(\bar x,0)(W(0))-W(\bar x)\big\|_{L^{1}((-\infty, b_{h}(\bar x)))}\\[5pt]
&\ \ \ \ \  \leq L\int^{\bar x}_{0}\lim_{s\rightarrow0^{+}}\inf\frac{\big\|\mathcal{P}_{h}(\varsigma+s, \varsigma)(W(\varsigma))-W(\varsigma+s)\big\|_{L^{1}((-\infty, b_{h}(\varsigma+s)))}}{s}d\varsigma,
\end{split}
\end{eqnarray*}
with $W(0)=U^{\nu}_{0}(\cdot)$ and $W(\bar x)=U^{(\tau)}_{h,\nu}(\bar  x,\cdot)$. It will be given in Lemma \ref{lem:5.2}.
Finally, with the uniqueness and compactness of map $\mathcal{P}_{h}$ and approximate solutions $U^{(\tau)}_{h,\nu}$ and $U_{h,\nu}$, we can establish estimate \eqref{eq:1.27} by letting $\nu\rightarrow +\infty$ and $h\rightarrow 0$. Furthermore, by $\eqref{eq:1.10}_3$ and \eqref{eq:1.12}, we can show estimate \eqref{eq:1.28}, which justifies the Van Dyke's similarity theory rigorously with convergence rate in order $\tau^2$, that is consistent with the one predicted by Newtonian-Busemann law (see (3.29) in  \cite[Page 67]{anderson} for more details).

With the help of Theorem \ref{thm:1.1}, to prove Theorem \ref{thm:1.2}, we only need to consider the convergence rate between solutions $U^{(\tau)}_{w}$ and $U_{w}$ for the Cauchy problems $\eqref{eq:1.34}_{1}$ and $\eqref{eq:1.37}_{1}$, respectively, with the same initial data $U_{w,0}$.
Notice that by Theorem 7.1 in \cite{liu}, the total variation of solution $U_{w}$ of the Cauchy problem $\eqref{eq:1.37}_{1}$ with the initial data $U_{w,0}$
 will decay along the flow direction $x$. Then, we employ the semigroup formula with uniformly Lipschitz constant $L^{*}>0$ that independent of $\tau$ 
to compare the trajectory $\mathcal{P}^{(\tau)}_{*}(x-\ell_{w})(U^{\nu}_{w}(\ell_{w},\cdot))$ and the approximate solution $U^{\nu}_{w}(x,\cdot)$ to the Cauchy problem of system $\eqref{eq:1.37}_{1}$ with $x>\ell_{w}$. 
Finally, letting $\nu\rightarrow +\infty$, using decay property of $U_{w}$ for $x>\ell_{w}$, and applying Theorem \ref{thm:1.1} on the $L^1$ difference estimate between $U^{(\tau)}_{w}(x,\cdot)$ and $\mathcal{P}^{(\tau)}_{*}(x-\ell_{w})(U_{w}(x,\cdot))$ for $0<x<\ell_{w}$, we can complete the proof of Theorem \ref{thm:1.2}.

Recently, the authors in \cite{jin-qu-yuan, jin-qu-yuan2,qu-yuan,qu-yuan-zhao} systematically studied the hypersonic limit, which is a different problem from ours since there is no hypersonic similarity structure. The reason is that the wedge angle (or cone angle) $\theta$ in \cite{jin-qu-yuan, jin-qu-yuan2,qu-yuan,qu-yuan-zhao} is fixed such that the similarity parameter $K$ tends to the infinity as $\textrm{M}_{\infty}\rightarrow\infty$.
There are also many literatures on the BV solutions for the steady supersonic compressible Euler flows with free boundaries of small data such that steady supersonic flow past a Lipschitz wedge or moving over a Lipschitz bending wall 
(see \cite{chen-kuang-zhang, chen-li, chen-zhang-zhu, zhang-1, zhang-2} for more details)
which involving the stabilities of the shock wave and rarefaction wave. 

The rest of this paper is organized as follows. In section 2, we present the properties of the elementary wave curves for the system \eqref{eq:1.16} and the system \eqref{eq:1.20}. In section 3, we are devoted to studying and comparing the Riemann problems for the two systems by taking the boundary into account. Based on them, in section 4, we construct the approximate solutions
to the initial-boundary value problem \eqref{eq:1.17}-\eqref{eq:1.19} by the modified wave front tracking scheme and establish some properties. We will give the existence of the Lipschitz continuous map $\mathcal{P}_{h}$ with properties corresponding to the approximate boundary for the initial-boundary value  problem \eqref{eq:1.20}-\eqref{eq:1.21} and \eqref{eq:1.11}.
In section 5, the $L^1$ difference estimates between the approximate solutions $U^{(\tau)}_{h,\nu}$ and the trajectory of $\mathcal{P}_{h}(U^{\nu}_0)$ are derived. 
Then by the properties of
$U^{(\tau)}_{h,\nu}$ and $\mathcal{P}_{h}$ obtained in section 4, we can complete the proof of estimate \eqref{eq:1.27}. Therefore, we can conclude the proofs of Theorem \ref{thm:1.1} and Theorem \ref{thm:1.2} one by one in section 5.
Finally, in the appendix, we establish some results on the interaction of weak waves and the boundary, 
which are used in section 4.

In the rest contents of the paper, we will denote $(\bar{x}, \bar{y})$ and $(\bar{\rho}, \bar{u}, \bar{v})$, $(\bar{\rho}^{(\tau)}, \bar{u}^{(\tau)}, \bar{v}^{(\tau)})$ as $(x,y)$ and $(\rho, u, v)$,
$(\rho^{(\tau)}, u^{(\tau)}, v^{(\tau)})$, respectively for the simplicity of the notations.

\section{Elementary wave curves for the system \eqref{eq:1.16} and the system \eqref{eq:1.20}}\setcounter{equation}{0}
In this section, as preliminaries, we will present some basic structures of the elementary wave curves for both systems \eqref{eq:1.16} and \eqref{eq:1.20},
which will be used in the following sections.

\subsection{Wave curves for equation \eqref{eq:1.16} with $\tau\neq 0$}
First, for $u(\rho, v, \tau^{2})$, we have the following lemma.
\begin{lemma}\label{lem:2.1}
For fixed $\tau>0$, then there holds for $u^{(\tau)}$ that
\begin{eqnarray}\label{eq:2.1}
\begin{split}
 \frac{\partial u^{(\tau)}}{\partial\rho^{(\tau)}}=-\frac{(\rho^{(\tau)})^{\gamma-2}}{a^{2}_{\infty}(1+\tau^2u^{(\tau)})},\ \ \ \
\frac{\partial u^{(\tau)}}{\partial v^{(\tau)}}=-\frac{v^{(\tau)}}{1+\tau^{2}u^{(\tau)}}.
\end{split}
\end{eqnarray}

\end{lemma}

\begin{proof}
First, 
%
taking derivative with respect to $\rho^{(\tau)}$ and $v^{(\tau)}$ on the third equation in \eqref{eq:1.7}, we have
\begin{eqnarray*}
\begin{split}
2\tau^2u^{(\tau)}\frac{\partial u^{(\tau)}}{\partial \rho^{(\tau)}}+2\frac{\partial u^{(\tau)}}{\partial \rho^{(\tau)}}
+\frac{2(\rho^{(\tau)})^{\gamma-1}}{a^{2}_{\infty}}=0,\quad
2\tau^2u^{(\tau)}\frac{\partial u^{(\tau)}}{\partial \rho^{(\tau)}}+2\frac{\partial u^{(\tau)}}{\partial \rho^{(\tau)}}
+2v^{(\tau)}=0.
\end{split}
\end{eqnarray*}
It gives \eqref{eq:2.1}.
\end{proof}

By Lemma \ref{lem:2.1}, the characteristic polynomial for the system \eqref{eq:1.13} is
\begin{eqnarray}\label{eq:2.2}
\begin{split}
\Big(a^{2}_{\infty}(1+\tau^{2}u^{(\tau)})^{2}-\tau^{2}(c^{(\tau)})^{2}\Big)\lambda^{2}
-2a^{2}_{\infty}(1+\tau^{2}u^{(\tau)})v^{(\tau)}\lambda
+a^{2}_{\infty}(v^{(\tau)})^{2}-(c^{(\tau)})^{2}=0,
\end{split}
\end{eqnarray}
which admits two roots (or the eigenvalues)
\begin{eqnarray}\label{eq:2.3}
\begin{split}
&\lambda_{j}(U^{(\tau)}, \tau^{2})\\[5pt]
&=\displaystyle\frac{a^{2}_{\infty}(1+\tau^2u^{(\tau)})v^{(\tau)}+(-1)^{j}c^{(\tau)}
\sqrt{a^{2}_{\infty}(1+\tau^2u^{(\tau)})^2+\tau^2\big(a^{2}_{\infty}(v^{(\tau)})^2-(c^{(\tau)})^2\big)}}
{a^{2}_{\infty}(1+\tau^2 u^{(\tau)})^2-\tau^2(c^{(\tau)})^2},
\end{split}
\end{eqnarray}
with corresponding right eigenvectors
\begin{eqnarray}\label{eq:2.4}
\begin{split}
\tilde{\boldsymbol{r}}_{j}(U^{(\tau)}, \tau^{2})=\Big(\frac{a^{2}_{\infty}\rho^{(\tau)}\lambda_{j}(U^{(\tau)}, \tau^{2})}{c^{(\tau)}}-v^{(\tau)}, \ c^{(\tau)}\Big)^{\top},
\end{split}
\end{eqnarray}
for $j=1,2$, where $c^{(\tau)}=(\rho^{(\tau)})^{\frac{\gamma-1}{2}}$.

\begin{remark}\label{rem:2.1}
When $U^{(\tau)}=\underline{U}$ defined in \eqref{eq-background}, by $\eqref{eq:1.12}$, $u^{(\tau)}=0$. Therefore, we have
\begin{eqnarray}\label{eq:2.5}
\begin{split}
\lambda_{j}(U^{(\tau)},\tau^2)\Big|_{U^{(\tau)}=\underline{U}}=\frac{(-1)^{j}\sqrt{a^{2}_{\infty}-\tau^2}}{a^{2}_{\infty}-\tau^2},
\quad \mbox{for} \ \  j=1,2.
\end{split}
\end{eqnarray}
\end{remark}

\begin{lemma}\label{lem:2.2}
For given $a_{\infty}$, there exist small constants $\epsilon_{0}>0$ and $\tau_{0}>0$ depending only on $\underline{U}$ and $a_{\infty}$, such that if $U^{(\tau)}\in \mathcal{O}_{\epsilon_{0}}(\underline{U})$ and $\tau\in(0,\tau_{0})$, then
\begin{eqnarray}\label{eq:2.6}
\begin{split}
\nabla_{U^{(\tau)}}\lambda_{j}(U^{(\tau)}, \tau^2)\cdot\tilde{\boldsymbol{r}}_{j}(U^{(\tau)}, \tau^{2})>0, \quad \mbox{for}\ \  j=1,2.
\end{split}
\end{eqnarray}
\end{lemma}

\begin{proof}
We only prove the case $j=1$, because the argument for the case $j=2$ is similar. 
Taking derivative on \eqref{eq:2.2} with respect to $\rho^{(\tau)}$ 
and noticing that
\begin{eqnarray*}
\begin{split}
&\big(a^{2}_{\infty}(1+\tau^2u^{(\tau)})^2-\tau^2(c^{(\tau)})^2\big)\lambda_{1}(U^{(\tau)}, \tau^2)-a^{2}_{\infty}(1+\tau^2u^{(\tau)})v^{(\tau)}\\[5pt]
&\qquad\qquad =-c^{(\tau)}\sqrt{a^2_{\infty}(1+\tau^2u^{(\tau)})^2+\tau^2(a^2_{\infty}(v^{(\tau)})^2-(c^{(\tau)})^2)},
\end{split}
\end{eqnarray*}
we have
\begin{eqnarray*}
\begin{split}
&\frac{\partial\lambda_{1}(U^{(\tau)}, \tau^2)}{\partial \rho^{(\tau)}} \\[5pt]
& =\displaystyle-\frac{(\gamma-1)(1+\tau^2u^{(\tau)})+\tau^2\Big((\gamma+1)(1+\tau^2u^{(\tau)})\lambda_{1}(U^{(\tau)}, \tau^2)-2v^{(\tau)}\Big)\lambda_{1}(U^{(\tau)}, \tau^2)}
{2(1+\tau^2u^{(\tau)})\sqrt{a^2_{\infty}(1+\tau^2u^{(\tau)})^2+\tau^2(a^2_{\infty}(v^{(\tau)})^2-(c^{(\tau)})^2)}}\frac{c^{(\tau)}}{\rho^{(\tau)}}.
\end{split}
\end{eqnarray*}

In the same way, we also have
\begin{eqnarray*}
\begin{split}
\frac{\partial\lambda_{1}(U^{(\tau)}, \tau^2)}{\partial v^{(\tau)}} =\displaystyle-\frac{a^2_{\infty}\Big((1+\tau^2u^{(\tau)})\lambda_{1}(U^{(\tau)}, \tau^2)-v^{(\tau)}\Big)\Big(1+\tau^2(u^{(\tau)}+v^{(\tau)}\lambda_{1}(U^{(\tau)}, \tau^2))\Big)}
{2c^{(\tau)}(1+\tau^2u^{(\tau)})\sqrt{a^2_{\infty}(1+\tau^2u^{(\tau)})^2+\tau^2(a^2_{\infty}(v^{(\tau)})^2-(c^{(\tau)})^2)}}.
\end{split}
\end{eqnarray*}

Then, combing the above estimates and by complicated calculations, it follows that
\begin{small}
\begin{eqnarray*}
\begin{split}
\nabla_{U^{(\tau)}}\lambda_{1}(U^{(\tau)}, \tau^2)\cdot \tilde{\boldsymbol{r}}_{1}(U^{(\tau)},\tau^2)=\displaystyle-\frac{(\gamma+1)a^2_{\infty}\Big(1+\tau^2\lambda^2_{1}(U^{(\tau)}, \tau^2)\Big)\Big((1+\tau^2u^{(\tau)})\lambda_{1}(U^{(\tau)}, \tau^2)-v^{(\tau)}\Big)}
{2\sqrt{a^2_{\infty}(1+\tau^2u^{(\tau)})^2+\tau^2(a^2_{\infty}(v^{(\tau)})^2-(c^{(\tau)})^2)}},
\end{split}
\end{eqnarray*}
\end{small}
and at the background state $U^{(\tau)}=\underline{U}$,
\begin{eqnarray*}
\begin{split}
\nabla_{U^{(\tau)}}\lambda_{1}(U^{(\tau)}, \tau^2)\cdot \tilde{\boldsymbol{r}}_{1}(U^{(\tau)},\tau^2)\Big|_{U^{(\tau)}=\underline{U}}=\frac{(\gamma+1)a^{2}_{\infty}}{2(a^{2}_{\infty}-\tau^2)}.
\end{split}
\end{eqnarray*}

Thus, we can choose $\tau_{0}=\frac{1}{2}a_{\infty}$, so that for $\tau\in(0,\tau_0)$, it holds
$$
\nabla_{U^{(\tau)}}\lambda_{1}(U^{(\tau)}, \tau^2)\cdot \tilde{\boldsymbol{r}}_{1}(U^{(\tau)},\tau^2)\Big|_{U^{(\tau)}=\underline{U}}>\frac{2(\gamma+1)}{3}>0.
$$
Hence we can further choose $\epsilon_{0}>0$
depending only on $\underline{U}$ and $a_{\infty}$, such that for $\epsilon\in(0, \epsilon_{0})$, 
$\nabla_{U^{(\tau)}}\lambda_{1}(U^{(\tau)}, \tau^2)\cdot \tilde{\boldsymbol{r}}_{1}(U^{(\tau)},\tau^2)>0$ if $U^{(\tau)}\in \mathcal{O}_{\epsilon_{0}}(\underline{U})$.
%
%
%
\end{proof}

By Lemma \ref{lem:2.2}, we can define
\begin{eqnarray}\label{eq:2.7}
e_{j}(U^{(\tau)},\tau^2)=\displaystyle\frac{1}{\nabla_{U^{(\tau)}}\lambda_{2}(U^{(\tau)},\tau^2)\cdot \tilde{r}_{2}(U^{(\tau)},\tau^2)}, \quad\mbox{for}\ \ j=1, 2.
\end{eqnarray}

Then, we can re-normalized the $\tilde{\boldsymbol{r}}^{(\tau)}_{j}, (j=1,2)$ by setting
\begin{eqnarray*}
\boldsymbol{r}_{j}(U^{(\tau)},\tau^2)=e_{j}(U^{(\tau)},\tau^2)\tilde{\boldsymbol{r}}^{(\tau)}_{j}(U^{(\tau)},\tau^2), \quad\mbox{for}\ \  j=1, 2,
\end{eqnarray*}
so that
\begin{eqnarray*}
\nabla_{U^{(\tau)}}\lambda_{j}(U^{(\tau)},\tau^2)\cdot \boldsymbol{r}_{j}(U^{(\tau)},\tau^2)\equiv1, \quad \mbox{for}\ \ j=1, 2.
\end{eqnarray*}

\begin{remark}\label{rem:2.2}
When $U^{(\tau)}=\underline{U}$, from the proof of the Lemma \ref{lem:2.2}, we can get
\begin{eqnarray}\label{eq:2.9}
\displaystyle e_{1}(U^{(\tau)},\tau^2)\Big|_{U^{(\tau)}=\underline{U}}= e_{2}(U^{(\tau)},\tau^2)\Big|_{U^{(\tau)}=\underline{U}}
=\frac{2(a^{2}_{\infty}-\tau^2)}{(\gamma+1)a^{2}_{\infty}}.
\end{eqnarray}
\end{remark}

It follows from Lemma \ref{lem:2.2} that system \eqref{eq:1.16} is strictly hyperbolic with two genuinely nonlinear characteristic fields.
Then, following the arguments in \cite{chen-wang, smoller}, we have
\begin{lemma}\label{lem:2.3}
Let $a_{\infty}$ be the hypersonic similarity parameter defined by \eqref{eq-hypersonic}. Then, there exist small constants $\tau_{1}\in(0,\tau_{0})$, $\epsilon_{1}\in(0,\epsilon_{0})$ and $\delta_{0}>0$ depending only on $\underline{U}$ and $a_{\infty}$, such that for any constant state
$U^{(\tau)}_{L}\in \mathcal{O}_{\epsilon_1}(\underline{U})$, the $k$th ($k=1,2$) physical admissible wave curve through $U^{(\tau)}_{L}$ can be
parameterized by $\alpha_{k}$ as $\alpha_{k}\mapsto \Phi_{k}(\alpha_{k}; U^{(\tau)}_{L},\tau^2)$ with
$\Phi_{k}\in C^{2}\big((-\delta_{0},\delta_{0})\times \mathcal{O}_{\epsilon_{1}}(\underline{U})\times(0,\tau_0)\big)$ and for $k=1, 2$
\begin{eqnarray}\label{eq:2.10}
\begin{split}
\displaystyle \left. \Phi_{k}(\alpha_{k}; U^{(\tau)}_{L},\tau^2)\right|_{\alpha_k=0}&= U^{(\tau)}_{L},\quad
\left. \frac{\partial\Phi_{k}(\alpha_{k}; U^{(\tau)}_{L},\tau^2)}{\partial\alpha_k}\right|_{\alpha_k=0}&=
r_{k}(U^{(\tau)}_{L},\tau^{2}).
\end{split}
\end{eqnarray}
Moreover, when $\alpha_{k}<0$, $\Phi_{k}(\alpha_{k}; U^{(\tau)}_{L},\tau^2)$ is called the shock wave and is denoted by $\mathcal{S}^{(\tau)}_{k}(U^{(\tau)}_{L})\cap \mathcal{O}_{\epsilon_{1}}(\underline{U})$; when $\alpha_{k}>0$, $\Phi_{k}(\alpha_{k}; U^{(\tau)}_{L},\tau^2)$ is called the rarefaction wave and is denoted by $\mathcal{R}^{(\tau)}_{k}(U^{(\tau)}_{L})\cap \mathcal{O}_{\epsilon_{1}}(\underline{U})$ for $k=1,2$.
\end{lemma}

\subsection{Wave curves for  equations \eqref{eq:1.20}}
The eigenvalues for system \eqref{eq:1.20} are
\begin{eqnarray}\label{eq:2.11}
\lambda_{j}(U)=v+(-1)^{j}\frac{c}{a_{\infty}}, \quad \mbox{for}\quad j=1, 2,
\end{eqnarray}
and the corresponding right eigenvectors are
\begin{eqnarray}\label{eq:2.12}
\boldsymbol{r}_{j}(U)=\frac{2}{(\gamma+1)c}\Big((-1)^{j}a_{\infty}\rho,\  c \Big)^{\top}, \quad \mbox{for}\ \ j=1, 2,
\end{eqnarray}
where $c=\rho^{\frac{\gamma-1}{2}}$.
By direct computation, we obtain that there exists a small parameter $\tilde{\epsilon}_{0}>0$ such that for $U\in O_{\tilde{\epsilon}_{0}}(\underline{U})$ (\emph{i.e.}, $c>0$), it holds
\begin{eqnarray}\label{eq:2.13}
\nabla_{U}\lambda_{j}(U)\cdot \tilde{\boldsymbol{r}}_{j}(U)=\frac{(\gamma+1)c}{2}\equiv1>0 \quad \mbox{for}\ j=1, 2.
\end{eqnarray}
It means that each characteristic fields for system \eqref{eq:1.20} are genuinely nonlinearly.
%
Then, following the ideas in \cite{chen-wang, smoller}, we also have
\begin{lemma}\label{lem:2.4}
There exist small constants $\tilde{\epsilon}_{1}\in(0,\tilde{\epsilon}_{0})$ and $\tilde{\delta}_{0}>0$ depending only on $\underline{U}$, such that for any constant state
$U_{L}\in \mathcal{O}_{\tilde{\epsilon}_1}(\underline{U})$, the $j$th ($j=1, 2$) physical admissible wave curve through $U_{L}$ can be
parameterized by $\tilde{\alpha}_{j}$ as $\tilde{\alpha}_{j}\mapsto \Phi_{j}(\tilde{\alpha}_{j}; U_{L})$ with
$\Phi_{j}\in C^{2}\big((-\tilde{\delta}_{0},\tilde{\delta}_{0})\times \mathcal{O}_{\tilde{\epsilon}_{1}}(\underline{U})\big)$ and
\begin{eqnarray}\label{eq:2.15}
\begin{split}
\displaystyle \left. \Phi_{j}\right|_{\tilde{\alpha}_j=0}&= U_{L},\quad\
\left. \frac{\partial\Phi_{j}}{\partial\tilde{\alpha}_j}\right|_{\tilde{\alpha}_j=0}&=
\boldsymbol{r}_{j}(U_{L}),\quad \mbox{for}\quad   j=1, 2.
\end{split}
\end{eqnarray}
Moreover, when $\tilde{\alpha}_{j}<0$, $\Phi_{j}(\tilde{\alpha}_{j}; U_{L})$ is called the shock wave, and is denoted by $\mathcal{S}_{j}(U_{L})\cap \mathcal{O}_{\tilde{\epsilon}_{1}}(\underline{U})$; when $\tilde{\alpha}_{j}>0$, $\Phi_{j}(\tilde{\alpha}_{j}; U_{L})$ is called the rarefaction wave, and is denoted by $\mathcal{R}_{j}(U_{L})\cap \mathcal{O}_{\tilde{\epsilon}_{1}}(\underline{U})$ for $j=1,2$.
\end{lemma}

\begin{remark}\label{remark:2.3}
For $|\alpha_{k}|<\min\{\delta_0,\tilde{\delta}_{0}\}$ and $U\in \mathcal{O}_{\min\{\epsilon_{1},\tilde{\epsilon}_{1}\}}(\underline{U})$, the elementary wave curves $\Phi_{k}(\alpha_k; U,\tau^2)$ given in Lemma \ref{lem:2.3} and $\Phi_{k}(\alpha_k; U)$ given in Lemma \ref{lem:2.4} satisfy
\begin{eqnarray}\label{eq:2.16}
\Phi_{k}(\alpha_k; U,0)=\Phi_{k}(\alpha_k; U),\quad k=1,2.
\end{eqnarray}
The reason is that for $k=1, 2$, $r_{k}(U,0)=r_{k}(U)$,
\begin{eqnarray*}
\begin{split}
\left. \Phi_{k}(\alpha_{k}; U,\tau^2)\right|_{\alpha_k=0}=\left. \Phi_{k}(\alpha_{k}; U)\right|_{\alpha_k=0}=U,
\end{split}
\end{eqnarray*}
and
\begin{eqnarray*}
\begin{split}
\displaystyle \left. \frac{\partial\Phi_{k}(\alpha_{k}; U,\tau^2)}{\partial\alpha_k}\right|_{\alpha_k=0, \tau=0}=r_{k}(U,0),\quad\
\left. \frac{\partial\Phi_{k}(\alpha_{k}; U)}{\partial\alpha_k}\right|_{\alpha_k=0}=r_{k}(U).
\end{split}
\end{eqnarray*}
So, by the uniqueness of solutions of the ordinary differential equations, we have \eqref{eq:2.16}.
\end{remark}

\section{Comparison of solutions of Riemann problems between $\tau\neq$ and $\tau=0$}
In this section, we will compare the Riemann solutions between problem \eqref{eq:1.16}-\eqref{eq:1.18} and problem \eqref{eq:1.20}-\eqref{eq:1.21} and \eqref{eq:1.11} involving the boundary perturbation.

\subsection{Comparison of solutions to Riemann problems away from the boundary}
We first consider the Riemann problem away from the boundary for system \eqref{eq:1.16} and system \eqref{eq:1.20}, respectively, \emph{i.e.},
\begin{eqnarray}\label{eq:3.1}
\left\{
\begin{array}{llll}
\partial_{x}G(U^{(\tau)}, \tau^2)+ \partial_{y}F(U^{(\tau)},\tau^2)=0, \\[5pt]
\left.U\right|_{x=\hat{x}}=\left\{
\begin{array}{llll}
U_{R}, \qquad  &y>\hat{y},\\[5pt]
U_{L}, \qquad  &y<\hat{y}.
\end{array}
\right.
\end{array}
\right.
\end{eqnarray}
and
\begin{eqnarray}\label{eq:3.1b}
\left\{
\begin{array}{llll}
\partial_{x}U+ \partial_{y}F(U)=0, \\[5pt]
\left.U\right|_{x=\hat{x}}=\left\{
\begin{array}{llll}
U_{R}, \qquad  &y>\hat{y},\\[5pt]
U_{L}, \qquad  &y<\hat{y}.
\end{array}
\right.
\end{array}
\right.
\end{eqnarray}

First, by \cite[Theorem 5.3]{bressan}, we have
\begin{lemma}\label{lem:3.1}
For a given hypersonic similarity parameter $a_{\infty}$ defined in \eqref{eq-hypersonic},
there exist constants $\epsilon'_{1}\in (0, \min\{\epsilon_{1}, \tilde{\epsilon}_{1}\})$ and $\tau'_{1}\in (0, \tau_{1})$ depending only on $\underline{U}$ and $a_{\infty}$, such that for any $U_{L}$, $U_{R}\in \mathcal{O}_{\epsilon'_{1}}(\underline{U})$ and $\tau\in (0, \tau'_{1})$, Riemann problem \eqref{eq:3.1} and Riemann problem \eqref{eq:3.1b} admit unique admissible piecewise constant solutions $U^{(\tau)}_{k}$ and $U_{k}$ for $k=0, 1, 2$, respectively. Both of them are separated by two elementary waves $\boldsymbol{\alpha}=(\alpha_{1}, \alpha_{2})$ and $\boldsymbol{\beta}=(\beta_{1}, \beta_{2})$ satisfying
\begin{eqnarray}\label{eq:3.2}
U^{(\tau)}_{0}=U_{L}, \quad U^{(\tau)}_{2}=U_{R}, \quad U^{(\tau)}_{k}=\Phi_{k}(\alpha_{k};U^{(\tau)}_{k-1},\tau^2),\ \ k=1,2,
\end{eqnarray}
and
\begin{eqnarray}\label{eq:3.3}
U_{0}=U_{L}, \quad U_{2}=U_{R}, \quad U_{k}=\Phi_{i}(\beta_{k};U_{k-1}),\ \ k=1,2,
\end{eqnarray}
where $\Phi^{(\tau)}_{k}$, $\Phi_{k},\ (k=1,2)$ satisfy properties \eqref{eq:2.10} and \eqref{eq:2.15}.
\end{lemma}

In the following, denote
\begin{eqnarray}\label{eq:3.4}
\Phi(\alpha_{1}, \alpha_{2}; U_{L},\tau^2)\doteq\Phi_{2}(\alpha_{2}; \Phi_{1}(\alpha_1;U_{L},\tau^2),\tau^2),\quad
\Phi(\beta_{1}, \beta_{2}; U_{L})\doteq\Phi_{2}(\beta_{2}; \Phi_{1}(\beta_1;U_{L})).
\end{eqnarray}

Then, we have
\begin{proposition}\label{prop:3.1}
Suppose $U_{L}=(\rho_L, v_{L})^{\top}$ and $U_{R}=(\rho_R, v_{R})^{\top}$ are two constant states with $U_{R},\ U_{L}\in \mathcal{O}_{\epsilon'_{1}}(\underline{U})$ for $\epsilon'_{1}>0$ given as in Lemma \ref{lem:3.1}. If
\begin{eqnarray}\label{eq:3.5}
U_{R}=\Phi(\beta_{1}, \beta_{2}; U_{L}),\quad U_{R}=\Phi_{k}(\alpha_{k}; U_{L},\tau^2),\quad k=1\mbox{ or }2,
\end{eqnarray}
then for $\tau\in(0,\tau'_{1})$ with $\tau'_{1}>0$ given as in Lemma \ref{lem:3.1},
\begin{eqnarray}\label{eq:3.7}
\begin{split}
\beta_{j}=\delta_{jk}\alpha_{k}+O(1)|\alpha_{k}|\tau^{2}, \quad \ \ \mbox{for}\quad j, k=1\mbox{ or } 2,
\end{split}
\end{eqnarray}
where $\delta_{jk}$ is the Kronecker symbol and the bounds of $O(1)$ depends only on $\underline{U}$ and $a_{\infty}$.
If
\begin{eqnarray}\label{eq:3.6}
U_{R}=\Phi(\beta_{1}, \beta_{2}; U_{L}), \quad  \alpha=|U_{R}-U_{L}|,
\end{eqnarray}
then for $\tau\in(0,\tau'_{1})$ with $\tau'_{1}>0$ given as in Lemma \ref{lem:3.1},
\begin{eqnarray}\label{eq:3.8}
\begin{split}
\beta_{j}=O(1)\alpha, \quad \mbox{for}\quad j=1, 2.
\end{split}
\end{eqnarray}
\end{proposition}

\begin{proof}
To show \eqref{eq:3.7}, we first need to solve the equation
\begin{eqnarray}\label{eq:3.9}
\Phi(\beta_{1}, \beta_{2}; U_{L})=\Phi_{k}(\alpha_{k}; U_{L},\tau^2).
\end{eqnarray}

Notice that for $U_{L}\in O_{\epsilon^{*}_{1}}(\underline{U})$,
\begin{eqnarray*}
\begin{split}
\det\bigg(\frac{\partial\Phi(\beta_{1}, \beta_{2}; U_{L})}{\partial(\beta_{1}, \beta_{2})}\bigg)\bigg|_{\beta_{1}=\beta_{2}=0}
=\det\big(r_{1}(U_{L}), r_{2}(U_{L})\big)
=-\frac{2\rho_{L}}{\gamma+1}\neq0.
\end{split}
\end{eqnarray*}

So we can apply the implicit function theorem to show that equation \eqref{eq:3.9}
admits a unique solution $\beta_{j}=\beta_{j}(\alpha_{k},\tau^2; U_{L})\in C^{2},\ j=1,2$ for fixed $k$.

Obvious, $\beta_{j}(0,0; U_{L})=\beta_{j}(0,\tau^2; U_{L})=0$.
On the other hand, notice that
\begin{eqnarray*}
\Phi_{k}(\alpha_{k};U_{L},\tau^2)\Big|_{\tau=0}=\Phi_{k}(\alpha_{k};U_{L}).
\end{eqnarray*}

So
\begin{eqnarray*}
\beta_{j}(\alpha_{k},0;U_{L})=\delta_{jk}\alpha_{k}.
\end{eqnarray*}

Then, by Taylor expansion formula and together with the above estimates, we have
\begin{eqnarray*}
\begin{split}
\beta_{j}(\alpha_{k},\tau^{2};U_{L})&=\beta_{j}(\alpha_{k},0;U_{L})+\beta_{j}(0,\tau^{2};U_{L})-\beta_{j}(0,0;U_{L})+O(1)|\alpha_{k}|\tau^2\\[5pt]
&=\delta_{jk}\alpha_{k}+O(1)|\alpha_{k}|\tau^2,
\end{split}
\end{eqnarray*}
which gives estimate \eqref{eq:3.7}.

For $U_{R},\ U_{L}\in \mathcal{O}_{\epsilon'_{1}}(\underline{U})$, there exists a constant $C>0$ depending only on $\underline{U}$
such that
\begin{eqnarray*}
C^{-1}\big(|\beta_{1}|+|\beta_{2}|\big)<\big|\Phi(\beta_{1}, \beta_{2}; U_{L})-U_{L}\big|<C\big(|\beta_{1}|+|\beta_{2}\big).
\end{eqnarray*}
So if $\alpha$ is given as in \eqref{eq:3.6}, then we have estimate \eqref{eq:3.8}.
\end{proof}

\begin{proposition}\label{prop:3.2}
Suppose $U_{L}=(\rho_L, v_{L})^{\top}$, $\hat{U}_{R}=(\hat{\rho}_R, \hat{v}_{R})^{\top}$ and $U_{R}=(\rho_R, v_{R})^{\top}$ are three constant states with $U_{R}$, $\hat{U}_{R}$, $U_{L}\in \mathcal{O}_{\epsilon'_{1}}(\underline{U})$ for $\epsilon'_{1}>0$ given as in Lemma \ref{lem:3.1}. If 
\begin{eqnarray}\label{eq:3.10}
U_{R}=\Phi(\beta_{1}, \beta_{2}; U_{L}),\quad U_{R}=\Phi(\alpha_{1}, \alpha_{2}; U_{L},\tau^2),
\end{eqnarray}
then for $\tau\in(0,\tau'_{1})$ with $\tau'_{1}>0$ given as in Lemma \ref{lem:3.1}, it holds that
\begin{eqnarray}\label{eq:3.12}
\begin{split}
\beta_{j}=\alpha_{j}+O(1)\big(|\alpha_{1}|+|\alpha_{2}|\big)\tau^{2}, \quad \ \ j=1,2,
\end{split}
\end{eqnarray}
where $O(1)$ depends only on $\underline{U}$ and $a_{\infty}$.
If
\begin{eqnarray}\label{eq:3.11}
U_{R}=\Phi(\beta_{1}, \beta_{2}; U_{L}),\quad \hat{U}_{R}=\Phi(\alpha_{1}, \alpha_{2}; U_{L},\tau^2), \quad  \alpha=|U_{R}-\hat{U}_{R}|,
\end{eqnarray}
then for $\tau\in(0,\tau'_{1})$ with $\tau'_{1}>0$ given as in Lemma \ref{lem:3.1},
\begin{eqnarray}\label{eq:3.13}
\begin{split}
\beta_{j}=\alpha_{j}+O(1)\big(|\alpha_{1}|+|\alpha_{2}|\big)\tau^{2}+O(1)\alpha, \quad \ \ j=1,2.
\end{split}
\end{eqnarray}
\end{proposition}

\begin{proof}
First, let us consider the case that $U_{L}$ and $U_{R}$ satisfy \eqref{eq:3.10}. From \eqref{eq:3.10}, we have
\begin{eqnarray*}
\Phi(\beta_{1}, \beta_{2}; U_{L})=\Phi(\alpha_{1}, \alpha_{2}; U_{L},\tau^2).
\end{eqnarray*}
As the argument in the proof of Proposition \ref{prop:3.1}, by applying the implicit function theorem, there exits a
$C^2$ solution $\beta_{j}=\beta_{j}(\alpha_1, \alpha_2,\tau^2; U_{L})$ for $j=1, 2$. To estimate $\beta_{j}$, set
\begin{eqnarray*}
U_{M}=\Phi_{1}(\alpha_1; U_{L},\tau^2),\quad U_{R}=\Phi_{2}(\alpha_2; U_{M},\tau^2),
\end{eqnarray*}
and choose $\beta'_{j}$ and $\beta''_{j}$ for $j=1,2$ such that
\begin{eqnarray*}
U_{M}=\Phi(\beta'_1, \beta'_{2}; U_{L}), \quad\ U_{R}=\Phi(\beta''_1, \beta''_{2}; U_{M}).
\end{eqnarray*}

Then, by Proposition \ref{prop:3.1}, we have
\begin{eqnarray}\label{eq:3.14}
\beta'_{j}=\delta_{j1}\alpha_{1}+O(1)|\alpha_1|\tau^2, \quad \beta''_{j}=\delta_{j2}\alpha_{2}+O(1)|\alpha_1|\tau^2, \quad j=1,2.
\end{eqnarray}

On the other hand, from the local interaction estimate as stated in \cite[Theorem 19.2]{smoller},
\begin{eqnarray}\label{eq:3.15}
\beta_{j}=\beta'_{j}+\beta''_{j}+O(1)\Delta(\boldsymbol{\beta}',\boldsymbol{\beta}''),
\end{eqnarray}
where $\Delta(\boldsymbol{\beta}',\boldsymbol{\beta}'')=|\beta'_{2}||\beta''_{1}|+\sum_{j=1,2}\Delta_{j}(\boldsymbol{\beta}',\boldsymbol{\beta}'')$
with
\begin{equation*}
\left. \Delta_{j}\right(\boldsymbol{\beta}',\boldsymbol{\beta}'')=\left\{
\begin{array}{llll}
|\beta'_{j}||\beta''_{j}|, \qquad  &\beta'_{j}<0\ \ \mbox{or}\ \ \beta''_{j}<0 ,\\[5pt]
0, \qquad  & \beta'_{j}>0\ \ \mbox{and}\ \ \beta''_{j}>0.
\end{array}
\right.
\end{equation*}

Combining \eqref{eq:3.14} and \eqref{eq:3.15}, we have \eqref{eq:3.12}.

Next, we assume \eqref{eq:3.11} holds. To derive \eqref{eq:3.13}, we choose
small parameters $\hat{\beta}_{1},\hat{\beta}_{2}$ so that $\hat{U}_{R}=\Phi(\hat{\beta}_{1},\hat{\beta}_{2}; U_{L})$.
Then, as done in the argument in the proof of Proposition \ref{prop:3.1}, it follows from equations $\Phi(\hat{\beta}_{1},\hat{\beta}_{2}; U_{L})=\Phi^{(\tau)}(\alpha_{1},\alpha_{2}; U_{L})$ that
\begin{eqnarray}\label{eq:3.16}
\hat{\beta}_{k}=\alpha_{k}+O(1)(|\alpha_1|+|\alpha_2|)\tau^2, \quad k=1,2.
\end{eqnarray}

On the other hand, following the argument in the proof of estimate \eqref{eq:3.8} in Proposition \ref{prop:3.1}, we obtain
\begin{eqnarray*}
\beta_{k}-\hat{\beta}_{k}=O(1)\alpha, \quad k=1,2,
\end{eqnarray*}
which implies \eqref{eq:3.12} by plugging \eqref{eq:3.16} into it.
\end{proof}

\subsection{Comparison of solutions of Riemann problems near the boundary}
Let $A_{k}=(x_k, b_k),(k=1,2, 3)$ be the corner points on the boundary with $x_{1}<x_{2}<x_3$ (see Fig. \ref{fig3.1}).
Denote
\begin{eqnarray*}
\theta_{1}=\arctan\Big(\frac{b_{2}-b_{1}}{x_{2}-x_1}\Big),\ \  \theta_{2}=\arctan\Big(\frac{b_{3}-b_{2}}{x_{3}-x_{2}}\Big),  \ \  \omega=\theta_{2}-\theta_{1},
\end{eqnarray*}
where $\omega$ represents the change of angles $\theta_1$, $\theta_{2}$ at the turning points $A_2$.
\par For $k=1,2$, we define
\begin{eqnarray*}
\begin{split}
&\Omega_{k}=\{(x,y):\ x_{k}\leq x< x_{k+1}, \ y<b_k+(x-x_{k})\tan(\theta_{k})\},\\[5pt]
&\Gamma_{k}=\{(x,y):\ x_{k}\leq x< x_{k+1}, \ y=b_k+(x-x_{k})\tan(\theta_{k})\}.
\end{split}
\end{eqnarray*}

Let $\textbf{n}_{k}$ be the outer unit normal vector to $\Gamma_{k}$ for $k=1,2$, as
\begin{equation*}
\textbf{n}_{k}=(\sin\theta_{k},-\cos\theta_{k}), \quad \mbox{for}\quad k=1,2.
\end{equation*}

\begin{figure}[ht]
\begin{center}
\begin{tikzpicture}[scale=1.0]
\draw [thick](-5.0,1.5)--(-3,1)--(-1,1.8);

\draw [line width=0.05cm](-5.0,1.5)--(-5.0,-1.5);
\draw [line width=0.05cm](-3,1)--(-3,-1.5);
\draw [line width=0.05cm](-1.0,1.8)--(-1.0,-1.5);

\draw [thin](-4.8,1.45)--(-4.5, 1.75);
\draw [thin](-4.5,1.38)--(-4.2, 1.68);
\draw [thin](-4.2,1.30)--(-3.9, 1.60);
\draw [thin](-3.9, 1.23)--(-3.6,1.53);
\draw [thin](-3.6, 1.16)--(-3.3,1.46);
\draw [thin](-3.3, 1.08)--(-3.0,1.38);
\draw [thin](-3.0, 1.0)--(-2.8,1.4);
\draw [thin](-2.7,1.12)--(-2.5, 1.52);
\draw [thin](-2.4,1.23)--(-2.2, 1.63);
\draw [thin](-2.1,1.35)--(-1.9, 1.75);
\draw [thin](-1.8,1.47)--(-1.6, 1.87);
\draw [thin](-1.5, 1.59)--(-1.3,1.99);
\draw [thin](-1.2,1.71)--(-1.0, 2.11);

\draw [thick][blue](-3,1)--(-1.8,0.6);
\draw [thick][blue](-3,1)--(-1.8,0.4);
\draw [thick](-3,1)--(-2.0,0.1);

\node at (-5.3, 1.6) {$A_{1}$};
\node at (-3.0, 1.5) {$A_{2}$};
\node at (-0.6, 1.8) {$A_{3}$};
\node at (-1.5, 0.5) {$\beta_{1}$};
\node at (-1.6, -0.1) {$\alpha_{1}$};
\node at (-2.0, 1.0) {$U_{R}$};
\node at (-2.4, -0.2) {$U_{L}$};
\node at (1, 2) {$$};

\node at (-4.0, -0.8) {$\Omega_{1}$};
\node at (-2.0, -0.8) {$\Omega_{2}$};

\node at (-5.0, -1.9) {$x_{1}$};
\node at (-3.0, -1.9) {$x_{2}$};
\node at (-0.9, -1.9) {$x_{3}$};
\end{tikzpicture}
\caption{Mixed Riemann problems}\label{fig3.1}
\end{center}
\end{figure}
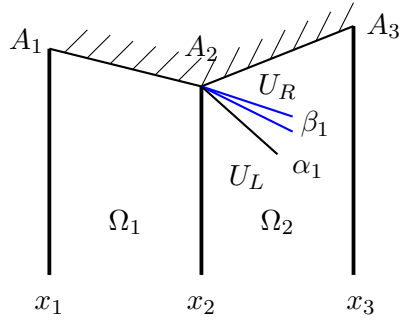

Let us consider the following two mixed Riemann problems.
\begin{eqnarray}\label{eq:3.17}
\left\{
\begin{array}{llll}
\partial_{x}G(U^{(\tau)}, \tau^2)+ \partial_{y}F(U^{(\tau)},\tau^2)=0, \quad & in\ \ \ \Omega_{2},\\[5pt]
U^{(\tau)}=U^{(\tau)}_L, \quad & on\ \ \ \{x=x_{2}\}\cap\Omega_{2},\\[5pt]
\big((1+\tau^2u^{(\tau)}),v^{(\tau)}\big)\cdot\textbf{n}_{2}=0, \quad & on\ \ \ \Gamma_{2},
\end{array}
\right.
\end{eqnarray}
and
\begin{eqnarray}\label{eq:3.18}
\left\{
\begin{array}{llll}
\partial_{x}U+ \partial_{y}F(U)=0, \quad & in\ \ \ \Omega_{2},\\[5pt]
U=U_L, \quad & on\ \ \ \{x=x_{2}\}\cap\Omega_{2},\\[5pt]
v=\tan(\theta_{2}), \quad & on\ \ \ \Gamma_{2},
\end{array}
\right.
\end{eqnarray}
where $U^{(\tau)}_L=(\rho^{(\tau)}_{L}, v^{(\tau)}_{L})^{\top}$ and $U_L=(\rho_{L}, v_{L})^{\top}$ are constant states.

We have the following lemma on the solvability of problem \eqref{eq:3.17} and problem \eqref{eq:3.18}.
\begin{lemma}\label{lem:3.2}
For a given hypersonic similarity parameter $a_{\infty}$, there exist small parameters $\epsilon'_{2}\in (0, \min\{\epsilon_{1}, \tilde{\epsilon}_{1}\})$ and $\tau'_{2}\in(0,\tau_1)$
depending only on $\underline{U}$, $a_{\infty}$, such that if $\big((1+\tau^2 u^{(\tau)}_{L}),v^{(\tau)}_{L}\big)\cdot \mathbf{n}_{1}=0$
and $v_{L}=\tan(\theta_1)$ for $U^{(\tau)}_{L}, U_{L}\in \mathcal{O}_{\epsilon'_{2}}(\underline{U})$,
and $|\omega|+|\theta_{1}|<\epsilon'_{2}$, then for $\tau\in(0,\tau'_{2})$, the mixed Riemann problem \eqref{eq:3.17} (or
\eqref{eq:3.18}) admits a unique solution $U^{(\tau)}$ (or $U$), which is connected to $U^{(\tau)}_{L}$ (or  $U_L$) by 1th waves
 $\alpha_1$ (or $\beta_{1}$, respectively). Moreover, it holds that
\begin{eqnarray}\label{eq:3.19}
\alpha_1=K^{(\tau)}_{c}\omega\qquad\mbox{and}\qquad  \beta_1=K_{c}\omega,
\end{eqnarray}
where 
$K^{(\tau)}_{c}>0$ and $K_{c}>0$ depend only on $\underline{U}$, but not on $\tau$.
\end{lemma}

\begin{proof}
We only consider problem \eqref{eq:3.17} since problem \eqref{eq:3.18} can be proved in the same way.
To show the existence of the mixed Riemann problem \eqref{eq:3.17}, we denote
\begin{eqnarray}\label{eq:3.20}
\rho^{(\tau)}=\Phi^{(1)}_{1}(\alpha_1; U_{L},\tau^2), \ \ v^{(\tau)}=\Phi^{(2)}_{1}(\alpha_1; U_{L},\tau^2).
\end{eqnarray}
Obviously, $u^{(\tau)}$ satisfies
\begin{eqnarray}\label{eq:3.22}
\begin{split}
\tau^2(u^{(\tau)})^2+2u^{(\tau)}+\Big(\Phi^{(2)}_{1}(\alpha_1; U^{(\tau)}_{L},\tau^2)\Big)^2+\frac{2\Big(\big(\Phi^{(1)}_{1}(\alpha_1; U^{(\tau)}_{L},\tau^2)\big)^2-1\Big)}{(\gamma-1)a^2_{\infty}}=0.
\end{split}
\end{eqnarray}
Let
\begin{eqnarray}\label{eq:3.21}
\mathcal{L}(\alpha_{1}, \omega, \theta_1, \tau^2; U^{(\tau)}_{L} )
\doteq(1+\tau^2 u^{(\tau)})\sin(\theta_1+\omega)-\Phi^{(2)}_{1}(\alpha_1; U^{(\tau)}_{L},\tau^2)\cos(\theta_1+\omega),
\end{eqnarray}
then $\mathcal{L}(0, 0, 0, \tau^2; \underline{U})\equiv0$. Moreover, by Lemma \ref{lem:2.1}, we get
\begin{eqnarray*}
\begin{split}
&\frac{\partial\mathcal{L}(\alpha_{1}, \omega, \theta_1, \tau^2; U^{(\tau)}_{L} )}{\partial \alpha_1}\\[5pt]
=&\bigg(\frac{(\rho^{(\tau)})^{\gamma-2}}{a^2_{\infty}(1+\tau^2u^{(\tau)})}\frac{\partial\Phi^{(1)}_{1}(\alpha_1; U_{L},\tau^2)}{\partial\alpha_1}
-\frac{v^{(\tau)}}{1+\tau^2u^{(\tau)}}\frac{\partial\Phi^{(2)}_{1}(\alpha_1; U_{L},\tau^2)}{\partial\alpha_1}\bigg)\cdot \sin(\theta_1+\omega)\\[5pt]
&\ \ \ -\frac{\partial\Phi^{(2)}_{1}(\alpha_1; U_{L},\tau^2)}{\partial\alpha_1}\cos(\theta_1+\omega).
\end{split}
\end{eqnarray*}

Then, by Remark \ref{rem:2.2}, 
choosing a small constant $\tilde\tau_2>0$ depending only on $a_{\infty}$, 
for $\tau\in (0,\tilde\tau_2)$, we have
\begin{eqnarray*}
\begin{split}
\frac{\partial\mathcal{L}(\alpha_{1}, \omega, \theta_1, \tau^2; U^{(\tau)}_{L} )}{\partial \alpha_1}\bigg|_{\alpha_1=\theta_1=\omega=0,U_{L}=\underline{U}}=-e_{1}(U^{(\tau)},\tau^2)\Big|_{U^{(\tau)}=\underline{U}}<-\frac{3}{\gamma+1}.
\end{split}
\end{eqnarray*}

Therefore, by applying the implicit function theorem, we can get a unique $C^2$ function $\alpha_1=\alpha_{1}(\omega,\theta_1, \tau^{2}; U_{L})$
near $\alpha_1=\theta_1=\omega=0$ and $U^{(\tau)}_{L}=\underline{U}$ for $\tau$ sufficiently small, such that
$$
\mathcal{L}(\alpha_{1}, \omega, \theta_1, \tau^2; U^{(\tau)}_{L} )=0.
$$

Finally, since $\big((1+\tau^2 u^{(\tau)}_{L}),v^{(\tau)}_{L}\big)\cdot \textbf{n}_{1}=0$, $\alpha_{1}(0, \theta_1, \tau^{2}; U_{L})=0$. So estimate \eqref{eq:3.19} is from the Taylor expansion directly by taking $K^{(\tau)}_{c}=\frac{\partial\alpha_1}{\partial \omega}$.
Moreover, 
notice that $\frac{\partial\mathcal{L}(\alpha_{1}, \omega, \theta_1, \tau^2; U^{(\tau)}_{L} )}{\partial \omega}\bigg|_{\alpha_1=\theta_1=\omega=0,U^{(\tau)}_{L}=\underline{U}}=1$, hence
\begin{eqnarray*}
\begin{split}
\frac{\partial\alpha_1}{\partial \omega}\bigg|_{\alpha_1=\theta_1=\omega=0,U^{(\tau)}_{L}=\underline{U}}
&=\displaystyle-\frac{\frac{\partial\mathcal{L}(\alpha_{1}, \omega, \theta_1, \tau^2; U^{(\tau)}_{L} )}{\partial \omega}\bigg|_{\alpha_1=\theta_1=\omega=0,U^{(\tau)}_{L}=\underline{U}}}{\frac{\partial\mathcal{L}(\alpha_{1}, \omega, \theta_1, \tau^2; U^{(\tau)}_{L} )}{\partial \alpha_1}\bigg|_{\alpha_1=\theta_1=\omega=0,U^{(\tau)}_{L}=\underline{U}}}\\[5pt]
&=\frac{1}{e_{1}(U^{(\tau)},\tau^2)\Big|_{U^{(\tau)}}}>\frac{\gamma+1}{3}, \qquad \mbox{for}\quad \tau\in (0,\tau'_2).
\end{split}
\end{eqnarray*}

Thus, choose $\epsilon'_{2}\in (0, \min\{\epsilon_{1}, \tilde{\epsilon}_{1}\})$ small, then $K^{(\tau)}_{c}$
is bounded independently on $\tau$.
\end{proof}

Now, we will compare the Riemann solvers near the boundary and away from the corner points.
\begin{proposition}\label{prop:3.3}
Let $U_{b}=(\rho_b, v_{b})^{\top}$ and $U_{L}=(\rho_L, v_{L})^{\top}$ be the constant states satisfying
\begin{eqnarray*}
\begin{split}
U_{b}=\Phi_1(\alpha_1; U_{L},\tau^2), \quad U_{b}=\Phi(\beta_1,\beta_2; U_{L}),
\end{split}
\end{eqnarray*}
and
\begin{eqnarray*}
\begin{split}
\big((1+\tau^2 u_{L}), v_{L}\big)\cdot \mathbf{n}_{1}=0, \quad \big((1+\tau^2 u_{b}), v_{b}\big)\cdot \mathbf{n}_{2}=0,
\end{split}
\end{eqnarray*}
with $U_{b},U_{L}\in \mathcal{O}_{\epsilon'_{2}}(\underline{U})$ and $|\theta_1|+|\omega|<\epsilon'_{2}$, where $u_{L}=u_{L}(\rho_{L},v_{L},\tau^2)$ and $u_{b}=u_{b}(\rho_{b},v_{b},\tau^2)$ are given by \eqref{eq:1.12}, and $\mathbf{n}_k=(\sin(\theta_k),-\cos(\theta_k)),\,(k=1,2)$.
Then, for $\tau\in(0,\tau'_{2})$ with $\tau'_{2}>0$ being given as in Lemma \ref{lem:3.2}, it holds that
\begin{eqnarray}\label{eq:3.23}
\beta_{1}=\alpha_1+O(1)|\alpha_1|\tau^{2},\quad \beta_{2}=O(1)|\alpha_1|\tau^{2}.
\end{eqnarray}
Moreover,
\begin{eqnarray}\label{eq:3.24}
\beta_{1}=O(1)(1+|\omega|)\tau^{2}, \quad \beta_{2}=O(1)|\omega|\tau^{2},
\end{eqnarray}
where
$\omega=\theta_{2}-\theta_{1}$, and $O(1)$ depends only on
$\underline{U}$ and $a_{\infty}$ but not on $\epsilon'_2$ and $\tau$.
\end{proposition}

\begin{proof}
By the assumptions, we know that $\beta_1$ and $\beta_2$ satisfy the following equation
\begin{eqnarray*}
\Phi(\beta_{1},\beta_2; U_{L})=\Phi(\alpha_{1}; U_{L}, \tau^2).
\end{eqnarray*}

Then, by the implicit function theorem and following the argument in the proof of Proposition \ref{prop:3.1}, we know that the above equation admits a unique $C^2$ solution $\beta_{k}=\beta_k(\alpha_{1}, \tau^{2}, U_{L})$ for $U_{L}\in \mathcal{O}_{\epsilon'_{2}}(\underline{U})$ and $\tau\in(0,\tau'_{2})$ for $k=1,2$, which satisfies \eqref{eq:3.23}. Moreover, by Lemma \ref{lem:3.2}, we further have \eqref{eq:3.24}.
\end{proof}

Next, we will compare the Riemann solutions of problem \eqref{eq:3.17} and problem \eqref{eq:3.18} with corner points.
\begin{proposition}\label{prop:3.4}
Let $U^{(\tau)}_{b}=(\rho^{(\tau)}_b, v^{(\tau)}_{b})^{\top}$, $U_{b}=(\rho_b, v_{b})^{\top}$ and $U_{L}=(\rho_L, v_{L})^{\top}$
be the three constant states satisfying
\begin{eqnarray*}
\begin{split}
U^{(\tau)}_{b}=\Phi(\alpha_1; U_{L},\tau^2), \quad U_{b}=\Phi(\beta_1; U_{L}),
\end{split}
\end{eqnarray*}
and
\begin{eqnarray*}
\begin{split}
\big((1+\tau^2 u_{L}), v_{L}\big)\cdot \mathbf{n}_{1}=0, \quad \big((1+\tau^2 u^{(\tau)}_{b}), v^{(\tau)}_{b}\big)\cdot \mathbf{n}_{2}=0,\quad
v_{b}=\tan(\theta_{2}),
\end{split}
\end{eqnarray*}
with $U^{(\tau)}_{b}, U_{b},U_{L}\in \mathcal{O}_{\epsilon'_{2}}(\underline{U})$ and $|\theta_1|+|\omega|<\epsilon''_{2}$, where $u_{L}=u_{L}(\rho_{L},v_{L},\tau^2)$ and $u^{(\tau)}_{b}=u_{b}(\rho_{b},v_{b},\tau^2)$ are given by \eqref{eq:1.12}, and $\mathbf{n}_k=(\sin(\theta_k),-\cos(\theta_k)),\,(k=1,2)$.
Then, for $\tau\in(0,\tau''_{1})$ with $\tau''_{1}>0$ being given in Lemma \ref{lem:3.2}, it holds that
\begin{eqnarray}\label{eq:3.25}
\beta_{1}=\alpha_1+O(1)(1+|\alpha_1|)\tau^{2}.
\end{eqnarray}
Moreover,
\begin{eqnarray}\label{eq:3.26}
\beta_{1}=O(1)(|\omega|+\tau^{2}),
\end{eqnarray}
where $\omega=\theta_{2}-\theta_{1}$, and $O(1)$ depends only
$\underline{U}$ and $a_{\infty}$ but not on $\epsilon'_2$ and $\tau$.
\end{proposition}

\begin{proof}
Denote
\begin{eqnarray*}
\rho^{(\tau)}_{b}=\Phi^{(1)}_{1}(\alpha_1; U_{L},\tau^2),\quad v^{(\tau)}_{b}=\Phi^{(2)}_{1}(\alpha_1; U_{L},\tau^2),\quad v_{b}=\Phi^{(2)}_{1}(\beta_1; U_L).
\end{eqnarray*}

By assumptions, we have the relation
\begin{eqnarray}\label{eq:3.27}
\Big(1+\tau^{2}u^{(\tau)}_{b}\big(\Phi_{1}(\alpha_1; U_{L},\tau^2)\big)\Big)\cdot\Phi^{(2)}_{1}(\beta_{1}; U_{L})=\Phi^{(2)}_{1}(\alpha_{1}; U_{L}, \tau^2).
\end{eqnarray}

Since $\displaystyle \frac{\partial\Phi^{(2)}_{1}(\beta_{1}; U_{L})}{\partial \beta_1}\Big|_{\beta_1=0, U_{L}=\underline{U}}=\frac{2}{\gamma+1}>0$, it follows from the implicit function theorem that equation \eqref{eq:3.27} has a unique $C^2$ solution $\beta_{1}=\beta_1(\alpha_{1}, \tau^{2}, U_{L})$ for $U_{L}\in \mathcal{O}_{\epsilon'_{2}}(\underline{U})$ and $\tau\in(0,\tau'_{2})$. To estimate $\beta_1$, notice that by \eqref{eq:3.27} again, when $\alpha_1=0$, $\beta_1$ satisfies that
\begin{eqnarray*}
\big(1+\tau^{2}u_L\big)\Phi^{(2)}_{1}(\beta_{1}; U_{L})=v_{L}.
\end{eqnarray*}
Then $\beta_{1}(0,\tau^{2}, U_{L})=O(1)\tau^2$. On the other hand, for $\tau=0$, we have
\begin{eqnarray*}
\Phi^{(2)}_{1}(\beta_{1}; U_{L})=\Phi^{(2)}_{1}(\alpha_{1}; U_{L}, 0).
\end{eqnarray*}
So, $\beta_{1}(\alpha_1,0, U_{L})=\alpha_1$. With these two facts, we obtain from the Taylor expansion that
\begin{eqnarray*}
\begin{split}
\beta_1(\alpha_1, \tau^{2}, U_{L})&=\beta_1(\alpha_1, 0, U_{L})+\beta_1(0, \tau^{2}, U_{L})-\beta_1(0, 0, U_{L})+O(1)|\alpha_1|\tau^2\\[5pt]
&=\alpha_1+O(1)(1+|\alpha_1|)\tau^{2}.
\end{split}
\end{eqnarray*}

Finally, we can further deduce estimate \eqref{eq:3.26} by applying Lemma \ref{lem:3.2}.
\end{proof}

\section{Wave front tracking scheme and existence of weak solutions}
In this section, we will establish the global existence and $L^1$ stability of weak solutions of
problem \eqref{eq:1.16}-\eqref{eq:1.18} and problem \eqref{eq:1.20}-\eqref{eq:1.21} and \eqref{eq:1.11}.
First, let us introduce the $(h, \nu)$-approximate solutions $U^{(\tau)}_{h,\nu}$ to the initial-boundary value
problem \eqref{eq:1.16}-\eqref{eq:1.18} via the wave-front tracking scheme.

Choose a mesh length $h=\Delta x>0$ in the $x$-direction and let $A_{k}=(x_k, b_k)\doteq(kh, b(kh)), (k\geq 0)$
be the points on the boundary $y=b(x)$ and denote
\begin{eqnarray}\label{eq:4.1}
\begin{split}
&\theta_{k}=\arctan\bigg(\frac{b_{k+1}-b_{k}}{h}\bigg),\quad \omega_{0}=\arctan\bigg(\frac{b_{1}-b_{0}}{h}\bigg), \quad \omega_{k}=\theta_{k+1}-\theta_{k},\quad k\ge 1,
\end{split}
\end{eqnarray}
where $\omega_k$ represents the change of angle $\theta_{k}$ at the corner points $A_k$ for $k\ge 0$.

Since $T.V.\{b'(\cdot);\mathbb{R}_{+}\}<+\infty$ by assumption $\mathbf{(H2)}$, there exists a $b'_{\infty}$ such that $\lim_{x\rightarrow +\infty}b'(x)=b'_{\infty}$. So there exists $k_*\in\mathbb{N}_{+}$ such that $\|b'_{\infty}-b'(\cdot)\|_{L^{\infty}(\{x\geq k_*h\})}\leq h$.
Then, define 
\begin{eqnarray}\label{eq:4.2}
\begin{split}
b_{h}(x)=b_k+(x-x_k)\tan(\theta_{k}) \qquad \mbox{for any $x\in[x_{k}, x_{k+1}),\ k=0,\cdots,k_*-1$},
\end{split}
\end{eqnarray}
and
\begin{equation}\label{eq:4.2x}
b_h(x)=b_{k_*}+(x-x_{k_*})b'_{\infty} \qquad \mbox{for any $x\in[x_{k_{*}},+\infty)$}.
\end{equation}

Obviously,
\begin{eqnarray}\label{eq:4.3}
\begin{split}
\|b'_{h}-b'\|_{L^{1}(\mathbb{R}_{+})}\leq h,\quad \lim_{x\rightarrow +\infty}(b'_h-b')=0,
\quad\mbox{and}\quad T.V.\{b'_{h};\,\mathbb{R}_{+}\}\leq T.V.\{b';\,\mathbb{R}_{+}\}.
\end{split}
\end{eqnarray}

The corresponding approximate domain for $\Omega$ is defined by
\begin{eqnarray*}
\begin{split}
&\Omega_{h}=\bigcup_{k\geq0}\Omega_{h, k}, \quad\ \Omega_{h, k}=\{(x,y):  x_{k}\leq x< x_{k+1}, \ y<b_{h}(x)\},
\end{split}
\end{eqnarray*}
with its boundary being defined as
\begin{eqnarray*}\label{eq:2.19}
\begin{split}
\Gamma_{h}=\bigcup_{k\geq0}\Gamma_{h, k},\quad\ \Gamma_{h, k}=\{(x,y): x_{k}\leq x< x_{k+1}, \ y=b_{h}(x)\}.
\end{split}
\end{eqnarray*}

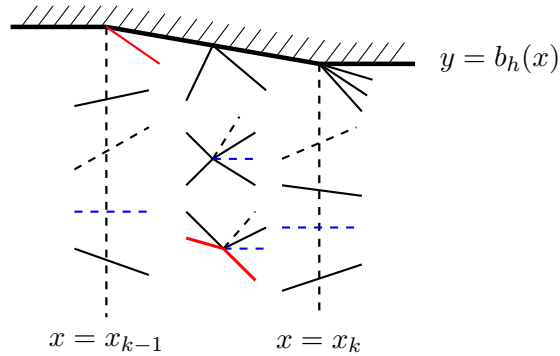
\begin{figure}[ht]
\begin{center}
\begin{tikzpicture}[scale=0.7]
\draw [line width=0.06cm](-5.8,2)--(-4,2);
\draw [line width=0.06cm](-4,2)--(0,1.3);
\draw [line width=0.06cm](0,1.3)--(1.8,1.3);
\draw [line width=0.03cm][dashed](-4,2)--(-4,-3.5);
\draw [line width=0.03cm][dashed](0,1.3)--(0,-3.5);

\draw [thin](-5.6,2)--(-5.3,2.4);
\draw [thin](-5.3,2)--(-5.0,2.4);
\draw [thin](-5,2)--(-4.7,2.4);
\draw [thin](-4.7,2)--(-4.4,2.4);
\draw [thin](-4.4,2)--(-4.1,2.4);
\draw [thin](-4.1,2)--(-3.8,2.4);
\draw [thin](-3.8,1.95)--(-3.5,2.35);
\draw [thin](-3.5, 1.9)--(-3.2,2.3);
\draw [thin](-3.2,1.85)--(-2.9,2.25);
\draw [thin](-2.9,1.80)--(-2.6,2.2);
\draw [thin](-2.6,1.75)--(-2.3,2.15);
\draw [thin](-2.3,1.70)--(-2.0,2.1);
\draw [thin](-2.0,1.65)--(-1.70,2.05);
\draw [thin](-1.7,1.60)--(-1.4,2.00);
\draw [thin](-1.4,1.55)--(-1.1,1.95);
\draw [thin](-1.1,1.50)--(-0.8,1.90);
\draw [thin](-0.8,1.45)--(-0.5,1.85);
\draw [thin](-0.5,1.40)--(-0.2,1.8);
\draw [thin](-0.2,1.35)--(0.1,1.75);
\draw [thin](0.1,1.30)--(0.4,1.70);
\draw [thin](0.4,1.30)--(0.7,1.70);
\draw [thin](0.7,1.30)--(1.0,1.70);
\draw [thin](1.0,1.30)--(1.3,1.70);
\draw [thin](1.3,1.30)--(1.6,1.70);

\draw [thick][red](-4,2)--(-3,1.3);

\draw [thick](-4.6,0.5)--(-3.2,0.8);
\draw [thick][dashed](-4.6,-0.7)--(-3.2,0.1);
\draw [thick][dashed][blue](-4.6,-1.5)--(-3.2,-1.5);
\draw [thick](-4.6,-2.2)--(-3.2,-2.7);

\draw [thick](-2.5,0.6)--(-2,1.65);
\draw [thick](-2,1.65)--(-1,0.8);

\draw [thick](-2.5,0)--(-2,-0.5);
\draw [thick](-2.5,-1)--(-2,-0.5);

\draw [thick][dashed][black](-2,-0.5)--(-1.5,0.3);
\draw [thick](-2,-0.5)--(-1.2,0);
\draw [thick][dashed][blue](-2,-0.5)--(-1.2,-0.5);
\draw [thick](-2,-0.5)--(-1.2,-1);

\draw [thick](-2.5,-1.5)--(-1.8,-2.2);
\draw [line width=0.04cm][red](-2.5,-2)--(-1.8,-2.2);

\draw [thick][dashed][black](-1.8,-2.2)--(-1.2,-1.5);
\draw [thick](-1.8,-2.2)--(-1.0,-1.8);
\draw [thick][dashed][blue](-1.8,-2.2)--(-1.0,-2.2);
\draw [line width=0.04cm][red](-1.8,-2.2)--(-1.2,-2.8);

\draw[thick](0,1.3)--(1,1);
\draw[thick](0,1.3)--(0.9,0.7);
\draw[thick](0,1.3)--(0.8,0.4);

\draw [thick][dashed](-0.7,-0.5)--(0.7,0.1);
\draw [thick](-0.7,-1)--(0.8,-1.2);
\draw [thick][dashed][blue](-0.7,-1.8)--(0.8,-1.8);
\draw [thick](-0.7,-3.0)--(0.8,-2.5);


\node at (3.4, 1.4) {$y=b_{h}(x)$};
\node at (2.6, 2) {$$};
\node at (-4, -4.0) {$x=x_{k-1}$};
\node at (0, -4.0) {$x=x_{k}$};
\end{tikzpicture}
\end{center}
\caption{Wave front tracking scheme}\label{fig4.1}
\end{figure}

Choose a  parameter $\nu\in \mathbb{N}_{+}$ and for the given initial data $U_{0}$, we define an sequence of piecewise constant functions $U^{\nu}_{0}$ such that
\begin{eqnarray}\label{eq:4.4}
\|U^{\nu}_{0}(\cdot)-U_{0}(\cdot)\|_{L^{1}(\mathcal{I})}<2^{-\nu}\quad \mbox{and}\quad T.V.\{U^{\nu}_{0}(\cdot); \mathcal{I}\}< T.V.\{U_{0}(\cdot); \mathcal{I}\}.
\end{eqnarray}

Then, the $(h,\nu)$-approximate solution $U^{(\tau)}_{h, \nu}$ in $\Omega_{h}$ is constructed in the following way:

At the initial line $x=0$, at each discontinuity point of $U^{\nu}_{0}$, a Riemann problem is solved with the solution consisting of shocks and rarefaction waves, according to Section 3. As done in \cite{bressan}, rarefaction waves need to be partitioned into several small central rarefaction fans with strength less than $\nu^{-1}$ (see \cite[pp.129--pp.132]{bressan} for the accurate Riemann solver, which is denoted by $(ARS)$ for simplicity, and the simplified Riemann solver, which is denoted by $(SRS)$ for simplicity). At the initial line, we always use the $(ARS)$.

The piecewise constant approximate solution $U^{(\tau)}_{h, \nu}$ can be prolonged until the interaction occur in $\Omega_{h}$, the reflection on non-corner points on $\Gamma_{h}$ (\emph{i.e.}, a 2th wave front hits the boundary and then a 1st reflected wave front is formed, see Appendix A), or the boundary corner (a weak wave will issue from the corner as constructed in Section 2.2). Then we can solve the Riemann problem again according to Section 3. When a wave hits the boundary on $\Gamma_{h}$ or when a weak wave generated at corner $A_{k} (k\geq 0)$, as done in \cite{amadori, colombo-guerra}, we always choose the $(ARS)$.
At the interaction in $\Omega_h$, to decide which Riemann solver is used, we introduce a threshold parameter $\varrho>0$ which is a function of $\nu$ (see  Remark \ref{rem:4.1} below). If the strength of the two wave fronts $\alpha$, $\beta$ satisfy that $|\alpha||\beta|>\varrho$, the $(ARS)$ is used otherwise we use the $(SRS)$. We remark that the simplified Riemann solver, in which all the new waves are lumped into a single non-physical wave traveling with a fixed speed $\hat{\lambda}$ larger than all the characteristic speed, is introduced to control the number of the wave fronts in $\Omega_{h}$.

We remark that in the above construction, we assume that no more than two wave fronts interact and that only one wave hits the boundary at the non-corner point, and that only one weak wave is generated by the corner point. 
It can be achieved by changing the speed of the corresponding single fronts slightly by a quality less than $2^{-\nu}$.

Denote all the fronts in $U^{(\tau)}_{h, \nu}$ by $\mathcal{J}(U^{(\tau)}_{h, \nu})=\mathcal{S}^{(\tau)}\cup \mathcal{R}^{(\tau)}\cup \mathcal{NP}^{(\tau)}$, \emph{i.e.}, shocks, rarefaction fans and non-physical wave. For each wave front, an order is introduced to count how many interactions were needed to produce such a front as follows. 

\rm (i) 
All weak wave fronts generated by the corner are order one.

\rm (ii) A wave front $\alpha$ of order $k_{\alpha}$ hits the boundary at the non-corner point $(\hat{x}, b_{h}(\hat{x}))$.
Then 
the order of the new wave from $(\hat{x}, b_{h}(\hat{x}))$ is set to be $k_{\alpha}$.

\rm(iii) The $i_{\alpha}$th wave front $\alpha$ of order $k_{\alpha}$, and $i_{\beta}$th wave
front $\beta$ of order $k_{\beta}$,
interact at $(\hat{x}, \hat{y})$. Assume $\alpha$ lies below $\beta$.
Then the orders of the new wave fronts are given as below:

$\quad$ $\rm (iii)_{1}$  If $k_{\alpha}, k_{\beta}<\nu$,
the outgoing wave fronts generated from $(\hat{x}, \hat{y})$
are constructed by the $(ARS)$, and
the order of the $j$th wave is given by
\begin{eqnarray}\label{eq:4.5}
\left\{
\begin{array}{llll}
\max\{k_{\alpha}, k_{\beta}\}+1,\quad &j \neq i_{\alpha}\mbox{ or }\ i_{\beta}, \\[5pt]
\min\{k_{\alpha}, k_{\beta}\}, \quad &j= i_{\alpha}= i_{\beta},\\[5pt]
k_{\alpha},\quad &j= i_{\alpha}\neq i_{\beta}, \\[5pt]
k_{\beta}, \quad &j= i_{\beta}\neq i_{\alpha}.
\end{array}
\right.
\end{eqnarray}

$\quad$ $\rm (iii)_{2}$ If $\max\{k_{\alpha}, k_{\beta}\}=\nu$, the outgoing wave fronts
generated from $(\hat{x}, \hat{y})$ are given by the $(SRS)$.
The order of the outgoing non-physical wave front is $\nu+1$.

$\quad$ $\rm (iii)_{3}$ If $k_{\alpha}=\nu+1$ (\emph{i.e.}, $\alpha$ is a non-physical wave front) and $k_{\beta}\leq\nu$,
then we use the $(SRS)$ to construct the outgoing wave fronts from $(\hat{x}, \hat{y})$.
The order of the outgoing non-physical wave front is $\nu+1$,
and the order of the outgoing physical wave front is $k_{\beta}$.


\begin{remark}\label{rem:4.1}
The small constant $\nu^{-1}$ controls the maximum strength of rarefaction fronts.
Moreover, the threshold parameter $\varrho$ is obviously a function of $\nu^{-1}$. 
\end{remark}

Applying the arguments in \cite{amadori,bressan,colombo-guerra} and Lemma \ref{lem-A1}, we have the following proposition for problem \eqref{eq:1.16}-\eqref{eq:1.18}.
\begin{proposition}\label{prop:4.1}
Under assumptions $\mathbf{(H1)}$-$\mathbf{(H2)}$, for a given hypersonic similarity parameter $a_{\infty}$,
there exist constants $\epsilon^*_{0}>0$ and $\tau^*_0>0$ depending only on $\underline{U}$ and $a_{\infty}$ such that if
$\tau\in(0,\tau^*_0)$ and
\begin{eqnarray}\label{eq:4.7}
T.V.\{U_{0}(\cdot); \mathcal{I}\}+|b'(0)|+T.V.\{b'(\cdot); \mathbb{R}_{+}\}<\epsilon^*_0,
\end{eqnarray}
then the wave front tracking scheme generates a global $(h,\nu)$-approximate solution $U^{(\tau)}_{h,\nu}$ satisfying
\begin{eqnarray}\label{eq:4.8}
\begin{split}
&\sup_{x>0}\big\|U^{(\tau)}_{h,\nu}(\cdot)-\underline{U}\big\|_{L^{\infty}((-\infty, b_{h}(x)))}
+T.V.\{U^{(\tau)}_{h,\nu}(x,\cdot); (-\infty, b_{h}(x))\}\\[5pt]
&\qquad \leq C_{1}\Big(T.V.\{U_{0}(\cdot); \mathcal{I}\}+|b'(0)|+T.V.\{b'(\cdot); \mathbb{R}_{+}\}\Big),
\end{split}
\end{eqnarray}
and
\begin{eqnarray}\label{eq:4.9}
\begin{split}
\big\|U^{(\tau)}_{h,\nu}(x',\cdot+b_{h}(x'))-U^{(\tau)}_{h,\nu}(x'',\cdot+b_{h}(x''))\big\|_{L^{1}((-\infty, 0))}\leq C_{2}|x'-x''|,
\end{split}
\end{eqnarray}
for any $x', x''>0$. Both the strength of each rarefaction front and the total strength of nonphysical waves in $U^{(\tau)}_{h,\nu}$ are small, i.e.,
\begin{eqnarray}\label{eq:4.10}
\max_{\alpha\in \mathcal{R}^{(\tau)}}|\alpha|\leq C_{2}\nu^{-1} \qquad \mbox{and} \qquad \sum_{\alpha\in \mathcal{NP}^{(\tau)}}|\alpha|\leq C_{2}2^{-\nu},
\end{eqnarray}
where constants $C_{1}>0$ and $C_{2}>0$ depend only on $\underline{U}$ and $a_{\infty}$.

Moreover, there exists a unique {$U^{(\tau)}\in (L^{1}_{loc}\cap BV_{loc})(\Omega)$} such that
\begin{eqnarray}\label{eq:4.11}
U^{(\tau)}_{h,\nu}(x,\cdot)-U^{(\tau)}(x,\cdot)\rightarrow 0, \quad \mbox{in}\quad L^{1}_{loc}(\Omega),  \quad \mbox{as}\quad h\rightarrow0, \ \nu\rightarrow+\infty,
\end{eqnarray}
where $U^{(\tau)}$ is the entropy solution of the initial-boundary value problem \eqref{eq:1.16}-\eqref{eq:1.18} with estimate \eqref{eq:4.8} with replacing $U^{(\tau)}_{h,\nu}$ in \eqref{eq:4.8} by $U^{(\tau)}$.
\end{proposition}

\begin{remark}\label{rem:4.2}
If $U_{0}(y)-U_{\infty}$ is compactly supported, the study of initial-boundary value problem \eqref{eq:1.16}-\eqref{eq:1.18}  can be simplified by studying Cauchy problem \eqref{eq:1.16}-\eqref{eq:1.17}.
Since functions $G(U,\tau^2)$ and $F(U,\tau^2)$ in \eqref{eq:1.15} are continuous with respect to parameter $\tau^2$, under the assumptions in Proposition \ref{prop:4.1}, we can follow the argument in
\cite[Chapters 8-9]{bressan} to show that for $\tau\in(0,\bar{\tau}^{*}_{0})$ with $\bar{\tau}^{*}_{0}$ depending only on $\underline{U}$ and $a_{\infty}$,
there exist a domain $\bar{\mathcal{D}}^{*}\subseteq L^{1}(\mathbb{R}; \mathbb{R}^{2})$, a $L^{1}$-Lipschitz semigroup $\mathcal{P}^{(\tau)}_{*}(x):\mathbb{R}_{+}\times \mathcal{D}^{*}\mapsto \mathcal{D}^{*}$, and a Lipschitz constant $L^{*}>0$ which depends only on $\underline{U}$ and $a_{\infty}$ such that

\rm(1) For every $U\in\mathcal{D}^{*}$ and $x', x''>0$,
\begin{eqnarray*}
\mathcal{P}^{(\tau)}_{*}(0)(U)=U, \qquad \mathcal{P}^{(\tau)}_{*}(x'+x'')(U)=\mathcal{P}^{(\tau)}_{*}(x'')\mathcal{P}^{(\tau)}_{*}(x')(U),
\end{eqnarray*}
and for any $U',\ U''\in\mathcal{D}^{*}$, there holds
\begin{align*}
\begin{split}
\|\mathcal{P}^{(\tau)}_{*}(x')(U')-\mathcal{P}^{(\tau)}_{*}(x'')(U'')\|_{L^{1}(\mathbb{R})}\leq L^{*}\Big(\|U'-U''\|_{L^{1}(\mathbb{R})}+|x'-x''|\Big);
\end{split}
\end{align*}

\rm(2) $U^{(\tau)}(x)=\mathcal{P}^{(\tau)}_{*}(x)(U_{0})$ is the entropy solution of Cauchy problem \eqref{eq:1.16}-\eqref{eq:1.17};

\rm(3) If $U$ is a  piecewise constant function, then for $x>0$ sufficiently small, $\mathcal{P}^{(\tau)}_{*}(x)(U)$ coincides with the solution to Cauchy problem \eqref{eq:1.16}-\eqref{eq:1.17}
by piecing together the Lax solutions of Riemann problems determined by the jumps of $U$.
\end{remark}

Let
\begin{small}
\begin{eqnarray}\label{eq:4.14}
\mathcal{D}_{h, x}\doteq \mathbf{cl}\left\{U_h(x,\cdot):\,U_{h}-\underline{U}\in (L^{1}\cap BV)(\mathbb{R};\mathbb{R}^{2})\Bigg|
\begin{array}{l}
U_{h}=\underline{U}, \ \ \mbox{for} \ \ y>b_{h}(x)\\[5pt]
\mbox{and}\ \ \mathbf{V}(U_{h})+\mathcal{K}\mathbf{Q}(U_{h})<\epsilon
\end{array}
\right \},
\end{eqnarray}
\end{small}
where $\mathbf{cl}$ denotes the closure in $L^{1}$-topology, and $\mathbf{V}$, $\mathbf{Q}$ are defined by
\begin{eqnarray*}
	\begin{split}
		\mathbf{V}(U_{h})=\mathbf{V}_{1}(U_{h})+\mathcal{K}_{b}\mathbf{V}_{2}(U_{h})
		+\mathcal{K}_{c}\mathbf{V}_{c}(U_{h}),
	\end{split}
\end{eqnarray*}
with constants $\mathcal{K}>0$, $\mathcal{K}_{b}>0$, $\mathcal{K}_c>0$ depending on $\tilde{K}_b$ in Lemma \ref{lem-A2} and $K_{c}$ in Lemma \ref{lem:3.2}, and
\begin{eqnarray*}
\begin{split}
&\mathbf{V}_{k}(U_{h})=\sum\{|\alpha_{k}|: \alpha_{k}\ is\ the\ k-th\ physical\ wave\ in\ U_{h}\},\quad \mbox{for}\ k=1, 2,\\[5pt]
&\mathbf{V}_{c}(U_{h})=\sum_{k>[\frac{x}{h}]}\{|\omega_{k}|: \omega_{k}\ is\ the\ changed\ angle\ at\ corner\ points\ A_{k}\ for\ k\geq 0\},
\end{split}
\end{eqnarray*}
and
\begin{eqnarray*}
	\begin{split}
		\mathbf{Q}(U_{h})=\sum\{|\alpha_i||\beta_j|: both\ \alpha_i\ and\ \beta_{j}\ are\ physical\ approaching\ waves\ for\ i,j=1,2 \}.
	\end{split}
\end{eqnarray*}

Due to \cite{amadori,colombo-guerra}, and by Lemma \ref{lem:3.2} and Lemma \ref{lem-A2}, we have the following proposition for problem \eqref{eq:1.20}-\eqref{eq:1.21} and \eqref{eq:1.11}.
\begin{proposition}\label{prop:4.2}
Under assumptions $\mathbf{(H1)}$-$\mathbf{(H2)}$, there exists a constant $\epsilon^{*}_{1}>0$ depending only on $\underline{U}$
such that if $|b'(0)|+T.V.\{b(\cdot);\mathbb{R}_{+}\}< \epsilon^{*}_{1}$,
then for any sufficiently small $h>0$, there exists 
a unique uniformly Lipschitz continuous map $\mathcal{P}_{h}(x,x'_{0})$:
\begin{eqnarray}\label{eq:2.63}
\begin{split}
\mathcal{P}_{h}(x,x'_{0}):\mathcal{D}_{h,x'_{0}}\mapsto \mathcal{D}_{h, x}, \ \ \ \mbox{for}\ \mbox{all}\ x,\ x'_{0}\geq 0,
\end{split}
\end{eqnarray}
whose trajectory $U_{h}(x,\cdot)=\mathcal{P}_{h}(x,x'_0)(U_{h}(x'_0,\cdot))\in \mathcal{O}_{C_{3}\epsilon^*_1}(\underline{U})$ are obtained as the limit $\nu\rightarrow +\infty$ from $(h,\nu)$-approximate solutions $U_{h,\nu}$ constructed via the wave front tracking scheme. Moreover,

\rm (i)\  for any $x\geq 0$ and $U_{h}(x,\cdot)\in \mathcal{D}_{h, x}$,  $\mathcal{P}_{h}(x,x)(U_{h}(x,\cdot))=U_{h}(x,\cdot)$, and
for all $x\geq x'\geq x'_{0}\geq 0$ and $U_{h}(x'_{0},\cdot)\in \mathcal{D}_{h, x'_{0}}$
\begin{eqnarray}\label{eq:2.65}
\begin{split}
\mathcal{P}_{h}(x,x'_{0})(U_{h}(x'_{0},\cdot))=\mathcal{P}_{h}(x,x')\circ \mathcal{P}_{h}(x',x'_{0})(U_{h}(x'_{0},\cdot));
\end{split}
\end{eqnarray}

\rm(ii)\ 
for any $x''\geq x'>x'_{0}>0$, $U_{h}(x'_{0},\cdot)\in \mathcal{D}_{h, x'_{0}}$, and $V_{h'}(x'_0,\cdot)\in \mathcal{D}_{h,x'_0}$,
\begin{eqnarray}\label{eq:2.66}
\begin{split}
&\ \ \ \big\|\mathcal{P}_{h}(x',x'_0)(U_{h}(x'_0,\cdot))-\mathcal{P}_{h'}(x'',x_0)(V_{h'}(x'_0,\cdot))
\big\|_{L^{1}((-\infty, \hat{b}_{h}(x'')))}\\[5pt]
&\leq L\bigg(|x'-x''|+\|U_{h}(x_0,\cdot)-V_{h'}(x'_0,\cdot)\|_{L^{1}((-\infty, \hat{b}_{h}(x'_0)))}
+\|b'_{h}-b'_{h'}\|_{L^{1}(\mathbb{R}_{+})}\bigg),
\end{split}
\end{eqnarray}
where $\hat{b}_{h}=\max\{b_{h},b_{h'}\}$, and the positive constants $C_3$ and $L$ depend only on $\underline{U}$;

\rm(iii)\ for any $x'_0, x \in [x_{k}, x_{k+1})$ with $x>x'_0$ and $U_{h}(x'_0,\cdot)\in \mathcal{D}_{h,x'_0}$, $\mathcal{P}_{h}(x,x'_0)(U_{h}(x'_0,\cdot))$ is obtained by piecing all the Riemann solutions \eqref{eq:3.1b}
and \eqref{eq:3.18} together, respectively.
\end{proposition}

Based on the Proposition \ref{prop:4.2}, we can further deduce that
\begin{proposition}\label{prop:4.3}
Under assumptions $\mathbf{(H1)}$-$\mathbf{(H2)}$, there exists a constant $\epsilon^{*}_{2}\in(0,\epsilon^{*}_{1})$ depending only on $\underline{U}$
such that if $|b'(0)|+T.V.\{b'(\cdot);\mathbb{R}_{+}\}< \epsilon$ for $\epsilon\in(0,\epsilon^{*}_{2})$,
then there exists 
a unique uniformly Lipschitz continuous map $\mathcal{P}(x,x'_{0}):\mathcal{D}_{x'_{0}}\mapsto \mathcal{D}_{ x}$ for $ x\geq x'_{0}\geq 0$, which satisfies:

\rm (1)\ there exists a subsequence $\{\mathcal{P}_{h_k}(x,x'_{0})\}^{+\infty}_{k=0}$ given in Proposition \ref{prop:4.2} for $h_k$ with $h_k\rightarrow0$ as $k\rightarrow\infty$, such that
$\mathcal{P}_{h_k}(x,x'_{0})(U(x'_0,\cdot))$ converges to $\mathcal{P}(x,x'_{0})(U(x'_0,\cdot))$ in $L^{1}((-\infty, b(x)))$
as $k\rightarrow +\infty$ for any $U(x'_0,\cdot)\in \mathcal{D}_{x'_0}$;

\rm (2)\ $U(x,\cdot)=\mathcal{P}(x,0)(U_0(\cdot))$ is the unique entropy solution to the initial-boundary value problem \eqref{eq:1.20}-\eqref{eq:1.21} and \eqref{eq:1.11} by the modified wave front tracking scheme;

\rm(3)\ $\mathcal{P}(x,x)(U(x,\cdot))=U(x,\cdot)$ for $U(x,\cdot)\in \mathcal{D}_{x}$, and
for all $x\geq x'\geq x'_{0}\geq 0$ 
\begin{eqnarray}\label{eq:4.15}
\begin{split}
\mathcal{P}(x,x'_{0})(U(x'_{0},\cdot))=\mathcal{P}(x,x')\circ \mathcal{P}(x',x'_{0})(U(x'_{0},\cdot)), \quad\mbox{for all}\quad  U(x'_{0},\cdot)\in \mathcal{D}_{x'_{0}};
\end{split}
\end{eqnarray}

\rm(4)\ if $\tilde{\mathcal{P}}$ and $\tilde{\mathcal{D}}_{x}$ are the map and domain corresponding
to boundary function $\tilde{b}(x)$ with $\tilde{b}(0)=b(0)$, then for any $x''\geq x'>x'_{0}>0$
and $\tilde{U}(x'_0,\cdot)\in \tilde{\mathcal{D}}_{x'_0}$,
\begin{eqnarray}\label{eq:4.16}
\begin{split}
&\ \ \ \big\|\mathcal{P}(x',x'_0)(\tilde{U}(x'_0,\cdot))-\tilde{\mathcal{P}}(x'',x'_0)(\tilde{U}(x'_0,\cdot))
\big\|_{L^{1}((-\infty, \hat{b}(x'')))}\\[5pt]
&\leq L\bigg(|x'-x''|+\|U(x'_0,\cdot)-\tilde{U}(x'_0,\cdot)\|_{L^{1}((-\infty, \hat{b}(x'_0)))}
+\|b'-\tilde{b}'\|_{L^{1}(\mathbb{R}_{+})}\bigg),
\end{split}
\end{eqnarray}
where $\hat{b}=\max\{b, \tilde{b}\}$ and the constant $L$ depends only on $\underline{U}$.
\end{proposition}

\begin{remark}\label{rem:4.3}
When $U_{0}-\underline{U}$ is compactly supported, the study of the initial-boundary value problem \eqref{eq:1.20}-\eqref{eq:1.21} and \eqref{eq:1.11}
is equivalent to the study of the Cauchy problem \eqref{eq:1.20}-\eqref{eq:1.21}. So similarly as done in Remark \ref{rem:4.2},
we can also construct an approximate solution $U^{\nu}$ to Cauchy problem \eqref{eq:1.20}-\eqref{eq:1.21} via the wave front tracking scheme with its jumps being denoted by $\mathcal{J}(U^{\nu})=\mathcal{S}\cup\mathcal{R}\cup{\mathcal{NP}}$.
Then, under the assumptions given in Proposition \ref{prop:4.2}, we know that $\nu$-approximate solution $U^{\nu}$ satisfies
\begin{eqnarray*}
\begin{split}
&\sup_{x>0}\big\|U^{\nu}-\underline{U}\big\|_{L^{\infty}(\mathbb{R})}+T.V.\{U^{\nu}(x,\cdot); \mathbb{R}\}<+\infty,
\end{split}
\end{eqnarray*}
and
\begin{eqnarray*}
\begin{split}
\big\|U^{\nu}(x',\cdot)-U^{\nu}(x'',\cdot)\big\|_{L^{1}(\mathbb{R})}\leq C_{3}|x'-x''|, \quad \mbox{for}\quad x', x''>0.
\end{split}
\end{eqnarray*}

Both the strength of each rarefaction front and the total strength of nonphysical waves in $U^{\nu}$ are small, i.e.,
\begin{eqnarray*}
\max_{\alpha\in \mathcal{R}}|\alpha|\leq C_{4}\nu^{-1} \qquad \mbox{and} \qquad \sum_{\alpha\in \mathcal{NP}}|\alpha|\leq C_{4}2^{-\nu},
\end{eqnarray*}
where constants $C_{3}>0$ and $C_{4}>0$ depend only on $\underline{U}$ and $a_{\infty}$. Therefore, there exists a $U\in (L^{1}_{loc}\cap BV_{loc})(\mathbb{R}_{+}\times \mathbb{R})$ such that
\begin{eqnarray*}
U^{\nu}(x,\cdot)\rightarrow U(x,\cdot), \quad \mbox{in}\quad L^{1}_{loc}(\mathbb{R}_{+}\times \mathbb{R}),  \quad \mbox{as}\quad \nu\rightarrow+\infty,
\end{eqnarray*}
where $U$ is the entropy solution of problem \eqref{eq:1.20}--\eqref{eq:1.21}. Moreover, by Theorem 7.1 in \cite{liu}, we can further have that $U(x,y)$ satisfies
\begin{eqnarray*}
T.V.\{U(x,\cdot); \mathbb{R}\}<C_{5}x^{-\frac{1}{2}},
\end{eqnarray*}
where constant $C_{5}>0$ depends only on $\underline{U}$ and $a_{\infty}$.
\end{remark}

\section{$L^1$ difference estimates between solutions to problem \eqref{eq:1.16}-\eqref{eq:1.18} and to problem \eqref{eq:1.20}-\eqref{eq:1.21} and \eqref{eq:1.11}}
In this section, we will establish the $L^1$ difference between the entropy solutions to initial-boundary value problem
\eqref{eq:1.16}-\eqref{eq:1.18} and to problem \eqref{eq:1.20}-\eqref{eq:1.21} and \eqref{eq:1.11} for the initial data $U_{0}$ without (or with) compact support.
To make it, we will first consider the local $L^1$ difference estimate
and then the global $L^1$ difference estimate between the trajectory
of $\mathcal{P}_{h}(x,0)(U^{\nu}_{0}(y))$ and the approximate solution $U^{(\tau)}_{h,\nu}(x,y)$. Then, we pass the limit $h\rightarrow0$ and $\nu\rightarrow+\infty$.

\subsection{Local $L^1$ differences estimate between the approximate solution $U^{(\tau)}_{h,\nu}(x,y)$ and the trajectory $\mathcal{P}_{h}(x,0)(U^{\nu}_{0}(y))$}
Let us consider the approximate solution $U^{(\tau)}_{h,\nu}$ to the initial-boundary value problem \eqref{eq:1.16}-\eqref{eq:1.18}. First, assume that there is only one jump point at $(\varsigma,y_{I})$ away from the boundary when $x=\varsigma$, that is, if
\begin{eqnarray}\label{eq:5.1}
U_{L}\doteq U^{(\tau)}_{h,\nu}(\varsigma, y_{I}-),\quad U_{R}\doteq U^{(\tau)}_{h,\nu}(\varsigma, y_{I}+),
\end{eqnarray}
then
\begin{eqnarray}\label{eq:5.2}
U^{(\tau)}_{h,\nu}(\varsigma,y)=
\left\{
\begin{array}{llll}
U_{R}, \quad & y>y_{I},\\[5pt]
U_L, \quad & y<y_{I}.
\end{array}
\right.
\end{eqnarray}

\begin{lemma}\label{lem:5.1}
Assume $U^{(\tau)}_{h,\nu}(\varsigma)$ is a piecewise constant function defined as in \eqref{eq:5.2} with $U_{L}, U_{R}\in \displaystyle{\mathcal{O}_{\min\{C_1\epsilon^*_0, C_3\epsilon^*_1\}}(\underline{U})}$ for $\tau\in(0,\tau^{*}_{0})$.
Let $\mathcal{P}_{h}$ be the uniformly Lipschtiz continuous map given by Proposition \ref{prop:4.2}.
Let $\hat{\lambda}$ be a fixed constant satisfying $\hat{\lambda}>\max\{\lambda_{j}(U^{(\tau)},\tau^{2}), \lambda_{j}(U)\}$
for all $U^{(\tau)}, U\in \mathcal{O}_{\min\{C_1\epsilon^*_0, C_3\epsilon^*_1\}}(\underline{U})$, $\tau\in(0,\tau^*_0)$ and $j=1,2$.

Denote
\begin{eqnarray}\label{eq:5.3}
U^{(\tau)}_{h,\nu}(\varsigma+s, y)=
\left\{
\begin{array}{llll}
U_{R}, \quad & y>y_{I}+\dot{y}_{\alpha_k}s,\\[5pt]
U_L, \quad & y<y_{I}+\dot{y}_{\alpha_k}s,
\end{array}
\right.
\end{eqnarray}
where $|\dot{y}_{\alpha_k}|\leq \hat{\lambda}$ for $k=1,2$. Then, for sufficiently small $s>0$,

\rm (1)\ if $U_{L}$ and $U_{R}$ are connected by a $k$th-shock wave $\alpha_k\in \mathcal{S}^{(\tau)}_k,(k=1,2)$, that is, $U_{R}=\Phi_{k}(\alpha_k; U_{L},\tau^2)$ for $\alpha_k<0$, and if $|\dot{y}_{\alpha_k}-\dot{\mathcal{S}}_{k}(\alpha_k, \tau^2)|<2^{-\nu}$,$(k=1,2)$, then
\begin{eqnarray}\label{eq:5.4}
\begin{split} \int^{y_{I}+\hat{\lambda}s}_{y_{I}-\hat{\lambda}s}\big|\mathcal{P}_{h}(\varsigma+s,\varsigma)(U^{(\tau)}_{h,\nu}(\varsigma,\cdot))
-U^{(\tau)}_{h,\nu}(\varsigma+s,\cdot)\big|dy
\leq C_{4}|\alpha_{k}|\big(\tau^{2}+2^{-\nu}\big)s,
\end{split}
\end{eqnarray}
where $\dot{\mathcal{S}}_{k}(\alpha_k, \tau^2)$ is the speed of the $k$th-shock wave;

\rm (2)\ if $U_{L}$ and $U_{R}$ are connected by a $k$th-rarefaction front $\alpha_k\in \mathcal{R}^{(\tau)}_k,(k=1,2)$, that is, $U_{R}=\Phi_{k}(\alpha_k; U_{L},\tau^2)$ for $\alpha_k>0$, and if $|\dot{y}_{\alpha_k}-\lambda_{k}(U_{R}, \tau^2)|<2^{-\nu}$,$(k=1,2)$, then
\begin{eqnarray}\label{eq:5.5}
\begin{split} \int^{y_{I}+\hat{\lambda}s}_{y_{I}-\hat{\lambda}s}\big|\mathcal{P}_{h}(\varsigma+s,\varsigma)(U^{(\tau)}_{h,\nu}(\varsigma))
-U^{(\tau)}_{h,\nu}(\varsigma+s)\big|dy
\leq C_{4}\big(\tau^{2}+\nu^{-1}+2^{-\nu}\big)|\alpha_{k}|s,
\end{split}
\end{eqnarray}
where $\lambda_{k}(U_{R}, \tau^2)$ is the speed of the $k$th-rarefaction front;

\rm (3)\ if $U_{L}$ and $U_{R}$ are connected by a non-physical wave $\alpha_{\mathcal{NP}}\in \mathcal{NP}^{(\tau)}$, that is, $\alpha_{\mathcal{NP}}=|U_{R}-U_{L}|$ with speed $\hat{\lambda}$, then
\begin{eqnarray}\label{eq:5.6}
\begin{split} \int^{y_{I}+\hat{\lambda}h}_{y_{I}-\hat{\lambda}h}\big|\mathcal{P}_{h}(\varsigma+s,\varsigma)(U^{(\tau)}_{h,\nu}(\varsigma))
-U^{(\tau)}_{h,\nu}(\varsigma+s)\big|dy
\leq C_{4}\alpha_{\mathcal{NP}}s.
\end{split}
\end{eqnarray}
Here, constant $C_{4}>0$ depends only on $\underline{U}$ and $a_{\infty}$, but not on $\tau, h, \nu$.
\end{lemma}

\begin{proof}
$\rm (1)$ \ Without loss of the generality, we assume $k=1$, because the case $k=2$ can be dealt with in the same way.
From Proposition \ref{prop:4.2}, we know that $\mathcal{P}_{h}(\varsigma+s,\varsigma)(U^{(\tau)}_{h,\nu}(\varsigma,y))$
is the approximate solution at $x=\varsigma+s$ obtained by Riemann solver with Riemann data $U_{L}$ and $U_{R}$, which can be derived by
\begin{eqnarray}\label{eq:5.7}
\Phi(\beta_1,\beta_2; U_L)=\Phi_{1}(\alpha_1;U_L,\tau^2).
\end{eqnarray}

Then, by Proposition \ref{prop:3.1}, equation \eqref{eq:5.7} admits a unique solution $\beta_{j}$, such that
\begin{eqnarray}\label{eq:5.8}
\beta_{1}=\alpha_1+O(1)|\alpha_{1}|\tau^2,\quad \beta_{2}=O(1)|\alpha_{1}|\tau^2,
\end{eqnarray}
where $\beta_1<0$ for $\tau$ and $\alpha_1$ being sufficiently small.

\vspace{-5mm}
\begin{figure}[ht]
\begin{center}
\begin{tikzpicture}[scale=0.7]
\draw [line width=0.05cm](-4,3.5)--(-4,-4.5);
\draw [line width=0.05cm](0.5,3.5)--(0.5,-4.5);

\draw [thick](-4,0)--(0.5, 2.5);
\draw [thick](-4,0)--(0.5, 0.6);
\draw [thick][blue](-4,0)--(0.5, -2.5);
\draw [thick][red](-4,0)--(0.5, -1.2);
\draw [thin](-4,0)--(0.5, -3.8);

\draw [thin][<->](-3.9,3.2)--(0.4,3.2);

\node at (2.6, 2) {$$};
\node at (-1.8, 3.5){$s$};
\node at (-5.0, 0) {$(\varsigma, y_{I})$};
\node at (-4, -4.9) {$x=\varsigma$};
\node at (0, -4.9) {$x=\varsigma+s$};

\node at (2.1, 2.4) {$y=y_{I}+\hat{\lambda}s$};
\node at (1.2, 0.5) {$\beta_{2}$};
\node at (1.2, -1.2) {$\alpha_1$};
\node at (1.2, -2.6) {$\beta_{1}$};
\node at (2.1, -3.8) {$y=y_{I}-\hat{\lambda}s$};

\node at (-1.5, 2.5){$U_{R}$};
\node at (-1.2, -0.2){$U_{M}$};
\node at (-1.5, -3.2){$U_{L}$};

\end{tikzpicture}
\end{center}
\caption{Comparison of Riemann solvers for $\alpha\in \mathcal{S}^{(\tau)}$}\label{fig5.1}
\end{figure}
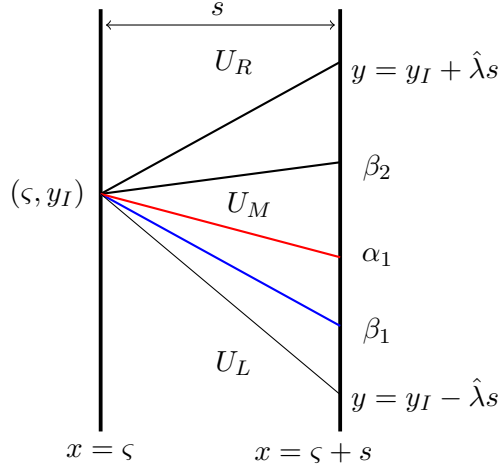

As shown in Fig.\ref{fig5.1},
let $U_{M}$ be the middle state of the Riemann solution $\mathcal{P}_{h}(\varsigma+s,y)(U^{(\tau)}_{h,\nu}(\varsigma,y))$.
By Lemma \ref{lem:3.1}, relation \eqref{eq:3.3} holds. Let $\dot{\mathcal{S}}_{1}(\alpha_1,\tau^2)$ and $\dot{\mathcal{S}}_{1}(\beta_{1})$
be the speeds of the shock fronts $\alpha_1$ and $\beta_{1}$, respectively.
From Lemma \ref{lem:3.1}, we have
\begin{eqnarray*}
\begin{split}
\dot{\mathcal{S}}_{1}(\alpha_1,\tau^2)\Big|_{\alpha_{1}=0}=\lambda_{1}(U_{L},\tau^2), \quad
\frac{\partial\dot{\mathcal{S}}_{1}(\alpha_1,\tau^2) }{\partial \alpha_1}\Bigg|_{\alpha_{1}=0}
=\frac{1}{2}\nabla_{U}\lambda_{1}(U_{L},\tau^2)\cdot \boldsymbol{r}_{1}(U_{L},\tau^2)=1,
\end{split}
\end{eqnarray*}
and
\begin{eqnarray*}
\begin{split}
\dot{\mathcal{S}}_{1}(\beta_1)\Big|_{\beta_1=0}=\lambda_{1}(U_{L}), \quad \frac{\partial\dot{\mathcal{S}}_{1}(\beta_1) }{\partial \beta_1}\Bigg|_{\beta_1=0}
=\frac{1}{2}\nabla_{U}\lambda_{1}(U_{L})\cdot \boldsymbol{r}_{1}(U_{L})=1.
\end{split}
\end{eqnarray*}

Thus, by the Taylor formula and \eqref{eq:5.8}, we obtain
\begin{eqnarray}\label{eq:5.9}
\begin{split}
\dot{\mathcal{S}}_{1}(\beta_1)-\dot{\mathcal{S}}_{1}(\alpha_1,\tau^2)
&=\dot{S}_{1}(\beta_1)\Big|_{\beta_{1}=0}
+\frac{\partial\dot{\mathcal{S}}_{1}(\beta_1) }{\partial \beta_1}\bigg|_{\beta_1=0}\beta_1+O(1)\beta^{2}_1\\[5pt]
&\ \ \ -\bigg(\dot{\mathcal{S}}_{1}(\alpha_{1},\tau^2)\Big|_{\alpha_{1}=0}+\frac{\partial\dot{\mathcal{S}}_{1}(\alpha_{1},\tau^2) }{\partial \alpha_{1}}\bigg|_{\alpha_{1}=0}\alpha_{1}+O(1)\alpha^{2}_{1}\bigg)\\[5pt]
&=\lambda_{1}(U_{L})-\lambda_{1}(U_{L},\tau^{2})+\frac{1}{2}\big(\beta_1-\alpha_{1}\big)
+O(1)\big(\beta_1-\alpha_{1}\big)^{2}+O(1)|\beta_1|^{2}\tau^2\\[5pt]
&=O(1)\big(1+|\alpha_{1}|\big)\tau^{2},
\end{split}
\end{eqnarray}
where we have used the fact that $\lambda_{1}(U_{L},\tau^{2})\Big|_{\tau=0}=\lambda_{1}(U_{L})$.

So, we have
\begin{eqnarray}\label{5.10}
\begin{split}
&\ \ \ \int^{y_{I}+\hat{\lambda}s}_{y_{I}-\hat{\lambda}h}\big|\mathcal{P}_{h}(\varsigma+s,\varsigma)(U^{(\tau)}_{h,\nu}(\varsigma,\cdot))
-U^{(\tau)}_{h,\nu}(\varsigma+s,\cdot)|dy\\[5pt]
& =\int^{\min\{y_{I}+\dot{\mathcal{S}}_{1}(\beta_1)s,y_{I}+\dot{y}_{\alpha_1}s\}}
_{y_{I}-\hat{\lambda}s}\big|\mathcal{P}_{h}(\varsigma+s,\varsigma)(U^{(\tau)}_{h,\nu}(\varsigma,\cdot))
-U^{(\tau)}_{h,\nu}(\varsigma+s,\cdot)|dy\\[5pt]
&\ \ \ +\int^{\max\{y_{I}+\dot{\mathcal{S}}_{1}(\beta_1)s,y_{I}+\dot{y}_{\alpha_1}s\}}
_{\min\{y_{I}+\dot{\mathcal{S}}_{1}(\beta_1)s,y_{I}+\dot{y}_{\alpha_1}s\}}
\big|\mathcal{P}_{h}(\varsigma+s,\varsigma)(U^{(\tau)}_{h,\nu}(\varsigma,\cdot))
-U^{(\tau)}_{h,\nu}(\varsigma+s,\cdot)|dy\\[5pt]
&\ \ \ +\int^{y_{I}+\hat{\lambda}h}
_{\max\{y_{I}+\dot{\mathcal{S}}_{1}(\beta_1)s,y_{I}+\dot{y}_{\alpha_1}s\}}
\big|\mathcal{P}_{h}(\varsigma+s,\varsigma)(U^{(\tau)}_{h,\nu}(\varsigma,\cdot))
-U^{(\tau)}_{h,\nu}(\varsigma+s,\cdot)|dy\\[5pt]
&\doteq I'_{\mathcal{S}}+I''_{\mathcal{S}}+I'''_{\mathcal{S}}.
\end{split}
\end{eqnarray}

Obviously, $I'_{\mathcal{S}}=0$.  For the second term $I''_{S}$, by Lemma \ref{lem:3.1}, we get either
\begin{eqnarray*}
\begin{split}
I''_{\mathcal{S}}&\leq \big|\dot{\mathcal{S}}(\beta_1)-\dot{y}_{\alpha_1}\big|\big|U_{M}-U_{L}\big|s\\[5pt]
&\leq \Big(\big|\dot{\mathcal{S}}_1(\beta_1)-\dot{\mathcal{S}}_1(\alpha_1,\tau^2)\big|
+\big|\dot{\mathcal{S}}_1(\alpha_1,\tau^2)-\dot{y}_{\alpha_1}\big|\Big)|\beta_1|s\\[5pt]
&\leq O(1)\Big(\big(1+|\alpha_{1}|\big)\tau^{2}+2^{-\nu}\Big)\big(1+O(1)\tau^2\big)|\alpha_1|s\\[5pt]
&\leq O(1)|\alpha_{1}|(\tau^{2}+2^{-\nu})s,
\end{split}
\end{eqnarray*}
or
\begin{eqnarray*}
\begin{split}
I''_{\mathcal{S}}&\leq \big|\dot{\mathcal{S}}(\beta_1)-\dot{y}_{\alpha_1}\big|\big|U_{R}-U_{L}\big|s\\[5pt]
&\leq O(1)\Big(\big(1+|\alpha_{1}|\big)\tau^{2}+2^{-\nu}\Big)|\alpha_1|s\\[5pt]
&\leq O(1)|\alpha_{1}|(\tau^{2}+2^{-\nu})s.
\end{split}
\end{eqnarray*}
For $I'''_{\mathcal{S}}$, by Lemma \ref{lem:3.1} again, we obtain
\begin{eqnarray*}
\begin{split}
I'''_{\mathcal{S}}\leq O(1)\big|U_{M}-U_{R}\big|s \leq O(1)|\alpha_{1}|\tau^{2}s.
\end{split}
\end{eqnarray*}

Therefore, estimate \eqref{eq:5.3} holds by combing the estimates on $I'_{\mathcal{S}}$, $I''_{\mathcal{S}}$, and $I'''_{\mathcal{S}}$, for $\tau$ sufficiently small.

\smallskip
$\rm(2)$ Without loss of the generality, we only consider the case $k=1$ again. Similar to \rm (1) in \eqref{eq:5.7}-\eqref{eq:5.8}, by Lemma \ref{lem:3.1}, we can get $\beta_j, (j=1,2)$ of Riemann solution $\mathcal{P}_{h}(\varsigma+s,\varsigma)(U^{(\tau)}_{h,\nu}(\varsigma,y))$,  satisfying \eqref{eq:5.8} with $\beta_1>0$. 

Let $U_M$ be the middle state of the Riemann solution $\mathcal{P}_{h}(\varsigma+s,\varsigma)(U^{(\tau)}_{h,\nu}(\varsigma,y))$ satisfying \eqref{eq:3.3} in Lemma \ref{lem:3.1}. Then
\begin{eqnarray*}
\begin{split}
\lambda_{1}(U_{M})&=\lambda_{1}(U_{L})+(\nabla_{U}\lambda_1\cdot \boldsymbol{r}_1)(U_L)\beta_1+O(1)\beta^{2}_{1}
=\lambda_{1}(U_{L})+\beta_1+O(1)\beta^{2}_{1},
\end{split}
\end{eqnarray*}
and
\begin{eqnarray*}
\begin{split}
\lambda_{1}(U_{R},\tau^2)&=\lambda_{1}(U_{L},\tau^2)+\alpha_1+O(1)\alpha^{2}_{1}.
\end{split}
\end{eqnarray*}

Thus,
\begin{align}\label{eq:5.10}
\begin{split}
\lambda_{1}(U_{M})-\lambda_{1}(U_{R},\tau^2)&=\lambda_{1}(U_L)-\lambda_{1}(U_L,\tau^2)+\beta_1-\alpha_1\\[5pt]
&\quad \ +O(1)(\beta_1-\alpha_1)^2+O(1)|\beta_1|^{2}\tau^2\\[5pt]
&=O(1)\big(1+|\alpha_1|\big)\tau^2.
\end{split}
\end{align}
and
\begin{eqnarray}\label{eq:5.11}
\begin{split}
\lambda_{1}(U_{R}, \tau^{2})-\lambda_1(U_L)&=\lambda_{1}(U_{L},\tau^{2})-\lambda_{1}(U_L)+\alpha_1+O(1)\alpha^{2}_1
=O(1)\big(\alpha_1+\tau^2\big),
\end{split}
\end{eqnarray}
where we used the fact that $\lambda_{1}(U_{L},\tau^{2})\big|_{\tau=0}=\lambda_{1}(U_L)$.

As done in \eqref{5.10}, we can rewrite $\int^{y_{I}+\hat{\lambda}s}_{y_{I}-\hat{\lambda}s}|\mathcal{P}_{h}(\varsigma+s, \varsigma)(U^{(\tau)}_{h,\nu}(\varsigma,\cdot))-U^{(\tau)}_{h,\nu}(\varsigma+s,\cdot)|dy$
into the summations of three terms. That is
\begin{eqnarray} \int^{y_{I}+\hat{\lambda}s}_{y_{I}-\hat{\lambda}h}\big|\mathcal{P}_{h}(\varsigma+s,\varsigma)(U^{(\tau)}_{h,\nu}(\varsigma,\cdot))
-U^{(\tau)}_{h,\nu}(\varsigma+s,\cdot)|dy\doteq I'_{\mathcal{R}}+I''_{\mathcal{R}}+I'''_{\mathcal{R}}.
\end{eqnarray}
First, by \eqref{eq:5.8}, we have
\begin{eqnarray*}
\begin{split}
I'_{\mathcal{R}}\doteq\int^{\min\{y_{I}+\dot{y}_{\alpha_1}s, y_{I}+\lambda_{1}(U_{L})s\}}_{y_\alpha-\hat{\lambda}s}
\big|\mathcal{P}_{h}(\varsigma+s,\varsigma)(U^{(\tau)}_{h,\nu}(\varsigma,\cdot))-U^{(\tau)}_{h,\nu}(\varsigma+s,\cdot)\big|dy=0,
\end{split}
\end{eqnarray*}
and
\begin{eqnarray*}
\begin{split}
I'''_{\mathcal{R}}&\doteq\int^{y_{I}+\hat{\lambda}s}_{\max\{y_{I}+\dot{y}_{\alpha_1}s, y_{I}+\lambda_{1}(U_{M})s\}}
\big|\mathcal{P}_{h}(\varsigma+s,\varsigma)(U^{(\tau)}_{h,\nu}(\varsigma,\cdot))-U^{(\tau)}_{h,\nu}(\varsigma+s,\cdot)\big|dy\\[5pt]
&\leq O(1)\hat{\lambda}|U_{M}-U_{R}|s\\[5pt]
&\leq O(1)|\beta_2|\leq O(1)|\alpha_{1}|\tau^{2}.
\end{split}
\end{eqnarray*}

Next, the second term $I''_{\mathcal{R}}$, that is
\begin{eqnarray*}
\begin{split}
I''_{\mathcal{R}}\doteq\int^{\max\{y_{I}+\dot{y}_{\alpha_1}s, y_{I}+\lambda_{1}(U_{M})s\}}_{\min\{y_{I}+\dot{y}_{\alpha_1}s, y_{I}+\lambda_{1}(U_{L})s\}}\big|\mathcal{P}_{h}(\varsigma+s, \varsigma)(U^{(\tau)}_{h,\nu}(\varsigma,\cdot))
-U^{(\tau)}_{h, \nu}(\varsigma+s, \cdot)\big|dy,
\end{split}
\end{eqnarray*}
will be considered by three cases.

$\emph{Case 1}$: $\dot{y}_{\alpha_1}\geq \lambda_{1}(U_{M})$ (see Fig.\ \ref{fig5.2}). By Proposition \ref{prop:4.1}, estimate \eqref{eq:5.10}, and the properties of rarefaction waves $\beta_1$, we have
\begin{eqnarray*}
\begin{split}
I''_{\mathcal{R}}&=\int^{y_{I}+\dot{y}_{\alpha_1}s}_{y_{I}+\lambda_{1}(U_{L})s}
\big|\mathcal{P}_{h}(\varsigma+s, \varsigma)(U^{(\tau)}_{h,\nu}(\varsigma,\cdot))
-U^{(\tau)}_{h, \nu}(\varsigma+s, \cdot)\big|dy\\[5pt]
&\leq|\dot{y}_{\alpha_1}-\lambda_1(U_L)||U_{L}-U_{M}|s\\[5pt]
&\leq O(1)\Big(\big|\lambda_{1}(U_{R},\tau^2)- \lambda_{1}(U_{L})\big|+2^{-\nu}\Big)|\beta_1|s\\[5pt]
&\leq O(1)\big(\tau^{2}+\nu^{-1}+2^{-\nu}\big)|\alpha_1|s.
\end{split}
\end{eqnarray*}

\vspace{-3mm}
\begin{figure}[ht]
\begin{center}
\begin{tikzpicture}[scale=0.7]
\draw [line width=0.05cm](-4,3.5)--(-4,-4.0);
\draw [line width=0.05cm](0.5,3.5)--(0.5,-4.0);

\draw [thin](-4,0)--(0.5, 2.5);
\draw [thick](-4,0)--(0.5, 0.6);
\draw [thick][blue](-4,0)--(0.5, -2.0);
\draw [thick][blue](-4,0)--(0.5, -2.3);
\draw [thick][blue](-4,0)--(0.5, -2.6);
\draw [thick][blue](-4,0)--(0.5, -2.9);
\draw [thick][red](-4,0)--(0.5, -1.2);
\draw [thin](-4,0)--(0.5, -3.8);

\draw [thin][<->](-3.9,3.0)--(0.4,3.0);

\node at (2.6, 2) {$$};
\node at (-1.8, 4.0){$s$};
\node at (-4.9, 0) {$(\varsigma, y_{I})$};
\node at (-4, -4.9) {$x=\varsigma$};
\node at (0.5, -4.9) {$x=\varsigma+s$};

\node at (2.1, 2.4) {$y=y_{I}-\hat{\lambda}s$};
\node at (1.0, 0.5) {$\beta_{2}$};
\node at (1.0, -1.2) {$\alpha_1$};
\node at (1.0, -2.6) {$\beta_{1}$};
\node at (2.1, -3.8) {$y=y_{I}-\hat{\lambda}s$};

\node at (-1.8, 2.2){$U_{R}$};
\node at (-1.6, -0.2){$U_{M}$};
\node at (-1.8, -3.0){$U_{L}$};

\end{tikzpicture}
\end{center}
\caption{Comparison of Riemann solvers for $\alpha_1\in \mathcal{R}^{(\tau)}$ and $\dot{y}_{\alpha_1}\geq \lambda_{1}(U_{M})$}\label{fig5.2}
\end{figure}

$\emph{Case 2}$: $\lambda_{1}(U_{L})<\dot{y}_{\alpha_1}< \lambda_{1}(U_{M})$ (see Fig.\ \ref{fig5.3}).
By estimate \eqref{eq:5.11}, Proposition \ref{prop:4.1}, and the properties of rarefaction waves $\beta_1$, we have
\begin{eqnarray*}
\begin{split}
I''_{\mathcal{R}}&=\int^{y_{I}+\lambda_{1}(U_M)s}_{y_{I}+\lambda_{1}(U_{L})s}
\big|\mathcal{P}_{h}(\varsigma+s,\varsigma)(U^{(\tau)}_{h,\nu}(\varsigma,y))-U^{(\tau)}_{h, \nu}(\varsigma+s, y)\big|dy\\[5pt]
&=\Bigg(\int^{y_{I}+\lambda_{1}(U_M)s}_{y_{I}+\dot{y}_{\alpha_1}s}
+\int^{y_{I}+\dot{y}_{\alpha_1}s}_{y_{I}+\lambda_{1}(U_{L})s}\Bigg)
\big|\mathcal{P}_{h}(\varsigma+s,\varsigma)(U^{(\tau)}_{h,\nu}(\varsigma,y))-U^{(\tau)}_{h, \nu}(\varsigma+s, y)\big|dy  \\[5pt]
&\leq |\dot{y}_{\alpha_1}-\lambda_1(U_L)||U_{M}-U_{L}|s+|\dot{y}_{\alpha_1}-\lambda_1(U_M)||U_{R}-U_{L}|s \\[5pt]
&\leq |\lambda_1(U_R, \tau^2)-\lambda_1(U_L)||\beta_1|s+|\lambda_1(U_R, \tau^2)-\lambda_1(U_M)||\alpha_1|s
+O(1)2^{-\nu}\big(|\beta_1|+|\alpha_1|\big)s\\[5pt]
&\leq O(1)\big(\tau^2+\nu^{-1}+2^{-\nu}\big)|\alpha_1|s.
\end{split}
\end{eqnarray*}

\vspace{-3mm}
\begin{figure}[ht]
\begin{center}
\begin{tikzpicture}[scale=0.7]
\draw [line width=0.05cm](-4,3.1)--(-4,-4.1);
\draw [line width=0.05cm](0.5,3.1)--(0.5,-4.1);

\draw [thin](-4,0)--(0.5, 2.2);
\draw [thick](-4,0)--(0.5, 0.6);
\draw [thick][blue](-4,0)--(0.5, -1.1);
\draw [thick][red](-4,0)--(0.5, -1.5);
\draw [thick][blue](-4,0)--(0.5, -2.3);
\draw [thick][blue](-4,0)--(0.5, -1.9);
\draw [thin](-4,0)--(0.5, -3.4);

\draw [thin][<->](-3.9,2.7)--(0.4,2.7);

\node at (2.6, 2) {$$};
\node at (-1.8, 3.2){$s$};
\node at (-4.9, 0) {$(\varsigma, y_{I})$};
\node at (-4, -4.9) {$x=\varsigma$};
\node at (0, -4.9) {$x=\varsigma+s$};

\node at (2.2, 2.2) {$y=y_{I}+\hat{\lambda}s$};
\node at (1.1, 0.5) {$\beta_{2}$};
\node at (1.1, -1.6) {$\alpha_1$};
\node at (1.1, -2.4) {$\beta_{1}$};
\node at (2.2, -3.8) {$y=y_{I}-\hat{\lambda}h$};

\node at (-1.6, 2.0){$U_{R}$};
\node at (-1.4, -0.2){$U_{M}$};
\node at (-1.6, -3.0){$U_{L}$};

\end{tikzpicture}
\end{center}
\caption{Comparison of Riemann solvers for $\alpha_1\in \mathcal{R}^{(\tau)}$  and $\lambda_{1}(U_{L})<\dot{y}_{\alpha_1}< \lambda_{1}(U_{M})$}\label{fig5.3}
\end{figure}
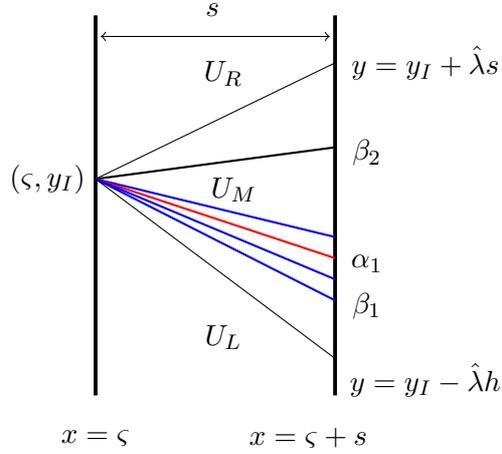

$\emph{Case 3}$: $\dot{y}_{\alpha_1}<\lambda_1(U_L)$ (see Fig.\ \ref{fig5.4}). By Proposition \ref{prop:4.1}, estimate \eqref{eq:5.10}, and the properties of rarefaction waves $\beta_1$ again, we have
\begin{eqnarray*}
\begin{split}
I''_{\mathcal{R}}&=\int^{y_{I}+\lambda_{1}(U_{M})s}_{y_{I}+\dot{y}_{\alpha_1}s}
\big|\mathcal{P}_{h}(\varsigma+s, \varsigma)(U^{(\tau)}_{h,\nu}(\varsigma,\cdot))-U^{(\tau)}_{h, \nu}(\varsigma+s, \cdot)\big|dy\\[5pt]
&\leq \big|\lambda_{1}(U_{M})-\dot{y}_{\alpha_1}\big|\Big(|U_{R}-U_{L}|+|U_{M}-U_{L}|\Big)s\\[5pt]
&\leq O(1)\Big(\big|\lambda_{1}(U_{M})-\lambda_{1}(U_{R},\tau^2)\big|+2^{-\nu}\Big)\big(|\alpha_1+|\beta_1|\big)s\\[5pt]
&\leq O(1)\big(\tau^2+2^{-\nu}\big)|\alpha_1|s.
\end{split}
\end{eqnarray*}

Combing all these three cases altogether, we have $I''_{\mathcal{R}}\leq O(1)\big(\tau^{2}+\nu^{-1}+2^{-\nu}\big)|\alpha_{1}|s$.
Hence, with the estimates on $I'_{\mathcal{R}}$, $I''_{\mathcal{R}}$ and $I'''_{\mathcal{R}}$, we can get estimate \eqref{eq:5.5}.

\begin{figure}[ht]
\begin{center}
\begin{tikzpicture}[scale=0.7]
\draw [line width=0.05cm](-4,3.2)--(-4,-4.3);
\draw [line width=0.05cm](0.5,3.2)--(0.5,-4.3);

\draw [thin](-4,0)--(0.5, 2.1);
\draw [thick](-4,0)--(0.5, 0.6);
\draw [thick][blue](-4,0)--(0.5, -1.3);
\draw [thick][blue](-4,0)--(0.5, -1.6);
\draw [thick][blue](-4,0)--(0.5, -1.9);
\draw [thick][blue](-4,0)--(0.5, -2.2);
\draw [thick][red](-4,0)--(0.5, -2.9);
\draw [thin](-4,0)--(0.5, -3.8);

\draw [thin][<->](-3.9,2.8)--(0.4,2.8);

\node at (2.6, 2) {$$};
\node at (-1.8, 3.2){$s$};
\node at (-4.9, 0) {$(\varsigma, y_{I})$};
\node at (-4, -4.9) {$x=\varsigma$};
\node at (0.5, -4.9) {$x=\varsigma+s$};

\node at (2.2, 2.3) {$y=y_{I}+\hat{\lambda}s$};
\node at (1.1, 0.5) {$\beta_{2}$};
\node at (1.1, -2.9) {$\alpha_1$};
\node at (1.1, -1.6) {$\beta_{1}$};
\node at (2.2, -3.8) {$y=y_{I}-\hat{\lambda}s$};

\node at (-1.6, 1.8){$U_{R}$};
\node at (-1.4, -0.2){$U_{M}$};
\node at (-1.6, -3.2){$U_{L}$};
\end{tikzpicture}
\end{center}
\caption{Comparison of Riemann solvers for $\alpha_1\in \mathcal{R}^{(\tau)}$ and $\dot{y}_{\alpha_1}<\lambda_{1}(U_{L})$}\label{fig5.4}
\end{figure}

$\rm(3)$ The front in $U^{(\tau)}_{h,\nu}(\varsigma+s, y)$ is a non-physical wave $\alpha_{NP}$. In this case, as shown in Fig. \ref{fig5.5}, the solution of Riemann problem $\mathcal{P}_h(\varsigma+s, \varsigma)(U^{(\tau)}_{h,\nu}(\varsigma, y))$ satisfies
\begin{eqnarray*}
\alpha_{\mathcal{NP}}=|U_R-U_{L}|, \quad \mbox{and}\quad U_{R}=\Phi(\beta_1,\beta_2; U_L).
\end{eqnarray*}

So, by Proposition \ref{prop:3.1}, we can further obtain
\begin{eqnarray}\label{eq:5.12}
\beta_{j}=O(1)\alpha_{NP},\quad j=1,2.
\end{eqnarray}

Denoted the middle state of $\mathcal{P}_h(\varsigma+s, \varsigma)(U^{(\tau)}_{h,\nu}(\varsigma, y))$ as $U_{M}$.
Then, by Proposition \ref{prop:4.2} and \eqref{eq:5.12}, we get
\begin{eqnarray*}
\begin{split}
&\int^{y_{I}+\hat{\lambda}s}_{y_{I}-\hat{\lambda}s}\big|\mathcal{P}(\varsigma+s, \varsigma)(U^{(\tau)}_{h,\nu}(\varsigma,\cdot))-U^{(\tau)}_{h,\nu}(\varsigma+s,\cdot)\big|dy\\[5pt]
&\ \ \qquad \leq O(1)\Big(|U_{R}-U_{L}|+|U_{M}-U_{L}|\Big)\\[5pt]
&\ \ \qquad \leq O(1)\Big(|\alpha_1|+\sum^{2}_{k=1}|\beta_{k}|\Big)s\\[5pt]
&\ \ \qquad \leq O(1)|\alpha_1|s.
\end{split}
\end{eqnarray*}

\vspace{-3mm}
\begin{figure}[ht]
\begin{center}
\begin{tikzpicture}[scale=0.7]
\draw [line width=0.05cm](-4,4)--(-4,-3.5);
\draw [line width=0.05cm](0.5,4)--(0.5,-3.5);

\draw [thick][red](-4,0)--(0.5, 2.5);
\draw [thick](-4,0)--(0.5, 0.6);
\draw [thick](-4,0)--(0.5, -1.3);
\draw [thin](-4,0)--(0.5, -2.9);

\draw [thin][<->](-3.9,3.5)--(0.4,3.5);

\node at (2.6, 2) {$$};
\node at (-1.8, 3.9){$s$};
\node at (-4.9, 0) {$(\varsigma, y_{I})$};
\node at (-4, -4.5) {$x=\varsigma$};
\node at (0, -4.5) {$x=\varsigma+s$};

\node at (1.2, 2.4) {$\alpha_{\mathcal{NP}}$};
\node at (1.2, 0.5) {$\beta_{2}$};
\node at (1.2, -1.6) {$\beta_{1}$};
\node at (2.2, -2.9) {$y=y_{I}-\hat{\lambda}s$};

\node at (-1.6, 2.4){$U_{R}$};
\node at (-1.4, -0.2){$U_{M}$};
\node at (-1.6, -2.8){$U_{L}$};

\end{tikzpicture}
\end{center}
\caption{Comparison of the Riemann solvers for $\alpha_{\mathcal{NP}}\in \mathcal{NP}^{(\tau)}$}\label{fig5.5}
\end{figure}
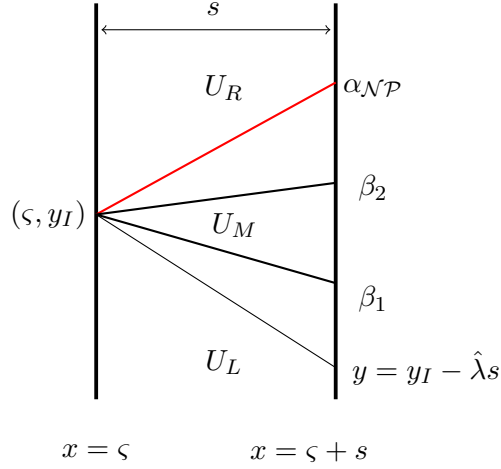

It completes the proof of this proposition.
\end{proof}

As a corollary, we consider the case that $U^{(\tau)}_{h,\nu}$ has more than one discontinuities.
\begin{corollary}\label{coro:5.1}
Let $U^{(\tau)}_{h,\nu}(\varsigma,y)$ be a piecewise constant function defined in \eqref{eq:5.1} with $U_{L}, U_M, \hat{U}_{R}, U_{R}\in \mathcal{O}_{\min\{C_1\epsilon^*_0, C_3\epsilon^*_1\}}(\underline{U})$
for $\tau\in(0,\tau^{*}_{0})$. Let $\mathcal{P}_{h}$ be the uniformly Lipschtiz continuous map given by Proposition \ref{prop:4.2}.
Let $\hat{\lambda}$ be a fixed constant satisfying $\hat{\lambda}>\max\{\lambda_{j}(U^{(\tau)},\tau^{2}), \lambda_{j}(U)\}$
for all $U^{(\tau)}, U\in \mathcal{O}_{\min\{C_1\epsilon^*_0, C_3\epsilon^*_1\}}(\underline{U})$, $\tau\in(0,\tau^*_0)$ and $j=1,2$.
Denotes
\begin{eqnarray}\label{eq:5.13}
U^{(\tau)}_{h,\nu}(\varsigma+s, y)=
\left\{
\begin{array}{llll}
U_{R}, \quad & y>y_{I}+\hat{\lambda}s,\\[5pt]
\hat{U}_{R}, \quad & y_{I}+\dot{y}_{\alpha_2}s<y<y_{I}+\hat{\lambda}s,\\[5pt]
U_{M},\quad & y_{I}+\dot{y}_{\alpha_1}s<y<y_{I}+\dot{y}_{\alpha_1}s,\\[5pt]
U_L, \quad & y<y_I+\dot{y}_{\alpha_1}s,
\end{array}
\right.
\end{eqnarray}
where $|\dot{y}_{\alpha_k}|\leq \hat{\lambda},\, (k=1,2)$. Then, for sufficiently small $s>0$,
if $U_{L}$, $U_{M}$ are connected by $1$th-shock wave $\alpha_1\in \mathcal{S}^{(\tau)}_1$ (or $1$th-rarefaction front $\alpha_1\in \mathcal{R}^{(\tau)}_1$),
and $U_{M}$, $U_{R}$ are connected by $2$th-shock wave $\alpha_2\in \mathcal{S}^{(\tau)}_2$ (or $2$th-rarefaction front $\alpha_2\in \mathcal{R}^{(\tau)}_2$),
and $\hat{U}_{R}$, $U_{R}$ are connected by non-physical wave front $\alpha_{\mathcal{NP}}\in \mathcal{NP}^{(\tau)}$ for $|\dot{y}_{\alpha_1}-\dot{\mathcal{S}}_{1}(\alpha_1,\tau^2)|<2^{-\nu}$ (or $|\dot{y}_{\alpha_1}-\lambda_{1}(U_{M}, \tau^2)|<2^{-\nu}$),
and $|\dot{y}_{\alpha_2}-\dot{\mathcal{S}}_{2}(\alpha_2,\tau^2)|<2^{-\nu}$ (or $|\dot{y}_{\alpha_2}-\lambda_{2}(\hat{U}_{R}, \tau^2)|<2^{-\nu}$),
then
\begin{eqnarray}\label{eq:5.14}
\begin{split}
&\int^{y_{I}+\hat{\lambda}s}_{y_{I}-\hat{\lambda}s}\big|\mathcal{P}_{h}(\varsigma+s,\varsigma)(U^{(\tau)}_{h,\nu}(\varsigma))
-U^{(\tau)}_{h,\nu}(\varsigma+s)\big|dy\\[5pt]
&\quad\quad\quad\quad  \leq C_{5}\Big(\big(\tau^{2}+\nu^{-1}+2^{-\nu}\big)\big(|\alpha_{1}|+|\alpha_2|\big)+\alpha_{\mathcal{NP}}\Big)s,
\end{split}
\end{eqnarray}
where $\dot{\mathcal{S}}_{1}(\alpha_1, \tau^2)$ (or $\lambda_{1}(U_{M}, \tau^2)$) and $\dot{\mathcal{S}}_{2}(\alpha_2, \tau^2)$ (or $\lambda_{2}(\hat{U}_{R}, \tau^2)$) are the speeds of the $1$th-shock wave (or $1$th-rarefaction front) and the $2$th-shock wave (or $2$th-rarefaction front), and the $\hat{\lambda}$
is the speed of the non-physical wave. Here, the constants $C_1>0$, $C_3>0$ are given by Proposition \ref{prop:4.1} and \ref{prop:4.3}, respectively,
and constant $C_{5}>0$ depends only on $\underline{U}$ and $a_{\infty}$ but not on $\tau, h, \nu$.
\end{corollary}

Now, we turn to the comparison between the Riemann solutions with the boundary. We first consider 
the location near the approximate boundary $\Gamma_{h}$ but away from the corner point. 
Assume the approximate $U^{(\tau)}_{h,\nu}$
of the initial-boundary value problem \eqref{eq:1.16}-\eqref{eq:1.18} has only one front $y_{\alpha_1}$ with strength $\alpha_1$, coming from corner point $A_{k}=(x_k, b_k)$ with $\omega_{k}=\theta_{k+1}-\theta_{k}$ 
for $k\geq 0$.
Let $(\varsigma, y_B)$ be a point on the front $y_{\alpha_1}$ away from $A_k$. 
Set
\begin{eqnarray}\label{eq:5.15}
U_{L}=(\rho_{L}, v_{L})^{\top}\doteq U^{(\tau)}_{h,\nu}(\varsigma, y_{B}-),\quad U_{B}=(\rho_{B}, v_{B})^{\top}\doteq U^{(\tau)}_{h,\nu}(\varsigma, {\color{red}y_{B}}+),
\end{eqnarray}
with
\begin{eqnarray}\label{eq:5.16}
\big(1+\tau^2 u_{L}, v_{L}\big)\cdot \mathbf{n}_{k}=0,\quad \big(1+\tau^2 u_{B}, v_{B}\big)\cdot\mathbf{n}_{k+1}=0,
\end{eqnarray}
where $\mathbf{n}_{k}=(\sin(\theta_{k}),-\cos(\theta_{k})), (k\geq 0)$ and $u_{j}=u_{j}(\rho_{j}, v_j,\tau^2), (j=L, B)$, are given by relation \eqref{eq:1.12}.

Define
\begin{eqnarray}\label{eq:5.17}
U^{(\tau),B}_{h,\nu}(\varsigma,y)=
\left\{
\begin{array}{llll}
U_{B}, \quad & y_{B}<y<b_{h}(\varsigma),\\[5pt]
U_L, \quad & y<y_{B}.
\end{array}
\right.
\end{eqnarray}

Then, we have
\begin{lemma}\label{lem:5.2}
Let $U^{(\tau), B}_{h,\nu}(\varsigma,y)$ be a piecewise constant function defined in \eqref{eq:5.17} with $U_{L}, U_{B}\in \mathcal{O}_{\min\{C_{1}\epsilon^*_0, C_3\epsilon^*_1\}}(\underline{U})$ satisfying \eqref{eq:5.16} for $\tau\in(0,\tau^{*}_{0})$.
Let $\mathcal{P}_{h}$ be a uniformly Lipschtiz continuous map obtained by Proposition \ref{prop:4.2}.
Let $\hat{\lambda}$ be a fixed constant with $\hat{\lambda}>\max\{\lambda_{j}(U^{(\tau)},\tau^{2}), \lambda_{j}(U)\}$
for all $U^{(\tau)}, U\in \mathcal{O}_{\min\{C_1\epsilon^*_0, C_3\epsilon^*_1\}}(\underline{U})$, $\tau\in(0,\tau^*_0)$ and $j=1,2$.
Denotes
\begin{eqnarray}\label{eq:5.18}
U^{(\tau),B}_{h,\nu}(\varsigma+s, y)=
\left\{
\begin{array}{llll}
U_{B}, \quad & y_{B}+\dot{y}_{\alpha_1}s<y<b_{h}(\varsigma+s),\\[5pt]
U_L, \quad & y<y_{B}+\dot{y}_{\alpha_1}s,
\end{array}
\right.
\end{eqnarray}
where $|\dot{y}_{\alpha_1}|\leq \hat{\lambda}$. Then, for sufficiently small $s>0$,

\rm (1b)\ if $\omega_{k}=\theta_{k+1}-\theta_{k}<0$, that is, $U_{L}$ and $U_{B}$ are connected by $1$-shock wave $\alpha_1\in \mathcal{S}^{(\tau)}_1$ with $U_{B}=\Phi_{1}(\alpha_1; U_{L},\tau^2)$ for $\alpha_1<0$ and $|\dot{y}_{\alpha_1}-\dot{\mathcal{S}}_{1}(\alpha_1, \tau^2)|<2^{-\nu}$,
then
\begin{eqnarray}\label{eq:5.19}
\begin{split} \int^{y_B+\hat{\lambda}s}_{y_B-\hat{\lambda}s}\big|\mathcal{P}_{h}(\varsigma+s,\varsigma)(U^{(\tau),B}_{h,\nu}(\varsigma,\cdot))
-U^{(\tau),B}_{h,\nu}(\varsigma+s,\cdot)\big|dy
\leq C_{5}|\omega_{k}|\big(\tau^{2}+2^{-\nu}\big),
\end{split}
\end{eqnarray}
where $\dot{\mathcal{S}}_{1}(\alpha_1, \tau^2)$ is the speed of the $1$th-shock wave,

\rm (2b)\ if $\omega_{k}=\theta_{k+1}-\theta_{k}>0$, that is, $U_{L}$ and $U_{B}$ are connected by $1$th-rarefaction front $\alpha_1\in \mathcal{R}^{(\tau)}_1$ with $U_{B}=\Phi_{1}(\alpha_1; U_{L},\tau^2)$ for $\alpha_1>0$ and $|\dot{y}_{\alpha_1}-\lambda_{1}({\color {red} U_{B}}, \tau^2)|<2^{-\nu}$, then
\begin{eqnarray}\label{eq:5.20}
\begin{split} \int^{y_B+\hat{\lambda}s}_{y_B-\hat{\lambda}s}\big|\mathcal{P}_{h}(\varsigma+s,\varsigma)(U^{(\tau),B}_{h,\nu}(\varsigma))
-U^{(\tau),B}_{h,\nu}(\varsigma+s)\big|dy
\leq C_{5}|\omega_{k}|\big(\tau^{2}+\nu^{-1}+2^{-\nu}\big)s,
\end{split}
\end{eqnarray}
where $\lambda_{1}(U_{B}, \tau^2)$ is the speed of the $1$th-rarefaction front.
Here, constant $C_{5}>0$ depends only on $\underline{U}$ and $a_{\infty}$, but not on $\tau, h, \nu$.
\end{lemma}

\begin{proof}
For \rm (1b), according to the construction of $\mathcal{P}_{h}$, with the help of the Proposition \ref{prop:3.1}, the solution $\mathcal{P}_{h}(\varsigma+s,\varsigma)(U^{(\tau),B}_{h,\nu}(\varsigma))$ can be solved from the equation
\begin{eqnarray*}
\Phi(\beta_1,\beta_2; U_{L})=\Phi_1(\alpha_1; U_L, \tau^2),
\end{eqnarray*}
which admits a unique $C^2$ solutions $\beta_1, \beta_2$ with respect to $\alpha_1, \tau^2$, satisfying estimates \eqref{eq:3.7}.
Then, we can follow the argument if the proof of Lemma \ref{lem:5.1} to have
\begin{eqnarray*}
\begin{split} \int^{y_B+\hat{\lambda}s}_{y_B-\hat{\lambda}s}\big|\mathcal{P}_{h}(\varsigma+s,\varsigma)(U^{(\tau),B}_{h,\nu}(\varsigma,\cdot))
-U^{(\tau),B}_{h,\nu}(\varsigma+s,\cdot)\big|dy
\leq O(1)|\alpha_{1}|\big(\tau^{2}+2^{-\nu}\big),
\end{split}
\end{eqnarray*}
where constant $O(1)$ depends only on $\underline{U}$ and $a_{\infty}$. Moreover, by \eqref{eq:5.16}, it follows from Lemma \ref{lem:3.2} that $|\alpha_1|=O(1)|\omega_{k}|$
for a constant $O(1)$ depending only on $\underline{U}$ and $a_{\infty}$.
So combing them together, we can get \eqref{eq:5.19} for case of shock front.

In the same way, we can also show the estimate \eqref{eq:5.20} for \rm (2b).
\end{proof}

Next, let us consider the comparison of
two approximate solutions $U^{(\tau)}_{h,\nu}$ to problem \eqref{eq:1.16}-\eqref{eq:1.18} and $U_{h,\nu}$ to problem \eqref{eq:1.20}-\eqref{eq:1.21} and \eqref{eq:1.11}, respectively, with the jump at the corner point $A_{k}=(x_k, b_k)$ 
for $k\geq 0$.
Denote
\begin{eqnarray}\label{eq:5.21}
U_{L}=(\rho_L, v_{L})^{\top}\doteq U^{(\tau)}_{h,\nu}(x_k, b_{k}-),\quad  U^{(\tau)}_{B}=(\rho^{(\tau)}_{B}, v^{(\tau)}_{B})^{\top}\doteq U^{(\tau)}_{h,\nu}(x_{k}, b_{k}),
\end{eqnarray}
and
\begin{eqnarray}\label{eq:5.22}
U_{B}=(\rho_{B}, v_{B})^{\top}\doteq U_{h,\nu}(x_{k}, b_{k}),
\end{eqnarray}
where
\begin{eqnarray}\label{eq:5.23}
(1+\tau^2 u_L, v_{L})\cdot \mathbf{n}_{k}=0, \quad (1+\tau^2 u^{(\tau)}_B, v^{(\tau)}_{B})\cdot \mathbf{n}_{k+1}=0, \quad
v_{B}=\tan\theta_{k+1},
\end{eqnarray}
$u_L=u_{L}(\rho_L, v_{L}, \tau^2)$ and $u^{(\tau)}_{B}=u^{(\tau)}_{B}(\rho^{(\tau)}_{B}, v^{(\tau)}_{B}, \tau^2)$ are given by \eqref{eq:1.12}, and $\mathbf{n}_{j}=(\sin\theta_j, -\cos\theta_j)$ for $j=k,k+1$.

Define
\begin{eqnarray}\label{eq:5.24}
U^{(\tau),B}_{h,\nu}(x_k,y)\doteq
\left\{
\begin{array}{llll}
U_{B}, \quad & y=b_{k},\\[5pt]
U_L, \quad & y<b_{k}.
\end{array}
\right.
\end{eqnarray}
Then, we have
\begin{lemma}\label{lem:5.3}
Let approximate solution
$U^{(\tau),B}_{h,\nu}(x_{k},y)$ be a piecewise constant function defined by \eqref{eq:5.24} satisfying \eqref{eq:5.23} with $U_{L}, U_{B}\in \mathcal{O}_{\min\{C_{1}\epsilon^*_0, C_3\epsilon^*_1\}}(\underline{U})$ for $\tau\in(0,\tau^{*}_{0})$.
Let $\mathcal{P}_{h}$ be a uniformly Lipschitz continuous map given by Proposition \ref{prop:4.2}, and
$\hat{\lambda}$ be a fixed constant satisfying $\hat{\lambda}>\max\{\lambda_{j}(U^{(\tau)},\tau^{2}), \lambda_{j}(U)\}$
for all $U^{(\tau)}, U\in \mathcal{O}_{\min\{C_1\epsilon^*_0, C_3\epsilon^*_1\}}(\underline{U})$, $\tau\in(0,\tau^*_0)$ and $j=1,2$.
Denote
\begin{eqnarray}\label{eq:5.25}
U^{(\tau),B}_{h,\nu}(x_k+s, y)=
\left\{
\begin{array}{llll}
U^{(\tau)}_{B}, \quad & b_{k}+\dot{y}_{\alpha_1}s<y<b_{k}+\tan(\theta_{k+1})s,\\[5pt]
U_L, \quad & y<b_k+\dot{y}_{\alpha_1}s,
\end{array}
\right.
\end{eqnarray}
where $|\dot{y}_{\alpha_1}|\leq \hat{\lambda}$. Then, for sufficiently small $s>0$,

\rm (1c)\ if $\omega_{k}=\theta_{k+1}-\theta_{k}<0$, that is, $U_{L}$ and $U^{(\tau)}_{B}$ are connected by $1$th-shock wave $\alpha_1\in \mathcal{S}^{(\tau)}_1$, with $U^{(\tau)}_{B}=\Phi_{1}(\alpha_1; U_L,\tau^2)$ for $\alpha_1<0$ and $|\dot{y}_{\alpha_1}-\dot{\mathcal{S}}_{1}(\alpha_1, \tau^2)|<2^{-\nu}$,
then
\begin{eqnarray}\label{eq:5.26}
\begin{split}
\int^{b_{k}+\tan(\theta_{k+1})s}_{b_k-\hat{\lambda}s}&\big|\mathcal{P}_{h}(x_k+s,x_k)(U^{(\tau),B}_{h,\nu}(x_k,y))
-U^{(\tau),B}_{h,\nu}(x_k+s,y)\big|dy\\[5pt]
&\leq C_{6}\Big( \big(1+|\omega_{k}|\big)\tau^2+\big(|\omega_{k}|+\tau^{2}\big)2^{-\nu}\Big)s,
\end{split}
\end{eqnarray}
where $\dot{\mathcal{S}}_{1}(\alpha_1, \tau^2)$ is the speed of the $1$-shock wave;

\rm (2c)\ if $\omega_{k}=\theta_{k+1}-\theta_{k}>0$, that is, $U_{L}$ and $U^{(\tau)}_{B}$ are connected by $1$th-rarefaction front $\alpha_1\in \mathcal{R}^{(\tau)}_1$ with $U^{(\tau)}_{B}=\Phi_{1}(\alpha_1; U_L,\tau^2)$ for $\alpha_1>0$ and $|\dot{y}_{\alpha_1}-\lambda_{1}(U^{(\tau)}_B, \tau^2)|<2^{-\nu}$, then
\begin{eqnarray}\label{eq:5.27}
\begin{split}
\int^{b_{k}+\tan(\theta_{k+1})s}_{b_k-\hat{\lambda}s}&\big|\mathcal{P}_{h}(x_k+s,x_k)(U^{(\tau),B}_{h,\nu}(x_k,y))
-U^{(\tau),B}_{h,\nu}(x_k+s,y)\big|dy\\[5pt]
&\leq C_{6}\Big( \big(1+|\omega_{k}|\big)\tau^2+\big(|\omega_{k}|+\tau^{2}\big)(\nu^{-1}+2^{-\nu})\Big)s,
\end{split}
\end{eqnarray}
where $\lambda_{1}(U^{(\tau)}_B, \tau^2)$ is the speed of the $1$-rarefaction front.
Here, constant $C_{6}>0$ depends only on $\underline{U}$ and $a_{\infty}$, but does not on $\tau, h, \nu$.
\end{lemma}

\begin{proof}
From the construction, 
$\mathcal{P}_{h}(\varsigma+s,\varsigma)(U^{(\tau),B}_{h,\nu}(x_k))$ near boundary $\Gamma_{h,k}$ satisfies 
\begin{eqnarray}\label{eq:5.28}
(1+\tau^2u^{(\tau)}_{B})\Phi^{(2)}_{1}(\beta_1; U_{L})=\Phi^{(2)}_{1}(\alpha_1; U_{L},\tau^2).
\end{eqnarray}

According to Proposition \ref{prop:3.4}, equation \eqref{eq:5.28} admits a unique solution $\beta_{1}(\alpha_1,\tau^2)$ near $\alpha_1=\tau=0$ and $U_{L}=\underline{U}$ with
\begin{eqnarray*}
\beta_1=\alpha_1+O(1)\big(1+|\alpha_1|\big)\tau^2.
\end{eqnarray*}

By Lemma \ref{lem:3.2}, we have $\alpha_1=\tilde{C}\omega_k$
with $\tilde{C}>0$ depending only on $\underline{U}$ and $a_{\infty}$. So, $\alpha_1<0$ (or $\alpha_1>0$) which is a shock front (or rarefaction front) if $\omega_{k}<0$ (or $\omega_k>0$), and it also implies that $\beta_1<0$ (or $\beta_1>0$) for $\tau>0$ sufficiently small.

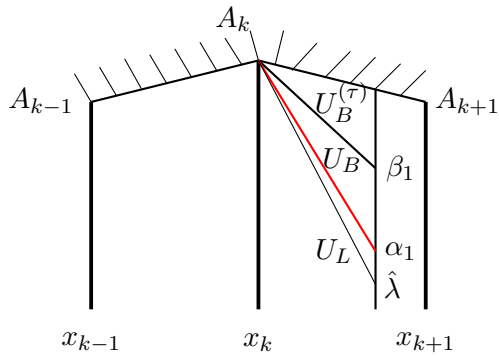
\begin{figure}[ht]
\begin{center}
\begin{tikzpicture}[scale=1.1]
\draw [thick](-5.0,1.5)--(-3,1);
\draw [thick](-7.0,1)--(-5,1.5);

\draw [line width=0.05cm](-7.0,1)--(-7.0,-1.5);
\draw [line width=0.05cm](-5.0,1.5)--(-5.0,-1.5);
\draw [line width=0.05cm](-3,1)--(-3,-1.5);
\draw[thick](-3.6,1.16)--(-3.6,-1.5);

\draw [thin](-7.0,1.00)--(-7.2, 1.34);
\draw [thin](-6.7,1.08)--(-6.9, 1.42);
\draw [thin](-6.4,1.16)--(-6.6, 1.52);
\draw [thin](-6.1,1.23)--(-6.3, 1.57);
\draw [thin](-5.8,1.30)--(-6.0, 1.64);
\draw [thin](-5.5,1.38)--(-5.7, 1.72);
\draw [thin](-5.2,1.45)--(-5.4, 1.77);
\draw [thin](-5.0,1.50)--(-5.1, 1.86);
\draw [thin](-4.8,1.45)--(-4.7, 1.86);
\draw [thin](-4.6,1.38)--(-4.3, 1.71);
\draw [thin](-4.2,1.30)--(-3.9, 1.61);
\draw [thin](-3.9, 1.23)--(-3.6,1.54);
\draw [thin](-3.6, 1.16)--(-3.3,1.46);
\draw [thin](-3.3, 1.08)--(-3.0,1.38);

\draw [thick](-5.0,1.5)--(-3.6, 0.2);
\draw [thick][red](-5.0,1.5)--(-3.6,-0.8);
\draw [thin](-5,1.5)--(-3.6,-1.2);

\node at (-7.6, 1.0) {$A_{k-1}$};
\node at (-5.3, 2.0) {$A_{k}$};
\node at (-2.5, 1.0) {$A_{k+1}$};
\node at (-3.3, 0.2) {$\beta_{1}$};
\node at (-3.3, -0.8) {$\alpha_{1}$};
\node at (-3.4, -1.2) {$\hat{\lambda}$};
\node at (-4.0, 1.0) {$U^{(\tau)}_{B}$};
\node at (-4.0, 0.3) {$U_{B}$};
\node at (-4.1, -0.8) {$U_{L}$};

\node at (-7.0, -1.9) {$x_{k-1}$};
\node at (-5.0, -1.9) {$x_{k}$};
\node at (-3.0, -1.9) {$x_{k+1}$};
\end{tikzpicture}
\caption{Case $\omega_{k}<0$}\label{fig5.6}
\end{center}
\end{figure}

$\bullet$\ Case \rm(1c): $\omega_{k}<0$. As shown above, we know that both $\alpha_1$ and $\beta_1$ are shock fronts (see Fig. \ref{fig5.6}).
Denoted by $\dot{\mathcal{S}}_{1}(\beta_1)$ the speed of the $1$-shock $\beta_{1}$. Then, we rewrite the left hand of \eqref{eq:5.26} as
\begin{eqnarray*}
\begin{split}
&\ \ \ \int^{b_{k}+\tan(\theta_{k+1})s}_{b_{k}-\hat{\lambda}h}\big|\mathcal{P}_{h}(x_k+s, x_{k})(U^{(\tau),B}_{h,\nu}(x_{k},y))-U^{(\tau),B}_{h,\nu}
 (x_{k}+h,y)\big|dy\\[5pt]
&=\int^{\min\big\{b_{k}+\dot{y}_{\alpha_1}s, b_{k}+\dot{\mathcal{S}}_{1}(\beta_1)s\big\}}_{b_{k}-\hat{\lambda}s}
\big|\mathcal{P}_{h}(x_k+s, x_{k})(U^{(\tau),B}_{h,\nu}(x_{k},y))-U^{(\tau),B}_{h,\nu}(x_{k}+s,y)\big|dy\\[5pt]
&\ \ \ +\int^{\max\big\{b_{k}+\dot{y}_{\alpha_1}s, b_{k}+\dot{\mathcal{S}}_{1}(\beta_1)s\big\}}_{\min\big\{b_{k}+\dot{y}_{\alpha_1}s, b_{k}+\dot{\mathcal{S}}_{1}(\beta_1)s\big\}}
\big|\mathcal{P}_{h}(x_k+s, x_{k})(U^{(\tau),B}_{h,\nu}(x_{k},y))-U^{(\tau),B}_{h,\nu}(x_{k}+s,y)\big|dy\\[5pt]
&\ \ \ +\int^{b_{k}+\tan(\theta_{k+1})s}_{\max\big\{b_{k}+\dot{y}_{\alpha_1}s, b_{k}+\dot{\mathcal{S}}_{1}(\beta_1)s\big\}}
\big|\mathcal{P}_{h}(x_k+s, x_{k})(U^{(\tau),B}_{h,\nu}(x_{k},y))-U^{(\tau),B}_{h,\nu}
(x_{k}+s,y)\big|dy\\[5pt]
&\doteq B'_{\mathcal{S}_1}+B''_{\mathcal{S}_1}+B'''_{\mathcal{S}_1}.
\end{split}
\end{eqnarray*}

Obviously, $B'_{\mathcal{S}_1}=0$. For $B'''_{\mathcal{S}_1}$, we have
\begin{eqnarray*}
\begin{split}
B'''_{\mathcal{S}_1}\leq |U^{(\tau)}_{B}-U_{B}|s=\big|\Phi_{1}(\alpha_1; U_{L},\tau^{2})-\Phi(\beta_1; U_{L})\big|s.
\end{split}
\end{eqnarray*}

Let $\mathcal{F}_{B}(\alpha_1,\tau^{2})=\Phi_{1}(\alpha_1; U_{L},\tau^{2})-\Phi(\beta_1; U_{L})$.
Notice that $\mathcal{F}_{B}(\alpha_1,0)=\mathcal{F}_{B}(0,0)=0$ and $\mathcal{F}_{B}(0,\tau^{2})=O(1)\tau^2$. Then
$$
\mathcal{F}_{B}(\alpha_1,\tau^{2})=O(1)(1+|\alpha_1|)\tau^{2}=O(1)(1+|\omega_{k})|\tau^{2}.$$

Therefore,
$$
B'''_{\mathcal{S}_1}\leq O(1)(1+|\omega_{k})\tau^{2}s.
$$

Finally, let us consider $B''_{\mathcal{S}_1}$. Via the argument as done in the proof of Lemma \ref{lem:5.1},
\begin{eqnarray}\label{eq:5.29}
\begin{split}
\dot{\mathcal{S}}_{1}(\beta_1)-\dot{\mathcal{S}}_{1}(\alpha_1,\tau^2)=O(1)\big(1+|\alpha_1|\big)\tau^{2}
=O(1)\big(1+|\omega_{k}|\big)\tau^{2}.
\end{split}
\end{eqnarray}

If $\dot{y}_{\alpha_1}>\dot{\mathcal{S}}_{1}(\beta_1)$, by \eqref{eq:5.29} and the estimates on $\beta_1$ and $\alpha_1$, we get
\begin{eqnarray*}
\begin{split}
B''_{\mathcal{S}_1}&=\int^{b_{k}+\dot{y}_{\alpha_1}s}_{b_{k}+\dot{\mathcal{S}}_{1}(\beta_1)s}
\big|\mathcal{P}_{h}(x_k+s, x_{k})(U^{(\tau),B}_{h,\nu}(x_{k},y))-U^{(\tau),B}_{h,\nu}(x_{k}+s,y)\big|dy\\[5pt]
&\leq|\dot{y}_{\alpha_1}-\dot{\mathcal{S}}_{1}(\beta_1)|\big|U_B-U_{L}\big|s\\[5pt]
&\leq \Big(\big|\dot{\mathcal{S}}_{1}(\beta_1)-\dot{\mathcal{S}}_{1}(\alpha_1,\tau^2)\big|
+2^{-\nu}\Big)\big|\Phi(\beta_1; U_{L})-U_{L}\big|s\\[5pt]
&\leq O(1)\Big((1+|\omega_{k}|)\tau^2+O(1)2^{-\nu})\Big)|\beta_1|s\\[5pt]
&\leq O(1)\Big(|\omega_{k}|\tau^2+2^{-\nu}(|\omega_{k}|+\tau^2)\Big)s.
\end{split}
\end{eqnarray*}

If $\dot{y}_{\alpha_1}<\dot{\mathcal{S}}_{1}(\beta_1)$, by \eqref{eq:5.29} and the estimate on $\alpha_1$, we have
\begin{eqnarray*}
\begin{split}
B''_{\mathcal{S}_1}&=\int^{b_{k}+\dot{\mathcal{S}}_{1}(\beta_1)s}_{b_{k}+\dot{y}_{\alpha_1}s}
\big|\mathcal{P}_{h}(x_k+s, x_{k})(U^{(\tau),B}_{h,\nu}(x_{k},y))-U^{(\tau),B}_{h,\nu}(x_{k}+s,y)\big|dy\\[5pt]
&\leq|\dot{y}_{\alpha_1}-\dot{\mathcal{S}}_{1}(\beta_1)|\big|U^{(\tau)}_{B}-U_{L}\big|s\\[5pt]
&\leq \Big(\big|\dot{\mathcal{S}}_{1}(\beta_1)-\dot{\mathcal{S}}_{1}(\alpha_1,\tau^2)\big|
+2^{-\nu}\Big)\big|\Phi(\alpha_1; U_{L},\tau^2)-U_{L}\big|s\\[5pt]
&\leq O(1)\big(|\omega_{k}|\tau^2+2^{-\nu}|\omega_k|\big)s.
\end{split}
\end{eqnarray*}

So we obtain $B''_{\mathcal{S}_1}\leq O(1)\Big(|\omega_{k}|\tau^2+2^{-\nu}(|\omega_{k}|+\tau^2)\Big)s$.
Hence, we can obtain estimate \eqref{eq:5.26}
by choosing a constant $C_{6}>0$ depending only on $\underline{U}$ and $a_{\infty}$. This completes the proof for this case. 

$\bullet$\ Case \rm(2c): $\omega_{k}>0$. In this case, it follows from Lemma \ref{lem:3.2} and Proposition \ref{prop:3.4} that
$\alpha_1>0, \beta_{1}>0$ (see Figs. \ref{fig5.7}-\ref{fig5.9} below). So both are rarefaction fronts.
Then 
\begin{eqnarray*}
\begin{split}
&\ \ \ \int^{b_{k}+\tan(\theta_{k+1})s}_{b_{k}-\hat{\lambda}s}\big|\mathcal{P}_{h}(x_k+s, x_{k})(U^{(\tau),B}_{h,\nu}(x_{k},y))-U^{(\tau),B}_{h,\nu}(x_{k}+s,y)\big|dy\\[5pt]
&=\int^{\min\big\{b_{k}+\dot{y}_{\alpha_1}s, b_{k}+\lambda_{1}(U_{L})s\big\}}_{b_{k}-\hat{\lambda}s}
\big|\mathcal{P}_{h}(x_k+s, x_{k})(U^{(\tau),B}_{h,\nu}(x_{k},y))-U^{(\tau),B}_{h,\nu}(x_{k}+s,y)\big|dy\\[5pt]
&\ \ \ +\int^{\max\big\{b_{k}+\dot{y}_{\alpha_1}s, b_{k}+\lambda_{1}(U_{B})s\big\}}_{\min\big\{b_{k}+\dot{y}_{\alpha_1}s, b_{k}+\lambda_{1}(U_{L})s\big\}}
\big|\mathcal{P}_{h}(x_k+s, x_{k})(U^{(\tau),B}_{h,\nu}(x_{k},y))-U^{(\tau),B}_{h,\nu}(x_{k}+s,y)\big|dy.
\end{split}
\end{eqnarray*}

Similarly as the proof done for case (1c), we can show that
\begin{eqnarray*}
\begin{split}
B'_{\mathcal{R}_1}+B'''_{\mathcal{R}_1}\leq O(1)\big|\Phi_{1}(\beta_1;U_{L})-\Phi_{1}(\alpha_1;U_{L},\tau^2)\big|=O(1)(1+|\omega_{k}|)\tau^2.
\end{split}
\end{eqnarray*}

Finally, for $B''_{\mathcal{R}_1}$, we have
\begin{eqnarray}\label{eq:5.30}
\begin{split}
\lambda_{1}(U_{B})-\lambda_{1}(U^{(\tau)}_{B}, \tau^{2})
&=\lambda_{1}(\Phi_{1}(\beta_1; U_{L}))-\lambda_{1}(\Phi_{1}(\alpha_1; U_{L},\tau^2), \tau^{2})\\[5pt]
&=\lambda_{1}(U_{L})-\lambda_{1}(U_{L},\tau^2)+O(1)(\beta_1-\alpha_1)\\[5pt]
&=O(1)(1+|\omega_{k}|)\tau^2,
\end{split}
\end{eqnarray}
and
\begin{eqnarray}\label{eq:5.31}
\begin{split}
\lambda_{1}(U^{(\tau)}_{B}, \tau^{2})-\lambda_{1}(U_{L})
&=\lambda_{1}(\Phi_{1}(\alpha_1; U_{L},\tau^2), \tau^{2})-\lambda_{1}(U_{L})
=O(1)\big(|\omega_{k}|+\tau^{2}\big).
\end{split}
\end{eqnarray}

\vspace{-3mm}
\begin{figure}[ht]
\begin{center}
\begin{tikzpicture}[scale=1.1]
\draw [thick](-5.0,1.5)--(-3,1)--(-1,1.8);

\draw [line width=0.05cm](-5.0,1.5)--(-5.0,-1.5);
\draw [line width=0.05cm](-3,1)--(-3,-1.5);
\draw [line width=0.05cm](-1.0,1.8)--(-1.0,-1.5);
\draw [thick](-1.8,1.5)--(-1.8,-1.5);

\draw [thin](-4.8,1.45)--(-4.5, 1.75);
\draw [thin](-4.5,1.38)--(-4.2, 1.68);
\draw [thin](-4.2,1.30)--(-3.9, 1.60);
\draw [thin](-3.9, 1.23)--(-3.6,1.53);
\draw [thin](-3.6, 1.16)--(-3.3,1.46);
\draw [thin](-3.3, 1.08)--(-3.0,1.38);
\draw [thin](-3.0, 1.0)--(-2.8,1.4);
\draw [thin](-2.7,1.12)--(-2.5, 1.52);
\draw [thin](-2.4,1.23)--(-2.2, 1.63);
\draw [thin](-2.1,1.35)--(-1.9, 1.75);
\draw [thin](-1.8,1.47)--(-1.6, 1.87);
\draw [thin](-1.5, 1.59)--(-1.3,1.99);
\draw [thin](-1.2,1.71)--(-1.0, 2.11);

\draw [thick](-3,1)--(-1.8,0.8);
\draw [thick](-3,1)--(-1.8,0.5);
\draw [thick][blue](-3,1)--(-1.8,-0.6);
\draw [thin](-3,1)--(-1.8,-1.0);

\node at (-5.4, 1.6) {$A_{k-1}$};
\node at (-3.0, 1.5) {$A_{k}$};
\node at (-0.6, 1.8) {$A_{k+1}$};
\node at (-1.5, 0.6) {$\beta_{1}$};
\node at (-1.5, -0.6) {$\alpha_{1}$};
\node at (-2.1, 1.1) {$U_{B}$};
\node at (-2.1, 0.3) {$U^{(\tau)}_{B}$};
\node at (-2.4, -0.4) {$U_{L}$};
\node at (1, 2) {$$};

\node at (-5.0, -1.9) {$x_{k-1}$};
\node at (-3.0, -1.9) {$x_{k}$};
\node at (-0.9, -1.9) {$x_{k+1}$};
\end{tikzpicture}
\caption{Case $\omega_{k}>0$ and $\dot{y}_{\alpha_1}<\lambda_{1}(U_{L})$}\label{fig5.7}
\end{center}
\end{figure}

If $\dot{y}_{\alpha_1}<\lambda_{1}(U_{L})$, then
\begin{eqnarray*}
\begin{split}
B''_{\mathcal{R}_1}&=\int^{ b_{k}+\lambda_{1}(U_{B})s}_{b_{k}+\dot{y}_{\alpha_1}s}
\big|\mathcal{P}_{h}(x_k+s, x_{k})(U^{(\tau),B}_{h,\nu}(x_{k},y))-U^{(\tau),B}_{h,\nu}(x_{k}+s,y)\big|dy\\[5pt]
&\leq \big|\lambda_{1}(U_{B})-\dot{y}_{\alpha_1}\big| \big|U_{B}-U^{(\tau)}_{B}\big|s\\[5pt]
&\leq \Big(\big|\lambda_{1}(U_{B})-\lambda_{1}(U^{(\tau)}_{B},\tau^2)\big|+2^{-\nu}\Big) \big|\Phi_{1}(\beta_1;U_{L})-\Phi_{1}(\alpha_1; U_{L},\tau^2)\big|s\\[5pt]
&\leq O(1)\big((1+|\omega_{k}|)\tau^2+2^{-\nu}\big)(1+|\omega_{k}|)\tau^2s.
\end{split}
\end{eqnarray*}

\vspace{-5mm}
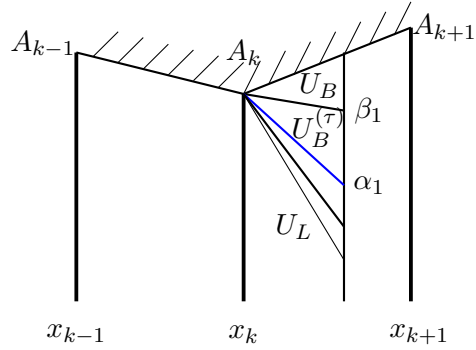
\begin{figure}[ht]
\begin{center}
\begin{tikzpicture}[scale=1.1]
\draw [thick](-5.0,1.5)--(-3,1)--(-1,1.8);

\draw [line width=0.05cm](-5.0,1.5)--(-5.0,-1.5);
\draw [line width=0.05cm](-3,1)--(-3,-1.5);
\draw [line width=0.05cm](-1.0,1.8)--(-1.0,-1.5);
\draw [thick](-1.8,1.5)--(-1.8,-1.5);

\draw [thin](-4.8,1.45)--(-4.5, 1.75);
\draw [thin](-4.5,1.38)--(-4.2, 1.68);
\draw [thin](-4.2,1.30)--(-3.9, 1.60);
\draw [thin](-3.9, 1.23)--(-3.6,1.53);
\draw [thin](-3.6, 1.16)--(-3.3,1.46);
\draw [thin](-3.3, 1.08)--(-3.0,1.38);
\draw [thin](-3.0, 1.0)--(-2.8,1.4);
\draw [thin](-2.7,1.12)--(-2.5, 1.52);
\draw [thin](-2.4,1.23)--(-2.2, 1.63);
\draw [thin](-2.1,1.35)--(-1.9, 1.75);
\draw [thin](-1.8,1.47)--(-1.6, 1.87);
\draw [thin](-1.5, 1.59)--(-1.3,1.99);
\draw [thin](-1.2,1.71)--(-1.0, 2.11);

\draw [thick](-3,1)--(-1.8,0.8);
\draw [thick][blue](-3,1)--(-1.8,-0.1);
\draw [thick](-3,1)--(-1.8,-0.6);
\draw [thin](-3,1)--(-1.8,-1.0);

\node at (-5.4, 1.6) {$A_{k-1}$};
\node at (-3.0, 1.5) {$A_{k}$};
\node at (-0.6, 1.8) {$A_{k+1}$};
\node at (-1.5, 0.8) {$\beta_{1}$};
\node at (-1.5, -0.1) {$\alpha_{1}$};
\node at (-2.1, 1.1) {$U_{B}$};
\node at (-2.1, 0.6) {$U^{(\tau)}_{B}$};
\node at (-2.4, -0.6) {$U_{L}$};
\node at (1, 2) {$$};

\node at (-5.0, -1.9) {$x_{k-1}$};
\node at (-3.0, -1.9) {$x_{k}$};
\node at (-0.9, -1.9) {$x_{k+1}$};
\end{tikzpicture}
\caption{Case $\omega_{k}>0$ and $\lambda_{1}(U_{L})<\dot{y}_{\alpha_1}\leq\lambda_{1}(U_{B})$}\label{fig5.8}
\end{center}
\end{figure}

If $\lambda_{1}(U_{L})<\dot{y}_{\alpha_1}\leq\lambda_{1}(U_{B})$, then
\begin{eqnarray*}
\begin{split}
B''_{\mathcal{R}_1}&=\int^{b_{k}+\lambda_{1}(U_{B})s}_{b_{k}+\lambda_{1}(U_{L})s}
\big|\mathcal{P}_{h}(x_k+s, x_{k})(U^{(\tau),B}_{h,\nu}(x_{k},y))-U^{(\tau),B}_{h,\nu}(x_{k}+s,y)\big|dy\\[5pt]
&=\Bigg(\int^{b_{k}+\dot{y}_{\alpha_1}s}_{b_{k}+\lambda_{1}(U_{L})s}+\int^{b_{k}+\lambda_{1}(U_{B})s}_{b_{k}+\dot{y}_{\alpha_1}s}\Bigg)
\big|\mathcal{P}_{h}(x_k+s, x_{k})(U^{(\tau),B}_{h,\nu}(x_{k},y))-U^{(\tau),B}_{h,\nu}(x_{k}+s,y)\big|dy\\[5pt]
&\leq \big|\dot{y}_{\alpha_1}-\lambda_{1}(U_{L})\big||U_B-U_{L}|s
+\big|\lambda_{1}(U_{B})-\dot{y}_{\alpha_1}\big||U_B-U^{(\tau)}_B|s\\[5pt]
&= \big|\dot{y}_{\alpha_1}-\lambda_{1}(U_{L})\big|\big|\Phi_{1}(\beta_1;U_{L})-U_{L}\big|s
+\big|\lambda_{1}(U_{B})-\dot{y}_{\alpha_1}\big|\big|\Phi_{1}(\beta_1;U_{L})-\Phi_{1}(\alpha_1;U_{L}, \tau^2)\big|s\\[5pt]
&\leq O(1)\Big(\big|\lambda_{1}(U^{(\tau)}_{B},\tau^2)-\lambda_{1}(U_{L})\big||\beta_1|
+\big|\lambda_{1}(U^{(\tau)}_{B},\tau^2)-\lambda_{1}(U_{B})\big||\alpha_1|\tau^2\Big)s\\[5pt]
&\quad +2^{-\nu}\big(|\beta_1|+|\alpha_1|\tau^2\big)s\\[5pt]
&\leq \Big(O(1)\big(|\alpha_1|+\tau^{2}\big)\big(|\alpha_1|+O(1)(1+|\alpha_1|)\tau^2\big)+O(1)(1+|\alpha_1|)|\alpha_1|\tau^4\Big)s\\[5pt]
&\quad +2^{-\nu}\big(|\alpha_1|+O(1)(1+|\alpha_1|)\tau^2\big)s\\[5pt]
&\leq O(1)\big( (1+|\omega_k|)\tau^2+(\nu^{-1}+2^{-\nu})(|\omega_k|+\tau^2)\big)s,
\end{split}
\end{eqnarray*}
where we have used \eqref{eq:4.10} for the rarefaction front in Proposition \ref{prop:4.1}.

\vspace{-5mm}
\begin{figure}[ht]
\begin{center}
\begin{tikzpicture}[scale=1.1]
\draw [thick](-5.0,1.5)--(-3,1)--(-1,1.8);

\draw [line width=0.05cm](-5.0,1.5)--(-5.0,-1.5);
\draw [line width=0.05cm](-3,1)--(-3,-1.5);
\draw [line width=0.05cm](-1.0,1.8)--(-1.0,-1.5);
\draw [thick](-1.8,1.5)--(-1.8,-1.5);

\draw [thin](-4.8,1.45)--(-4.5, 1.75);
\draw [thin](-4.5,1.38)--(-4.2, 1.68);
\draw [thin](-4.2,1.30)--(-3.9, 1.60);
\draw [thin](-3.9, 1.23)--(-3.6,1.53);
\draw [thin](-3.6, 1.16)--(-3.3,1.46);
\draw [thin](-3.3, 1.08)--(-3.0,1.38);
\draw [thin](-3.0, 1.0)--(-2.8,1.4);
\draw [thin](-2.7,1.12)--(-2.5, 1.52);
\draw [thin](-2.4,1.23)--(-2.2, 1.63);
\draw [thin](-2.1,1.35)--(-1.9, 1.75);
\draw [thin](-1.8,1.47)--(-1.6, 1.87);
\draw [thin](-1.5, 1.59)--(-1.3,1.99);
\draw [thin](-1.2,1.71)--(-1.0, 2.11);

\draw [thick][blue](-3,1)--(-1.8,0.8);
\draw [thick](-3,1)--(-1.8,-0.1);
\draw [thick](-3,1)--(-1.8,-0.6);
\draw [thin](-3,1)--(-1.8,-1.0);

\node at (-5.4, 1.6) {$A_{k-1}$};
\node at (-3.0, 1.5) {$A_{k}$};
\node at (-0.6, 1.8) {$A_{k+1}$};
\node at (-1.5, 0.8) {$\alpha_{1}$};
\node at (-1.5, -0.1) {$\beta_{1}$};
\node at (-2.1, 1.1) {$U^{(\tau)}_{B}$};
\node at (-2.1, 0.6) {$U_{B}$};
\node at (-2.4, -0.6) {$U_{L}$};
\node at (1, 2) {$$};

\node at (-5.0, -1.9) {$x_{k-1}$};
\node at (-3.0, -1.9) {$x_{k}$};
\node at (-0.9, -1.9) {$x_{k+1}$};
\end{tikzpicture}
\caption{Case $\omega_{k}>0$ and $\dot{y}_{\alpha_1}>\lambda_{1}(U_{B})$}\label{fig5.9}
\end{center}
\end{figure}

If $\dot{y}_{\alpha_1}>\lambda_{1}(U_{B})$, then
\begin{eqnarray*}
\begin{split}
B''_{\mathcal{R}_1}&=\int^{b_{k}+\dot{y}_{\alpha_1}s}_{b_{k}+\lambda_{1}(U_{L})s}
\big|\mathcal{P}_{h}(x_k+s, x_{k})(U^{(\tau),B}_{h,\nu}(x_{k},y))-U^{(\tau),B}_{h,\nu}(x_{k}+s,y)\big|dy\\[5pt]
&\leq \big|\dot{y}_{\alpha_1}-\lambda_{1}(U_L)\big||U_B-U_L|s\\[5pt]
&=\big|\dot{y}_{\alpha_1}-\lambda_{1}(U^{(\tau)}_B,\tau^2)+\lambda_{1}(U^{(\tau)}_B,\tau^2)-\lambda_{1}(U_L)\big|
\big|\Phi_1(\beta_1,U_{L})-U_L\big|s\\[5pt]
&\leq O(1)\big(|\lambda_{1}(U^{(\tau)}_B,\tau^2)-\lambda_{1}(U_L)|+2^{-\nu}\big)|\beta_1|s\\[5pt]
&\leq \big(O(1)(|\omega_{k}|+\tau^{2})+2^{-\nu}\big)\big(|\omega_k|+O(1)(1+|\omega_k|)\tau^2\big)s.
\end{split}
\end{eqnarray*}
Here, we also use estimate \eqref{eq:4.10} for the rarefaction front in Proposition \ref{prop:4.1}.
We thus obtain
$$B''_{R_1}\leq O(1)\big( (1+|\omega_k|)\tau^2+(\nu^{-1}+2^{-\nu})(|\omega_k|+\tau^2)\big)s.$$

Therefore, we can choose a constant $C_{6}>0$ depending only on $\underline{U}$
and $a_{\infty}$ so that estimate \eqref{eq:5.27} holds. This completes the proof of case \rm(2c).
\end{proof}

We are now ready 
to establish the following $L^1$ difference estimate between the Riemann solvers $\mathcal{P}_{h}(\varsigma+s,x)(U^{(\tau)}_{h,\nu}(\varsigma,y))$ and the approximate solution $U^{(\tau)}_{h,\nu}(\varsigma+s, y)$ for any $\varsigma>0$.
\begin{proposition}\label{prop:5.1}
Let $U^{(\tau)}_{h,\nu}(x,y)$ be the approximate solution to the initial-boundary value problem \eqref{eq:1.16}-\eqref{eq:1.18} with approximate boundary $y=b_h(x)$. Assume the initial data $U^{\nu}_{0}(y)$ satisfies the properties in Proposition \ref{prop:4.1}, and there is no wave interaction on the line $x=\varsigma$ and no wave reflection at $(\varsigma,b_{h}(\varsigma))$ for some $\varsigma>0$. Let $\mathcal{P}_{h}$ be the uniformly Lipschitz continuous map given by Proposition \ref{prop:4.2}.
Then, 
\begin{eqnarray}\label{eq:5.32}
\begin{split}
&\big\|\mathcal{P}_{h}(\varsigma+s,x)(U^{(\tau)}_{h,\nu}(\varsigma,\cdot))-U^{(\tau)}_{h,\nu}(\varsigma+s, \cdot)\big\|_{L^{1}((-\infty, b_{h}(\varsigma+s)))}\\[5pt]
&\ \ \  \leq O(1)\bigg(\Big(T.V.\{U_{0}(\cdot); \mathcal{I}\}+|b'(0)|+T.V.\{b'(\cdot);\mathbb{R}_{+}\}+1\Big)\big(\tau^2+\nu^{-1}+2^{-\nu}\big)+2^{-\nu}
\bigg)s,
\end{split}
\end{eqnarray}
where constant $O(1)$ depends only on $\underline{U}$ and $a_{\infty}$.
\end{proposition}

\begin{proof}
Suppose the jumps of the
$U^{(\tau)}_{h,\nu}(\varsigma,\cdot)$ are denoted as $y_{1}>y_{2}>\cdot\cdot\cdot y_{N}$ with $y_{1}\leq b_{h}(\varsigma)$. Let $\mathcal{S}^{(\tau)}$ (or $\mathcal{R}^{(\tau)}$)
be the set of indices $\alpha\in\{1,2, \cdot\cdot\cdot, N\}$ such that $U^{(\tau)}_{h,\nu}(\varsigma,y_{\alpha}+)$ and $U^{(\tau)}_{h,\nu}(\varsigma,y_{\alpha}-)$ are connected by a shock wave (or rarefaction front) with strength $\alpha$.
Similarly, denote by $\mathcal{NP}^{(\tau)}$ the set of indices $\alpha\in\{1,2, \cdot\cdot\cdot, N\}$ such that $U^{(\tau)}_{h,\nu}(\varsigma,y_{\alpha}+)$ and $U^{(\tau)}_{h,\nu}(\varsigma,y_{\alpha}-)$ are connected by a non-physical front with strength $\alpha_{\mathcal{NP}}$.

Since $U^{(\tau)}_{h,\nu}(x,\cdot)$ has no wave interactions on the interval $(\varsigma,\varsigma+s)$ and no wave reflection at $(x,b_{h}(x))$ for $x\in(\varsigma,\varsigma+s)$, if $\varsigma$ satisfies $\varsigma>[\frac{\varsigma}{h}]h$, then $(\varsigma,b_{h}(\varsigma))$ is not a corner point. Hence, by Lemma \ref{lem:5.1}, Corollary \ref{coro:5.1} and Lemma \ref{lem:5.2}, for sufficiently small $s>0$, we have
{\small\begin{eqnarray*}
\begin{split}
& \big\|\mathcal{P}_{h}(\varsigma+s,\varsigma)(U^{(\tau)}_{h,\nu}(\varsigma,\cdot))-U^{(\tau)}_{h,\nu}(\varsigma+s,\cdot)\big\|_{L^{1}((-\infty, b_{h}(\varsigma+s)))}\\[5pt]
=&\sum_{\alpha\in S^{(\tau)}\cup R^{(\tau)}\cup NP^{(\tau)}}\int^{y_{\alpha}+\eta}_{y_{\alpha}-\eta}
\big|\mathcal{P}_{h}(\varsigma+s,\varsigma)(U^{(\tau)}_{h,\nu}(\varsigma,\cdot))-U^{(\tau)}_{h,\nu}(\varsigma+s,\cdot)\big|dy \\[5pt]
&+\int^{b_{h}(\varsigma+s)}_{y_{1}-\eta}\big|\mathcal{P}_{h}(\varsigma+s,\varsigma)(U^{(\tau)}_{h,\nu}(\varsigma,\cdot))
-U^{(\tau)}_{h,\nu}(\varsigma+s,\cdot)\big|dy\\[5pt]
\leq& O(1)\big(\tau^2+\nu^{-1}+2^{-\nu}\big)\Big(\sum _{\alpha\in \mathcal{S}^{(\tau)}\cup \mathcal{R}^{(\tau)}}|\alpha|\Big)s+O(1)\Big(\sum _{\alpha\in \mathcal{NP}^{(\tau)}}\alpha_{\mathcal{NP}}\Big)s
+O(1)\big(\tau^2+\nu^{-1}+2^{-\nu}\big)\Big(\sum_{k\geq 0}|\omega_{k}|\Big)s\\[5pt]
\leq &O(1)\bigg(\Big(T.V.\{U^{(\tau)}_{h,\nu}(\varsigma, \cdot); (-\infty,b_{h}(\varsigma))\}+|b'(0)|+T.V.\{b'(\cdot);\mathbb{R}_{+}\}\Big)\big(\tau^2+\nu^{-1}+2^{-\nu}\big)+2^{-\nu}
\bigg)s\\[5pt]
\leq &O(1)\bigg(\Big(T.V.\{U_{0}(\cdot); \mathcal{I}\}+|b'(0)|+T.V.\{b'(\cdot);\mathbb{R}_{+}\}\Big)\big(\tau^2+\nu^{-1}+2^{-\nu}\big)+2^{-\nu}
\bigg)s,
\end{split}
\end{eqnarray*}}
where $k=[\frac{\varsigma}{h}]$, $\eta=\frac{1}{2}\min_{1\leq j\leq N-1}\big\{b_{h}(\varsigma)-y_1, y_j-y_{j+1}\big\}$,
and constant $O(1)$ depends only on $\underline{U}$ and $a_{\infty}$.

If $\varsigma=[\frac{\varsigma}{h}]h$, $\varsigma=x_{k}$ for $k=[\frac{\varsigma}{h}]$. So, by Lemma \ref{lem:5.1}, Corollary \ref{coro:5.1} and Lemma \ref{lem:5.3}, 
{\small\begin{eqnarray*}
\begin{split}
& \ \ \ \big\|\mathcal{P}_{h}(x_{k}+s,x_{k})(U^{(\tau)}_{h,\nu}(x_{k},\cdot))-U^{(\tau)}_{h,\nu}(x_{k}+s,\cdot)\big\|_{L^{1}((-\infty, b_{h}(x_{k}+s)))}\\[5pt]
&  =\sum_{\alpha\in \mathcal{S}^{(\tau)}\cup \mathcal{R}^{(\tau)}\cup \mathcal{NP}^{(\tau)}}\int^{y_{\alpha}+\eta}_{y_{\alpha}-\eta}
\big|\mathcal{P}_{h}(x_{k}+s,x_{k})(U^{(\tau)}_{h,\nu}(x_{k},\cdot))-U^{(\tau)}_{h,\nu}(x_{k}+s,\cdot)\big|dy \\[5pt]
&\ \ \  +\int^{b_{h}(x_{k}+s)}_{b_{k}-\eta}\big|\mathcal{P}_{h}(x_{k}+s,\varsigma)(U^{(\tau)}_{h,\nu}(x_{k},\cdot))
-U^{(\tau)}_{h,\nu}(x_{k}+s,\cdot)\big|dy\\[5pt]
&\leq O(1)\big(\tau^2+\nu^{-1}+2^{-\nu}\big)\Big(\sum _{\alpha\in \mathcal{S}^{(\tau)}\cup \mathcal{R}^{(\tau)}}|\alpha|\Big)s+O(1)\Big(\sum _{\alpha\in \mathcal{NP}^{(\tau)}}\alpha_{\mathcal{NP}}\Big)s\\[5pt]
&\ \ \  +O(1)\Bigg( \bigg(1+\Big(\sum_{k\geq 0}|\omega_{k}|\Big)\bigg)\tau^2+\bigg(\Big(\sum_{k\geq 0}|\omega_{k}|\Big)+\tau^{2}\bigg)(\nu^{-1}+2^{-\nu})\Bigg)s\\[5pt]
&\leq O(1)\bigg(\Big(T.V.\{U^{(\tau)}_{h,\nu}(\varsigma, \cdot); (-\infty,b_{h}(\varsigma))\}+|b'(0)|+T.V.\{b'(\cdot);\mathbb{R}_{+}\}\Big)\big(\tau^2+\nu^{-1}+2^{-\nu}\big)+2^{-\nu}
\bigg)s\\[5pt]
&\ \ \  +O(1)\bigg(1+\Big(\sum_{k\geq 0}|\omega_{k}|\Big)\bigg)\big(\tau^{2}+\nu^{-1}+2^{-\nu}\big)s\\[5pt]
&\leq O(1)\bigg(\Big(T.V.\{U_{0}(\cdot); \mathcal{I}\}+|b'(0)|+T.V.\{b'(\cdot);\mathbb{R}_{+}\}+1\Big)\big(\tau^2+\nu^{-1}+2^{-\nu}\big)+2^{-\nu}
\bigg)s,
\end{split}
\end{eqnarray*}}
where $\eta=\frac{1}{2}\min_{1\leq j\leq N-1}\big\{y_j-y_{j+1}\big\}$, and constant $O(1)$ depends only on $\underline{U}$ and $a_{\infty}$.

It thus yields estimate \eqref{eq:5.32}.
\end{proof}

\subsection{Global $L^1$ difference estimates between the approximate solution $U^{(\tau)}_{h,\nu}(x,y)$ and the trajectory $\mathcal{P}_{h}(x,0)(U^{\nu}_{0}(y))$}
First, we will present several lemmas for the later applications.
\begin{lemma}\label{lem:5.4}
Let $W(x): [0,X]\mapsto \mathcal{D}_{h,x}$ be a Lipschitz continuous map with a finite number of wave fronts for some $X>0$. Assume that the boundary $y=b_{h}(x)$ is a straight line and any fronts of $W$ do not interact with each other or hit the boundary when $[t, x]\subseteq[0,X]$. Then, for the map $\mathcal{P}_{h}$ which is defined by Proposition \ref{prop:4.1},
\begin{eqnarray}\label{eq:5.33}
\begin{split}
&\big\|\mathcal{P}_{h}(X, x)(W(x))-\mathcal{P}_{h}(X, t)(W(t))\big\|_{L^{1}((-\infty, b_{h}(X)))}\\[5pt]
&\ \ \ \ \  \leq L\int^{x}_{t}\overline{\lim}_{h\rightarrow0^{+}}\frac{\big\|\mathcal{P}_{h}(\varsigma+h, \varsigma)(W(\varsigma))-W(\varsigma+h)\big\|_{L^{1}((-\infty, b_{h}(\varsigma+h)])}}{h}d\varsigma,
\end{split}
\end{eqnarray}
where $\mathcal{D}_{h,x}$ is given by Proposition \ref{prop:4.1}, and constant $L>0$ depends only on $\underline{U}$.
\end{lemma}

\begin{proof}
Consider a partition $\{\tilde{x}_{i}\}^{k}_{i=1}$ of the interval $[t,x]$:
$
t\doteq \tilde{x}_{0}<\tilde{x}_{1}<\cdot\cdot\cdot<\tilde{x}_{k}\doteq x,
$
and let $\tilde{\lambda}=\max_{1\leq j\leq k}\{\tilde{x}_{j}-\tilde{x}_{j-1}\}$.
Then, 
\begin{eqnarray*}
\begin{split}
&\ \ \ \big\|\mathcal{P}_{h}(X, x)(W(x))-\mathcal{P}_{h}(X, t)(W(t))\big\|_{L^{1}((-\infty, b_{h}(X)))}\\[5pt]
&\leq\sum^{k}_{j=1}\big\|\mathcal{P}_{h}(X, \tilde{x}_{j-1})(W(\tilde{x}_{j-1}))-\mathcal{P}_{h}(X, \tilde{x}_{j})(W(\tilde{x}_{j}))\big\|_{L^{1}((-\infty, b_{h}(X)))}\\[5pt]
&\leq\sum^{k}_{j=1}\big\|\mathcal{P}_{h}(X, \tilde{x}_{j})\mathcal{P}_{h}( \tilde{x}_{j}, \tilde{x}_{j-1})(W(\tilde{x}_{j-1}))-P_{h}(X, \tilde{x}_{j})(W(\tilde{x}_{j}))\big\|_{L^{1}((-\infty, b_{h}(X)))}\\[5pt]
&\leq L\sum^{k}_{j=1}\big\|\mathcal{P}_{h}( \tilde{x}_{j}, \tilde{x}_{j-1})(W(\tilde{x}_{j-1}))-W(\tilde{x}_{j})\big\|_{L^{1}((-\infty, b_{h}(\tilde{x}_{j})))}\\[5pt]
&\leq L\sum^{k}_{j=1}\frac{\big\|\mathcal{P}_{h}( \tilde{x}_{j}, \tilde{x}_{j-1})(W(\tilde{x}_{j-1}))-W(\tilde{x}_{j})\big\|_{L^{1}((-\infty, b_{h}(\tilde{x}_{j})))}}{ \tilde{x}_{j-1}- \tilde{x}_{j}}(\tilde{x}_{j-1}- \tilde{x}_{j}).
\end{split}
\end{eqnarray*}
Taking $\tilde{\lambda}\rightarrow 0$ in the above inequality, we can get estimate \eqref{eq:5.33}.
\end{proof}

Next, let us consider more general case for boundary function $b_{h}$ and $W$. 
\begin{lemma}\label{lem:5.5}
Let $W(x): [0,X]\mapsto \mathcal{D}_{h,x}$ be a Lipschitz continuous map for some $X>0$ with finite number of wave fronts. Let boundary function $b_{h}(x)$ be defined by \eqref{eq:4.2}-\eqref{eq:4.2x}. Let $\mathcal{P}_{h}(x,x'_{0}): \mathbb{R}^2_{+}\times\mathcal{D}_{h,x'_{0}}\mapsto \mathcal{D}_{h,x}$ be a uniformly Lipschitz continuous map given by Proposition \ref{prop:4.2}. 
Then,
\begin{eqnarray}\label{eq:5.34}
\begin{split}
&\big\|\mathcal{P}_{h}(x, 0)(W(0))-W(x)\big\|_{L^{1}((-\infty, b_{h}(x)))}\\[5pt]
&\ \ \ \ \  \leq L\int^{x}_{0}\overline{\lim}_{s\rightarrow0^{+}}\frac{\big\|P_{h}(\varsigma+s, \varsigma)(W(\varsigma))-W(\varsigma+s)\big\|_{L^{1}((-\infty, b_{h}(\varsigma+h)))}}{s}d\varsigma,
\end{split}
\end{eqnarray}
where domain $\mathcal{D}_{h,x}$ is given by Proposition \ref{prop:4.1}, and constant $L>0$ depends only on $\underline{U}$.
\end{lemma}

\begin{proof}
Define 
\begin{eqnarray*}
\begin{split}
\Lambda&\doteq\Big\{x|\ at\ the\ "time"\ x, \ there\ are\ at\ least\ two\ wave\ fronts\ in\ W\ interacted\ or\ a\ wave \\[5pt]
&\qquad \quad   \ front \ hitting\ the\ boundary\  y=b_{h}(x)  \Big\}.
\end{split}
\end{eqnarray*}

Obviously, the set $\Lambda$ is finite. Let $a_{k}\in \Lambda $ for $k=1,2,\cdot\cdot\cdot l$, satisfying $a_{1}<a_{2}<\cdot\cdot\cdot<a_{l}$. For any $\hat{\epsilon}>0$, define
\begin{eqnarray*}
\begin{split}
\tilde{x}_{0}\doteq0, \ \tilde{x}_{1}=a_{1}-\frac{\hat{\epsilon}}{2l}, \ \tilde{x}_{2}=a_{1}+\frac{\hat{\epsilon}}{2 l},
\cdot\cdot\cdot,  \tilde{x}_{2l-1}=a_{l}-\frac{\hat{\epsilon}}{2l}, \ \tilde{x}_{2l}=a_{l}+\frac{\hat{\epsilon}}{2l},\
\tilde{x}_{2l+1}\doteq X.
\end{split}
\end{eqnarray*}

Then, we have
\begin{eqnarray*}
\begin{split}
&\ \ \ \big\|\mathcal{P}_{h}(X,0)(W(0))-W(X)\big\|_{L^1((-\infty, b_{h}(X)))}\\[5pt]
&\leq \sum^{2l+1}_{k=1}\big\|\mathcal{P}_{h}(X,\tilde{x}_{k-1})(W(\tilde{x}_{k-1}))
-\mathcal{P}_{h}(X,\tilde{x}_{k})(W(\tilde{x}_{k}))\big\|_{L^1((-\infty, b_{h}(X)))}\\[5pt]
&=\sum^{l}_{k=0}\big\|\mathcal{P}_{h}(X,\tilde{x}_{2k})(W(\tilde{x}_{2k}))
-\mathcal{P}_{h}(X,\tilde{x}_{2k+1})(W(\tilde{x}_{2k+1}))\big\|_{L^1((-\infty, b_{h}(X)))}\\[5pt]
&\ \ \ + \sum^{l}_{k=1}\big\|\mathcal{P}_{h}(X,\tilde{x}_{2k-1})(W(\tilde{x}_{2k-1}))
-\mathcal{P}_{h}(X,\tilde{x}_{2k})(W(\tilde{x}_{2k}))\big\|_{L^1((-\infty, b_{h}(X)))}
\\[5pt]
&\doteq J_{1}+J_{2}.
\end{split}
\end{eqnarray*}

For $J_{1}$, we further have 
\begin{eqnarray*}
\begin{split}
J_{1}&\leq \sum^{l}_{k=0}\big\|\mathcal{P}_{h}(X,\tilde{x}_{2k+1})\circ\mathcal{P}_{h}(\tilde{x}_{2k+1},\tilde{x}_{2k})(W(\tilde{x}_{2k}))
-\mathcal{P}_{h}(X,\tilde{x}_{2k+1})(W(\tilde{x}_{2k+1}))\big\|_{L^1((-\infty, b_{h}(X)))}\\[5pt]
&\leq L\sum^{l}_{k=0}\big\|\mathcal{P}_{h}(\tilde{x}_{2k+1},\tilde{x}_{2k})(W(\tilde{x}_{2k}))
-W(\tilde{x}_{2k+1})\big\|_{L^1((-\infty, b_{h}(\tilde{x}_{2k+1})))}\\[5pt]
&=L\sum^{l}_{k=0}\big\|\mathcal{P}_{h}(\tilde{x}_{2k+1},\tilde{x}_{2k})(W(\tilde{x}_{2k}))
-\mathcal{P}_{h}(\tilde{x}_{2k+1},\tilde{x}_{2k+1})(W(\tilde{x}_{2k+1}))\big\|_{L^1((-\infty, b_{h}(\tilde{x}_{2k+1})))}.
\end{split}
\end{eqnarray*}

Since the fronts in $W(x)$ do not interact with each other, or hit the boundary, or issue from the corner points of the boundary, and since the boundary $y=b_{h}(x)$ is straight in $[\tilde{x}_{2k}, \tilde{x}_{2k+1}]$, we can apply Lemma \ref{lem:5.1} to get that
\begin{eqnarray*}
\begin{split}
J_{1}\leq L\sum^{l}_{k=0}\int^{\tilde{x}_{2k+1}}_{\tilde{x}_{2k}}\frac{\big\|\mathcal{P}_{h}(\mu+s,\mu)(W(\mu))
-W(\mu+s)\big\|_{L^1((-\infty, b_{h}(\mu+s)))}}{s}d\mu.
\end{split}
\end{eqnarray*}

For the second term $J_{2}$, we have 
\begin{eqnarray*}
\begin{split}
J_{2}&= \sum^{l}_{k=0}\big\|\mathcal{P}_{h}(X,\tilde{x}_{2k})\circ\mathcal{P}_{h}(\tilde{x}_{2k},\tilde{x}_{2k-1})(W(\tilde{x}_{2k-1}))
-\mathcal{P}_{h}(X,\tilde{x}_{2k})(W(\tilde{x}_{2k}))\big\|_{L^1((-\infty, b_{h}(X)))}\\[5pt]
&\leq L\sum^{l}_{k=0}\big\|\mathcal{P}_{h}(\tilde{x}_{2k},\tilde{x}_{2k-1})(W(\tilde{x}_{2k-1}))
-W(\tilde{x}_{2k})\big\|_{L^1((-\infty, b_{h}(\tilde{x}_{2k})))}\\[5pt]
&\leq L\sum^{l}_{k=0}\big\|\mathcal{P}_{h}(\tilde{x}_{2k},\tilde{x}_{2k-1})(W(\tilde{x}_{2k-1}))
-W(\tilde{x}_{2k-1})\big\|_{L^1((-\infty, b_{h}(\tilde{x}_{2k}))}\\[5pt]
&\ \ \ +L\sum^{l}_{k=0}\big\|W(\tilde{x}_{2k-1})
-W(\tilde{x}_{2k})\big\|_{L^1((-\infty, b_{h}(\tilde{x}_{2k})))}\\[5pt]
&\doteq J_{21}+J_{22}.
\end{split}
\end{eqnarray*}

For $J_{21}$, 
\begin{eqnarray*}
\begin{split}
J_{21}&=L\sum^{l}_{k=0}\big\|\mathcal{P}_{h}(\tilde{x}_{2k},\tilde{x}_{2k-1})(W(\tilde{x}_{2k-1}))
-\mathcal{P}_{h}(\tilde{x}_{2k-1},\tilde{x}_{2k-1})(W(\tilde{x}_{2k-1}))\big\|_{L^1((-\infty, b_{h}(\tilde{x}_{2k})])}\\[5pt]
& \leq O(1)\sum^{l}_{k=0}(\tilde{x}_{2k}-\tilde{x}_{2k-1})\\[5pt]
&\leq O(1)l \frac{\hat{\epsilon}}{2l}=O(1)\hat{\epsilon}.
\end{split}
\end{eqnarray*}

For $J_{22}$, 
\begin{eqnarray*}
\begin{split}
J_{22}\leq O(1)\sum^{l}_{k=0}(\tilde{x}_{2k}-\tilde{x}_{2k-1})\leq O(1)\hat{\epsilon}.
\end{split}
\end{eqnarray*}

Thus, 
we obtain that $J_{2}\leq O(1)\hat{\epsilon}$.
Finally, by combing the estimates on $J_{1}$ and $J_{2}$, and by letting $\hat{\epsilon}\rightarrow 0$, we complete the proof of this lemma.
\end{proof}

By Lemma \ref{lem:5.5}, we can get the global $L^1$ error estimate between the approximate solutions $U^{(\tau)}_{h,\nu}$
and the trajectory of $\mathcal{P}_{h}$ as follows.
\begin{proposition}\label{prop:5.2}
Let $\mathcal{P}_{h}$ be the uniformly Lipschitz continuous map. 
Let $U^{(\tau)}_{h,\nu}(x,y)$ be the approximate solution of initial-boundary value problem \eqref{eq:1.16}-\eqref{eq:1.18}
with initial data $U^{\nu}_{0}(y)$ to the approximate boundary $y=b_h(x)$. Then
\begin{eqnarray}\label{eq:5.35}
\begin{split}
&\big\|\mathcal{P}_{h}(x,0)(U^{\nu}_{0}(\cdot))-U^{(\tau)}_{h,\nu}(x,\cdot)\big\|_{L^{1}((-\infty, b_{h}(x)))}\\[5pt]
&\ \ \  \leq O(1)\bigg(\Big(T.V.\{U_{0}(\cdot); \mathcal{I}\}+|b'(0)|+T.V.\{b'(\cdot);\mathbb{R}_{+}\}+1\Big)\big(\tau^2+\nu^{-1}+2^{-\nu}\big)+2^{-\nu}
\bigg)x,
\end{split}
\end{eqnarray}
where the constant $O(1)$ depends only on $\underline{U}$ and $a_{\infty}$.
\end{proposition}

\begin{proof}
By Proposition \ref{prop:5.1} and Lemma \ref{lem:5.5}, we have
\begin{eqnarray*}
\begin{split}
&\big\|\mathcal{P}_{h}(x,0)(U^{\nu}_{0}(\cdot))-U^{(\tau)}_{h,\nu}(x,\cdot)\big\|_{L^{1}((-\infty, b_{h}(x)))}\\[5pt]
& \leq
L\int^{x}_{0}\overline{\lim}_{s\rightarrow0^{+}}\frac{\big\|\mathcal{P}_{h}(\varsigma+s,\varsigma)(U^{(\tau)}_{h,\nu}(\varsigma,\cdot))
-U^{(\tau)}_{h,\nu}(\varsigma+s,\cdot)\big\|_{L^{1}((-\infty, b_{h}(\varsigma+h)))}}{s}d\varsigma\\[5pt]
& \leq L\int^{x}_{0} O(1)\bigg(\Big(T.V.\{U_{0}(\cdot); \mathcal{I}\}+|b'(0)|+T.V.\{b'(\cdot);\mathbb{R}_{+}\}+1\Big)\big(\tau^2+\nu^{-1}+2^{-\nu}\big)+2^{-\nu}
\bigg)d\varsigma \\[5pt]
&\leq O(1)\bigg(\Big(T.V.\{U_{0}(\cdot); \mathcal{I}\}+|b'(0)|+T.V.\{b'(\cdot);\mathbb{R}_{+}\}+1\Big)\big(\tau^2+\nu^{-1}+2^{-\nu}\big)+2^{-\nu}
\bigg) x.
\end{split}
\end{eqnarray*}
\end{proof}

\subsection{Proof of Theorem \ref{thm:1.1}}
Now, we are ready to prove Theorem \ref{thm:1.1} in this subsection. We devide the proof into two steps.

\emph{{Step 1}}. As shown in Proposition \ref{prop:4.1}-\ref{prop:4.3},
for any given $\nu\in \mathbb{N}_{+}$ and $h>0$, we can construct a global approximate solutions $U^{(\tau)}_{h,\nu}$ to problem \eqref{eq:1.16}-\eqref{eq:1.18}
and $U_{h,\nu}$ to problem \eqref{eq:1.20}-\eqref{eq:1.21} and \eqref{eq:1.18}, with initial data $U^{\nu}_{0}$ and approximate boundary $y=b_{h}(x)$. Let $\mathcal{P}_{h}$ be the uniformly Lipschitz continuous map  generated by wave front tracking scheme to problem \eqref{eq:1.16}-\eqref{eq:1.18}.
Let $U^{(\tau)}$ and $U$ be the entropy solutions to problem \eqref{eq:1.16}-\eqref{eq:1.18}
and problem \eqref{eq:1.20}-\eqref{eq:1.21} and \eqref{eq:1.18} with initial data $U_{0}$ and boundary $y=b(x)$.
Then, by the triangle inequality, we have
\begin{eqnarray}\label{eq:5.36}
\begin{split}
&\ \ \ \big\|U^{(\tau)}(x,\cdot)-U(x,\cdot)\big\|_{L^{1}((-\infty, b(x)))}\\[5pt]
&\leq \big\|U^{(\tau)}(x,\cdot)-U^{(\tau)}_{h,\nu}(x,\cdot)\big\|_{L^{1}((-\infty, b(x)))}
 +\big\|U^{(\tau)}_{h,\nu}(x,\cdot)-\mathcal{P}_{h}(x,0)(U^{\nu}_{0}(\cdot))\big\|_{L^{1}((-\infty, b(x)))}\\[5pt]
&\ \ \ +\big\|\mathcal{P}_{h}(x,0)(U^{\nu}_{0}(\cdot))-\mathcal{P}_{h}(x,0)(U_{0}(\cdot))\big\|_{L^{1}((-\infty, b(x)))}\\[5pt]
&\ \ \ +\big\|\mathcal{P}_{h}(x,0)(U_{0}(\cdot))-\mathcal{P}(x,0)(U_{0}(\cdot))\big\|_{L^{1}((-\infty, b(x)))}.
\end{split}
\end{eqnarray}

Notice that, by Proposition \ref{prop:4.1} and Proposition \ref{prop:4.3}, we have
\begin{eqnarray*}
\begin{split}
&U^{(\tau)}_{h,\nu}(x,\cdot)\rightarrow U^{(\tau)}(x,\cdot),\quad \mbox{in}\quad L^{1}_{loc}(\Omega),\quad \mbox{as}\quad  h\rightarrow0,\ \nu\rightarrow+\infty,
\end{split}
\end{eqnarray*}
and
\begin{eqnarray*}
\begin{split}
\mathcal{P}_{h}(x,0)(U_{0}(\cdot))\rightarrow \mathcal{P}(x,0)(U_{0}(\cdot)), \quad \mbox{in}\ L^{1}((-\infty,b(x)))=U(x,\cdot)\quad  \mbox{as}\quad h\rightarrow0.
\end{split}
\end{eqnarray*}

For the third term on the right hand of \eqref{eq:5.36}, by Proposition \ref{prop:4.2} and estimates \eqref{eq:4.3}-\eqref{eq:4.4},
we can choose $\epsilon^{*}<\min\{\epsilon_0, \epsilon^{*}_0\}$ and $\tau^{*}\in(0,\tau^{*}_{0})$ sufficiently small depending only on $\underline{U}$ and $a_{\infty}$,
such that if $\tau\in(0,\tau^{*})$ and $T.V.\{U_{0}(\cdot);\mathcal{I}\}+|b'(0)|+T.V.\{b'(\cdot);\mathcal{I}\}<\epsilon$ for $\epsilon\in (0,\epsilon^{*})$,
\begin{eqnarray*}
\begin{split}
&\ \ \ \ \big\|\mathcal{P}_{h}(x,0)(U^{\nu}_{0}(\cdot))-\mathcal{P}_{h}(x,0)(U_{0}(\cdot))\big\|_{L^{1}((-\infty, b(x)))}\\[5pt]
&\leq \big\|\mathcal{P}_{h}(x,0)(U^{\nu}_{0}(\cdot))-\mathcal{P}_{h}(x,0)(U_{0}(\cdot))\big\|_{L^{1}((-\infty, b_h(x)))}\\[5pt]
&\ \ \ +\big\|\mathcal{P}_{h}(x,0)(U^{\nu}_{0}(\cdot))-\mathcal{P}_{h}(x,0)(U_{0}(\cdot))\big\|_{L^{1}(( \min\{b_{h}(x), b(x)\}\ \max\{b_{h}(x), b(x)\}))}\\[5pt]
&\leq L\|U^{\nu}_0(\cdot)-U_0(\cdot)\|_{L^1(\mathcal{I})}
+\big\|\mathcal{P}_{h}(x,0)(U^{\nu}_{0}(\cdot))-\underline{U}\big\|_{L^{\infty}((-\infty,b_h(x)))}
\big\|b_{h}(\cdot)-b(\cdot)\big\|_{L^{\infty}(\mathbb{R}_{+})}\\[5pt]
&\ \ \ +\big\|\mathcal{P}_{h}(x,0)(U_{0}(\cdot))-\underline{U}\big\|_{L^{\infty}((-\infty,b_h(x)))}
\big\|b_{h}(\cdot)-b(\cdot)\big\|_{L^{\infty}(\mathbb{R}_{+})}\\[5pt]
&\longrightarrow 0, \quad \mbox{as}\quad \nu\rightarrow +\infty,\ h \rightarrow 0.
\end{split}
\end{eqnarray*}

For the second term on the right hand side of \eqref{eq:5.36}, by Proposition \ref{prop:5.2}, we obtain
\begin{eqnarray*}
\begin{split}
&\ \ \ \ \big\|U^{(\tau)}_{h,\nu}(x,\cdot)-\mathcal{P}_{h}(x,0)(U^{\nu}_{0}(\cdot))\big\|_{L^{1}((-\infty, b(x)))}\\[5pt]
&\leq \big\|U^{(\tau)}_{h,\nu}(x,\cdot)-\mathcal{P}_{h}(x,0)(U^{\nu}_{0}(\cdot))\big\|_{L^{1}((-\infty, b_h(x)))}\\[5pt]
&\ \ \ +\big\|U^{(\tau)}_{h,\nu}(x,\cdot)-\mathcal{P}_{h}(x,0)(U^{\nu}_{0}(\cdot))\big\|_{L^{1}(( \min\{b_{h}(x), b(x)\}, \max\{b_{h}(x), b(x)\}))}\\[5pt]
&\leq O(1)\bigg(\Big(T.V.\{U_{0}(\cdot); \mathcal{I}\}+|b'(0)|+T.V.\{b'(\cdot);\mathbb{R}_{+}\}+1\Big)\big(\tau^2+\nu^{-1}+2^{-\nu}\big)+2^{-\nu}
\bigg) x\\[5pt]
&\ \ \ +\Big(\big\|U^{(\tau)}_{h,\nu}(x,\cdot)-\underline{U}\big\|_{L^{\infty}((-\infty,b_{h}(x)))}
+\big\|\mathcal{P}_{h}(x,0)(U^{\nu}_{0}(\cdot))-\underline{U}\big\|_{L^{\infty}((-\infty,b_{h}(x)))}\Big)
\big\|b_{h}(\cdot)-b(\cdot)\big\|_{L^{\infty}(\mathbb{R}_{+})}\\[5pt]
&\rightarrow O(1)\Big(T.V.\{U_{0}(\cdot); \mathcal{I}\}+|b'(0)|+T.V.\{b'(\cdot);\mathbb{R}_{+}\}+1\Big)x\tau^2,
\end{split}
\end{eqnarray*}
as $\nu\rightarrow +\infty,\ h \rightarrow 0$.

Take $C^*_1=O(1)\Big(T.V.\{U_{0}(\cdot); \mathcal{I}\}+|b'(0)|+T.V.\{b'(\cdot);\mathbb{R}_{+}\}+1\Big)$ and
choose $\epsilon^{*}>0$, $\tau^*>0$ small, so that if $U_{0}$ and $b(x)$ satisfy \eqref{eq:1.26} for $\epsilon\in(0,\epsilon^{*})$ and $\tau\in(0,\tau^*)$, then we can get estimate \eqref{eq:1.27} by combining the above estimates together.

For estimate \eqref{eq:1.28}, we need only to control $\big\|u^{(\tau)}(x,\cdot)-u(x,\cdot)\big\|_{L^{1}((-\infty,b(x)))}$.
By the third equation in \eqref{eq:1.10}, formula \eqref{eq:1.12}, Proposition \ref{prop:4.1}, Proposition \ref{prop:4.3}, and estimate \eqref{eq:1.27}, we can deduce that
\begin{eqnarray*}
\begin{split}
&\ \ \ \ \big\|u^{(\tau)}(x,\cdot)-u(x,\cdot)\big\|_{L^{1}((-\infty,b(x)))}\\[5pt]
&=\bigg\|\displaystyle\frac{1}{\tau^2}\Bigg(-1+\sqrt{1-\tau^2\Big((v^{(\tau)})^2+\frac{2((\rho^{(\tau)})^{\gamma-1}-1)}{(\gamma-1)a^{2}_{\infty}}\Big)}\Bigg)
+\displaystyle\frac{v^2}{2}+\frac{\rho^{\gamma-1}-1}{(\gamma-1)a^2_{\infty}}\bigg\|_{L^{1}((-\infty,b(x)))}\\[5pt]
&=\Bigg\|\displaystyle\frac{2\Big(\frac{(v^{(\tau)})^2}{2}+\frac{(\rho^{(\tau)})^{\gamma-1}-1}{(\gamma-1)a^{2}_{\infty}}\Big)}
{1+\sqrt{1-\tau^2\Big((v^{(\tau)})^2+\frac{2((\rho^{(\tau)})^{\gamma-1}-1)}{(\gamma-1)a^{2}_{\infty}}\Big)}}
-\frac{v^2}{2}-\frac{\rho^{\gamma-1}-1}{(\gamma-1)a^2_{\infty}}\Bigg\|_{L^{1}((-\infty,b(x)))}\\[5pt]
&\leq \hat{C}^*_{1}\big\|(\rho^{\tau}-\rho, v^{(\tau)}-v)(x,\cdot)\big\|_{L^{1}((-\infty,b(x)))}\\[5pt]
&\leq  \hat{C}^*_{1}C_1x\tau^2,
\end{split}
\end{eqnarray*}
by choosing $\tau^*>0$ and $\epsilon^*>0$ sufficiently small, where $ \hat{C}^*_{1}$ depends only on $\underline{U}$ and $a_{\infty}$. Taking $C^*_{2}= (\hat{C}^*_{1}+1)C_{1}$, we can get estimate \eqref{eq:1.28}.
This completes the proof of Theorem \ref{thm:1.1}. $\Box$

\subsection{Proof of Theorem \ref{thm:1.2}}
We are now going to prove Theorem \ref{thm:1.2}. Let us first consider solutions $U^{(\tau)}_{w}$ to the initial-boundary value problem $\eqref{eq:1.34}_{1}$ and \eqref{eq:1.33} with initial data $U^{(\tau)}_{w}(0,y)=U_{w,0}(y)$, and solutions $U_{w}$ to the initial-boundary value problem $\eqref{eq:1.37}_{1}$ and \eqref{eq:1.36} with initial data $U_{w}(0,y)=U_{w,0}(y)$.

Because the perturbation $U_{w,0}(y)-\underline{U}_{w}$ is compactly supported, one only needs to consider the problems governed by equations $\eqref{eq:1.34}_{1}$ and $\eqref{eq:1.37}_{1}$, respectively, with the same initial data $U_{w,0}$.
To this end, let $U^{(\tau), \nu}_{w}$ and $U^{\nu}_{w}$ be the approximate solutions to problems $\eqref{eq:1.34}_{1}$ and $\eqref{eq:1.37}_{1}$ with the same initial data $U_{0}$, respectively,
where the properties as stated in Remarks \ref{rem:4.2}-\ref{rem:4.3} hold.

Let $\mathcal{P}^{(\tau)}_{*}$ be the standard Riemann semigroup generated by Cauchy problem for system $\eqref{eq:1.34}_{1}$, that is, problem
\begin{eqnarray*}
\left\{
\begin{array}{llll}
Equations \quad \eqref{eq:1.34}_{1}, & in\quad\ \Omega_{w}\cap \{x>\underline{x}\}\\[5pt]
U^{(\tau)}=U^{(\tau)}_1,  & on\quad\ \Omega_{w}\cap \{x=\underline{x}\},
\end{array}
\right.
\end{eqnarray*}
has a unique solution $U^{(\tau)}(x)=\mathcal{P}^{(\tau)}_{*}(x-\underline{x})(U^{(\tau)}_1)$ for $\underline{x}\geq\ell_{w}$.

Denote by $O(1)$ the universal positive constant depending only on $\underline{U}_{w}$ and $a_{\infty}$ throughout this subsection. We claim that
\begin{eqnarray}\label{eq:5.39}
\begin{split}
\|\mathcal{P}^{(\tau)}_{*}(s)(U^{\nu}_{w}(\xi,\cdot))-U^{\nu}_{w}(\xi+s,\cdot)\|_{L^{1}}
\leq O(1)s\big(T.V.\{U^{\nu}_{w}(\xi, \cdot); \mathbb{R}\}(\tau^2+\nu^{-1})+2^{-\nu}\big),
\end{split}
\end{eqnarray}
for $s>0$ sufficiently small and $\xi\geq \ell_{w}$.

Suppose that $U^{\nu}_{w}(x,y)$ has only one jump at the point $(\xi, y_{I})$.
Let
\begin{eqnarray}\label{eq:5.40}
U_{L}\doteq U^{\nu}_{w}(\xi, y_{I}-),\quad U_{R}\doteq U^{\nu}_{w}(\xi, y_{I}+).
\end{eqnarray}

We define
\begin{eqnarray}\label{eq:5.41}
U^{\nu}_{w}(\xi,y)=
\left\{
\begin{array}{llll}
U_{R}, \quad & y>y_{I},\\[5pt]
U_L, \quad & y<y_{I},
\end{array}
\right.
\end{eqnarray}
and define
\begin{eqnarray}\label{eq:5.42}
U^{\nu}_{w}(\xi+s, y)=
\left\{
\begin{array}{llll}
U_{R}, \quad & y>y_{I}+\dot{y}_{\alpha_k}s,\\[5pt]
U_L, \quad & y<y_{I}+\dot{y}_{\alpha_k}s,
\end{array}
\right.
\end{eqnarray}
where $\alpha_{k}\in \mathcal{S}\cup\mathcal{R}$ is the 
physical wave of the $U^{\nu}_{w}(\xi+s, y)$ with speed $\dot{y}_{\alpha_k}$ satisfying $ |\dot{y}_{\alpha_k}|\leq \hat{\lambda}$ for $k=1,2$.

Then, by Remark \ref{rem:4.2}, we know that $\mathcal{P}^{(\tau)}_{*}(\xi+s)(U^{\nu}_{w}(\xi,y))$
is the approximate solution at $x=\xi+s$ obtained by the Riemann solver with Riemann data $U_{L}$ and $U_{R}$, satisfying
\begin{eqnarray}\label{eq:5.43}
\Phi(\beta_1,\beta_2; U_L,\tau^2)=\Phi_{k}(\alpha_k;U_L).
\end{eqnarray}

Then, following the proof of Proposition \ref{prop:3.1}, we know that the solutions to equation \eqref{eq:5.43} satisfy
\begin{eqnarray}\label{eq:5.44}
\begin{split}
\beta_{j}=\delta_{jk}\alpha_{k}+O(1)|\alpha_{k}|\tau^{2}, \qquad j,\ k=1\ \mbox{or}\ 2,
\end{split}
\end{eqnarray}
where $\delta_{jk}$ is the Kronecker symbol.

If $U_{L}$ and $U_{R}$ are connected by non-physical wave $\alpha_{NP}\in\mathcal{NP}$ with strength $\alpha_{NP}=|U_{L}-U_{R}|$, then we have
\begin{eqnarray}\label{eq:5.45}
\begin{split}
\beta_{j}=O(1)\alpha_{NP}, \qquad \mbox{for}\quad j=1,2.
\end{split}
\end{eqnarray}

Then, following the proof of Lemma \ref{lem:5.1} and by estimates \eqref{eq:5.44}-\eqref{eq:5.45}, for $s>0$ sufficiently small, we obtain
\begin{eqnarray}\label{eq:5.46}
\begin{split}
&\big\|\mathcal{P}^{(\tau)}_{*}(s)(U^{\nu}_{w}(\xi,\cdot))-U^{\nu}_{w}(\xi+s,\cdot)\big\|_{L^1(y_{I}-\hat{\lambda}s,y_{I}+\hat{\lambda}s)}\\[5pt]
&\qquad\qquad\leq O(1)\big((\tau^{2}+\nu^{-1})|\alpha_{k}|+\alpha_{NP}\big)s.
\end{split}
\end{eqnarray}

Then by an induction argument, if $U^{\nu}_{w}(\xi+s, y)$ contains more than one waves, by the same way as done in the proof of Proposition \ref{prop:3.2}, we get
\begin{eqnarray}\label{eq:5.47}
\begin{split}
\beta_{j}=\alpha_{j}+O(1)\big(\sum_{k=1,2}|\alpha_{k}|\big)\tau^{2}+O(1)\alpha_{NP}, \qquad \mbox{for}\quad j=1,2.
\end{split}
\end{eqnarray}

Thus, by \eqref{eq:5.47}, we can derive
\begin{eqnarray}\label{eq:5.48}
\begin{split}
&\big\|\mathcal{P}^{(\tau)}_{*}(s)(U^{\nu}_{w}(\xi,\cdot))-U^{\nu}_{w}(\xi+s,\cdot)\big\|_{L^1(y_{I}-\hat{\lambda}s,y_{I}+\hat{\lambda}s)}\\[5pt]
&\qquad\qquad\quad\quad  \leq O(1)\Big(\big(\tau^{2}+\nu^{-1}\big)\big(|\alpha_{1}|+|\alpha_2|\big)+\alpha_{NP}\Big)s.
\end{split}
\end{eqnarray}

Now, let us consider the more general case. Suppose that the jumps of
$U^{\nu}_{w}(\xi,\cdot)$ locate at $-\infty<y_{N}<\cdot\cdot\cdot <y_{2}<y_{1}<+\infty$ . Denoted by $\mathcal{J}(U^{\nu}_{w})=\mathcal{S}\cup\mathcal{R}\cup\mathcal{NP}$ the set of indices $\alpha\in\{1,2, \cdot\cdot\cdot, N\}$ such that $U^{\nu}_{w}(\xi,y_{\alpha}+)$ and $U^{\nu}_{w}(\xi,y_{\alpha}-)$ are connected by a shock wave (or rarefaction front) with strength $\alpha$, or
are connected by a non-physical front with strength $\alpha_{NP}$.
Take $s>0$ sufficiently small so that there is no wave interaction on $(\xi,\xi+s)$ for $U^{\nu}_{w}$. Then, by estimates \eqref{eq:5.46} and \eqref{eq:5.48}, we obtain
\begin{eqnarray*}
\begin{split}
&\quad \big\|\mathcal{P}^{(\tau)}_{*}(s)(U^{\nu}_{w}(\xi,\cdot))-U^{\nu}_{w}(\xi+s,\cdot)\big\|_{L^{1}(\mathbb{R})}\\[5pt]
=&\sum_{\alpha\in J(U^{\nu}_{w})}\big\|\mathcal{P}^{(\tau)}_{*}(s)(U^{\nu}_{w}(\xi,\cdot))-U^{\nu}_{w}(\xi+s,\cdot)\big\|_{L^{1}(y_{\alpha}-\eta, y_{\alpha}+\eta)} \\[5pt]
\leq& O(1)\big(\tau^2+\nu^{-1}\big)\Big(\sum _{\alpha\in \mathcal{S}\cup \mathcal{R}}|\alpha|\Big)s+O(1)\Big(\sum _{\alpha\in \mathcal{NP}}\alpha_{NP}\Big)s\\[5pt]
\leq& O(1)s\big(T.V.\{U^{\nu}_{w}(\xi, \cdot); \mathbb{R}\}(\tau^2+\nu^{-1})+2^{-\nu}\big),
\end{split}
\end{eqnarray*}
where $\eta=\frac{1}{2}\min_{1\leq j\leq N-1}\big\{ y_j-y_{j+1}\big\}$.

Therefore, via the semigroup formula for $\mathcal{P}^{(\tau)}_{*}$ as stated in \cite[Theorem 2.9]{bressan} and with the estimate \eqref{eq:5.39}, we can establish that for $x>\ell_{w}$,
\begin{eqnarray}\label{eq:5.49}
\begin{split}
&\|\mathcal{P}^{(\tau)}_{*}(x-\ell_{w})(U^{\nu}_{w}(\ell_{w},\cdot))-U^{\nu}_{w}(x,\cdot)\|_{L^{1}(\mathbb{R})}\\[5pt]
&\quad\leq L^{*}\int^{x}_{\ell_{w}}\lim\inf_{s\rightarrow 0+}\frac{\|\mathcal{P}^{(\tau)}_{*}(s)(U^{\nu}_{w}(\xi,\cdot))-U^{\nu}_{w}(s+\xi,\cdot)\|_{L^{1}(\mathbb{R})}}{s}\rm d\xi \\[5pt]
&\quad\leq O(1)L^{*}\int^{x}_{\ell_{w}}\Big(T.V.\{U^{\nu}_{w}(\xi, \cdot); \mathbb{R}\}(\tau^2+\nu^{-1})+2^{-\nu}\Big)\rm d\xi.
\end{split}
\end{eqnarray}

From Remark \ref{rem:4.2} and estimate \eqref{eq:5.49}, we derive that
\begin{eqnarray}\label{eq:5.50}
\begin{split}
&\quad\ \|\mathcal{P}^{(\tau)}_{*}(x-\ell_{w})(U_{w}(\ell_{w},\cdot))-U_{w}(x,\cdot)\|_{L^{1}(\mathbb{R})}\\[5pt]
&\leq \|\mathcal{P}^{(\tau)}_{*}(x-\ell_{w})(U_{w}(\ell_{w},\cdot))-\mathcal{P}^{(\tau)}_{*}(x-\ell_{w})(U^{\nu}_{w}(\ell_{w},\cdot))\|_{L^{1}(\mathbb{R})}\\[5pt]
&\quad+\|\mathcal{P}^{(\tau)}_{*}(x-\ell_{w})(U^{\nu}(\ell_{w},\cdot))-U^{\nu}(x,\cdot)\|_{L^{1}(\mathbb{R})}
+\|U_{w}(x,\cdot)-U^{\nu}_{w}(x,\cdot)\|_{L^{1}(\mathbb{R})}\\[5pt]
&\leq L^{*}\|U_{w}(\ell_{w},\cdot)-U^{\nu}_{w}(\ell_{w},\cdot)\|_{L^{1}(\mathbb{R})}+\|U_{w}(x,\cdot)-U^{\nu}_{w}(x,\cdot)\|_{L^{1}(\mathbb{R})}\\[5pt]
&\quad+O(1)L^{*}\int^{x}_{\ell_{w}}\Big(T.V.\{U^{\nu}_{w}(\xi, \cdot); \mathbb{R}\}(\tau^2+\nu^{-1})+2^{-\nu}\Big)\rm d\xi.
\end{split}
\end{eqnarray}

Taking $\nu\rightarrow+\infty$ in \eqref{eq:5.50}, it follows from Remark \ref{rem:4.3} that
\begin{eqnarray}\label{eq:5.51}
\begin{split}
\|\mathcal{P}^{(\tau)}_{*}(x-\ell_{w})(U_{w}(\ell_{w},\cdot))-U_{w}(x,\cdot)\|_{L^{1}(\mathbb{R})}
&\leq O(1)\tau^2\int^{x}_{\ell_{w}}T.V.\{U_{w}(\xi, \cdot); \mathbb{R}\}{\rm d}\xi\\[5pt]
&\leq O(1)\tau^2\int^{x}_{\ell_{w}}\xi^{-\frac{1}{2}}{\rm d\xi} \\[5pt]
&\leq O(1)x^{\frac{1}{2}}\tau^2.
\end{split}
\end{eqnarray}

On the other hand, under the assumptions in Theorem \ref{thm:1.2} and in \eqref{eq:1.40}, we can apply Theorem \ref{thm:1.1} for $\epsilon>0$ and $\tau>0$ sufficiently small, to get that
\begin{eqnarray}\label{eq:5.52}
\begin{split}
& \|U^{(\tau)}_{w}(x,\cdot)-\mathcal{P}^{(\tau)}_{*}(x-\ell_{w})(U_{w}(x,\cdot))\|_{L^{1}(\mathbb{R})}\\[5pt]
&\quad\leq \|\mathcal{P}^{(\tau)}_{*}(x-\ell_{w})\mathcal{P}^{(\tau)}_{*}(\ell_{w})(U^{(\tau)}_{w}(0,\cdot))
-\mathcal{P}^{(\tau)}_{*}(x-\ell_{w})(U_{w}(\ell_{w},\cdot))\|_{L^{1}(\mathbb{R})}\\[5pt]
&\quad \leq L^{*}\|\mathcal{P}^{(\tau)}_{*}(\ell_{w})(U^{(\tau)}_{w}(0,\cdot))-U_{w}(\ell_{w},\cdot)\|_{L^{1}(\mathbb{R})}\\[5pt]
&\quad \leq O(1)\ell_{w}\tau^2.
\end{split}
\end{eqnarray}

Therefore, for $\epsilon>0$ and $\tau>0$ sufficiently small, it holds
\begin{eqnarray}\label{eq:5.53}
\begin{split}
\|U^{(\tau)}_{w}(x,\cdot)-U_{w}(x,\cdot)\|_{L^{1}}&\leq \|U^{(\tau)}_{w}(x,\cdot)-\mathcal{P}^{(\tau)}_{*}(x-\ell_{w})(U_{w}(x,\cdot))\|_{L^{1}(\mathbb{R})}\\[5pt]
&\qquad+\|\mathcal{P}^{(\tau)}_{*}(x-\ell_{w})(U_{w}(\ell_{w},\cdot))-U_{w}(x,\cdot)\|_{L^{1}(\mathbb{R})}\\[5pt]
&\leq O(1)\ell_{w}\tau^2+O(1)x^{\frac{1}{2}}\tau^2.
\end{split}
\end{eqnarray}

Thus, by estimate \eqref{eq:5.53}, we can further obtain for $\ell_{w}<x< O(\tau^{-1})$ that
\begin{eqnarray}\label{eq:5.54}
\begin{split}
\sup_{\ell_{w}<x< O(\tau^{-1})}\|U^{(\tau)}_{w}(x,\cdot)-U_{w}(x,\cdot)\|_{L^{1}(\mathbb{R})}
\leq O(1)\ell_{w}\tau^2+O(1)\tau^{-\frac{1}{2}}\tau^2\leq O(1)\tau^{\frac{3}{2}}.
\end{split}
\end{eqnarray}

Then, by estimates \eqref{eq:5.54}, and equations $\eqref{eq:1.34}_{2}$ and $\eqref{eq:1.37}_{2}$, we can establish estimate \eqref{eq:1.41} by choosing constant $C^{*}_{3}>0$ independent on $\tau$ and $\ell_{w}$. This completes the proof of Theorem \ref{thm:1.2}. $\Box$

\bigskip
\appendix
\section{Reflection of weak waves on the boundary}

In the appendix, we will state some results on the reflection of weak waves on the boundary. In the following, we will use the notations given in Section 3.2.
\vspace{-3mm}
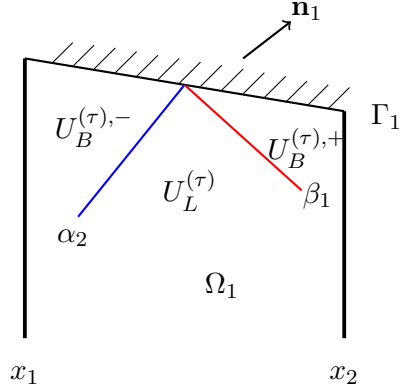
\begin{figure}[ht]
\begin{center}
\begin{tikzpicture}[scale=0.7]
\draw [line width=0.05cm] (-3.5,-3.8) --(-3.5,1.5);
\draw [line width=0.05cm] (2.5,-3.8) --(2.5,0.5);
\draw [thick] (-3.5,1.5)--(2.5,0.5);

\draw [thin] (-3, 1.4) --(-2.6, 1.8);
\draw [thin] (-2.6, 1.35) --(-2.2, 1.75);
\draw [thin] (-2.2, 1.30) --(-1.8, 1.70);
\draw [thin] (-1.8, 1.23) --(-1.4, 1.63);
\draw [thin] (-1.4, 1.16) --(-1.0, 1.56);
\draw [thin] (-1.0, 1.10) --(-0.6, 1.50);
\draw [thin] (-0.6, 1.03) --(-0.2, 1.43);
\draw [thin] (-0.2, 0.97) --(0.2, 1.37);
\draw [thin] (0.2, 0.9) --(0.6, 1.30);
\draw [thin] (0.6, 0.83) --(1, 1.23);
\draw [thin] (1, 0.76) --(1.4, 1.16);
\draw [thin] (1.4, 0.67) --(1.8, 1.07);
\draw [thin] (1.8, 0.60) --(2.2, 1.0);
\draw [thin] (2.2, 0.55) --(2.6, 0.95);

\draw [thick][blue](-2.5,-1.5)--(-0.5,1);
\draw [thick][red](-0.5,1)--(1.7,-1.0);

\draw [thick][<-] (1.5,2.2)--(0.6,1.5);

\node at (1.8, 2.4){$\mathbf{n}_{1}$};

\node at (0.2, -2.8){$\Omega_{1}$};
\node at (3.3, 0.4){$\Gamma_{1}$};

\node at (-2.6, -1.9){$\alpha_2$};
\node at (2.0, -1.1){$\beta_1$};

\node at (-0.4, -1.0){$U^{(\tau)}_{L}$};
\node at (-2.2, 0.2){$U^{(\tau),-}_{B}$};
\node at (1.8, -0.2){$U^{(\tau),+}_{B}$};

\node at (-3.5, -4.5){$x_{1}$};
\node at (2.5, -4.5){$x_{2}$};
\end{tikzpicture}
\end{center}
\caption{Reflection of weak waves on the boundary for system \eqref{eq:1.16}}\label{fig-A1}
\end{figure}
Let us first consider the reflection of weak wave for system \eqref{eq:1.16} on the boundary. 
\begin{lemma}\label{lem-A1}
Let $U^{(\tau),-}_{B}=(\rho^{(\tau),-}_{B}, v^{(\tau),-}_{B})^{\top}$, $U^{(\tau)}_L=(\rho^{(\tau)}_{L}, v^{(\tau)}_{L})^{\top}\in \mathcal{O}_{\epsilon'_{2}}(\underline{U})$
be two constant states near $\Gamma_{1}$ satisfying
\begin{equation}\label{eq-A1}
U^{(\tau),-}_{B}=\Phi_{2}(\alpha_2; U^{(\tau)}_L,\tau^2), \qquad
\Big((1+\tau^{2}u^{(\tau),-}_{B}), v^{(\tau),-}_{B}\Big)\cdot\mathbf{n}_1=0,
\end{equation}
for $\tau\in(0,\tau'_{2})$, where $u^{(\tau),-}_{B}=u^{(\tau),-}_{B}(\rho^{(\tau),-}_{B},v^{(\tau),-}_{B},\tau^2)$ through the relation \eqref{eq:1.12},
where $\mathbf{n}_1=(\sin\theta_1, -\cos\theta_1)$.
Then, for constant state $U^{(\tau),+}_{B}=(\rho^{(\tau),+}_{B}, v^{(\tau),+}_{B})^{\top}\in \mathcal{O}_{\epsilon'_{2}}(\underline{U})$ and $\tau\in(0,\tau'_{2})$ with
\begin{eqnarray}\label{eq-A2}
\begin{split}
U^{(\tau),+}_{B}=\Phi_{1}(\beta_1;U^{(\tau)}_L,\tau^2), \quad \Big((1+\tau^{2}u^{(\tau),+}_{B}), v^{(\tau),+}_{B}\Big)\cdot\mathbf{n}_1=0,
\end{split}
\end{eqnarray}
and $u^{(\tau),+}_{B}=u^{(\tau),+}_{B}(\rho^{(\tau),+}_{B},v^{(\tau),+}_{B},\tau^2)$ through relation \eqref{eq:1.12}, it holds that
\begin{equation}\label{eq-A3}
\beta_{1}=K^{(\tau)}_{b} \alpha_{2},
\end{equation}
where $K^{(\tau)}_{b}$ is a $C^{2}$-function of $(\alpha_{2}, \tau^{2}, U^{(\tau)}_L)$ satisfying
\begin{equation}\label{eq-A4}
K^{(\tau)}_{b}|_{\beta_1=\alpha_2=0,\ U^{(\tau)}_{L}=\underline{U}}=1.
\end{equation}
\end{lemma}

\begin{proof}
By \eqref{eq-A1} and \eqref{eq-A2}, we know that $\beta_1$ can be solved as a $C^2$-function of $\alpha_2, U^{(\tau)}_L, \tau^2$ from the following equation
\begin{eqnarray*}
\begin{split}
\mathcal{L}_{B}(\alpha_2,\beta_1;U^{(\tau)}_L,\tau^2)&\doteq\Big(1+\tau^{2}u^{(\tau),-}_{B}\Big)\Phi^{(2)}_{1}(\beta_{1}; U^{(\tau)}_{L},\tau^2)\\[5pt]
&\qquad -\Big(1+\tau^{2}u^{(\tau),+}_{B}\Big)\Phi^{(2)}_{2}(\alpha_2; U^{(\tau)}_{L},\tau^2)=0.
\end{split}
\end{eqnarray*}

Then, by the Taylor formula, we have estimate \eqref{eq-A3}. To estimate  $K^{(\tau)}_b$, notice that
\begin{eqnarray*}
\begin{split}
&K^{(\tau)}_{b}\Big|_{\beta_1=\alpha_2=0,\ U^{(\tau)}_{L}=\underline{U}}\\[5pt]
&=-\frac{\partial_{\alpha_{2}}\mathcal{L}_{B}(\alpha_2,\beta_1;U^{(\tau)}_L,\tau^2)}
{\partial_{\beta_{1}}\mathcal{L}_{B}(\alpha_2,\beta_1;U^{(\tau)}_L,\tau^2)}\Bigg|_{\beta_1=\alpha_2=0, U^{(\tau)}_{L}=\underline{U}}\\[5pt]
&=-\frac{(1+\tau^{2}\partial_{\alpha_2}u^{(\tau),-}_{B})\Phi^{(2)}_1(\beta_1;U^{(\tau)}_L,\tau^2)
-(1+\tau^{2}u^{(\tau),+}_{B})\partial_{\alpha_2}\Phi^{(2)}_2(\alpha_2;U^{(\tau)}_L,\tau^2)}
{(1+\tau^{2}u^{(\tau),-}_{B})\partial_{\beta_1}\Phi^{(2)}_1(\beta_1;U^{(\tau)}_L,\tau^2)
-(1+\tau^{2}\partial_{\beta_1}u^{(\tau),+}_{B})\Phi^{(2)}_2(\alpha_2;U^{(\tau)}_L,\tau^2)}\Bigg|_{\beta_{1}=\alpha_{2}=0, U^{(\tau)}_{L}=\underline{U}}\\[5pt]
&=\frac{\mathbf{r}^{(2)}_{2}(\underline{U},\tau^{2})}{\mathbf{r}^{(2)}_{1}(\underline{U},\tau^{2})}\\[5pt]
&=1.
\end{split}
\end{eqnarray*}
 This completes the proof of the lemma.
\end{proof}

\vspace{-3mm}
\begin{figure}[ht]
\begin{center}
\begin{tikzpicture}[scale=0.7]
\draw [line width=0.05cm] (-3.5,-3.8) --(-3.5,1.5);
\draw [line width=0.05cm] (2.5,-3.8) --(2.5,0.5);
\draw [thick] (-3.5,1.5)--(2.5,0.5);

\draw [thin] (-3, 1.4) --(-2.6, 1.8);
\draw [thin] (-2.6, 1.35) --(-2.2, 1.75);
\draw [thin] (-2.2, 1.30) --(-1.8, 1.70);
\draw [thin] (-1.8, 1.23) --(-1.4, 1.63);
\draw [thin] (-1.4, 1.16) --(-1.0, 1.56);
\draw [thin] (-1.0, 1.10) --(-0.6, 1.50);
\draw [thin] (-0.6, 1.03) --(-0.2, 1.43);
\draw [thin] (-0.2, 0.97) --(0.2, 1.37);
\draw [thin] (0.2, 0.9) --(0.6, 1.30);
\draw [thin] (0.6, 0.83) --(1, 1.23);
\draw [thin] (1, 0.76) --(1.4, 1.16);
\draw [thin] (1.4, 0.67) --(1.8, 1.07);
\draw [thin] (1.8, 0.60) --(2.2, 1.0);
\draw [thin] (2.2, 0.55) --(2.6, 0.95);

\draw [thick][blue](-2.5,-1.5)--(-0.5,1);
\draw [thick][red](-0.5,1)--(1.7,-1.0);

\draw [thick][<-] (1.5,2.2)--(0.6,1.5);

\node at (1.8, 2.4){$\mathbf{n}_{1}$};

\node at (0.2, -2.8){$\Omega_{1}$};
\node at (3.3, 0.4){$\Gamma_{1}$};

\node at (-2.6, -1.9){$\alpha_2$};
\node at (2.0, -1.1){$\beta_1$};

\node at (-0.4, -1.0){$U_{L}$};
\node at (-2.2, 0.2){$U^{-}_{B}$};
\node at (1.8, -0.2){$U^{+}_{B}$};

\node at (-3.5, -4.5){$x_{1}$};
\node at (2.5, -4.5){$x_{2}$};
\end{tikzpicture}
\end{center}
\caption{Reflection of weak wave on the boundary for the system \eqref{eq:1.20}}\label{fig-A2}
\end{figure}

Next, in the similar way as done in the proof of Lemma \ref{lem-A1}, we can establish a lemma on the reflection of weak wave for system \eqref{eq:1.20} on the boundary. 
\begin{lemma} \label{lem-A2}
Let $U^{-}_{B}=(\rho^{-}_{B}, v^{-}_{B})^{\top},\ U_L=(\rho_{L}, v_{L})^{\top}\in \mathcal{O}_{\epsilon'_{2}}(\underline{U})$
be two constant states 
satisfying that
\begin{equation*}
U^{-}_{B}=\Phi_{1}(\tilde{\alpha}_{2}; U_L), \quad   v^{-}_{B}=\tan(\theta_{1}).
\end{equation*}
Then, for the constant state $U^{+}_{B}\in \mathcal{O}_{\epsilon'_{2}}(\underline{U})$ with
\begin{eqnarray*}
\begin{split}
U^{+}_{B}=\Phi_{1}(\tilde{\beta}_{1};U_b), \ \ \ v^{+}_{B}=\tan(\theta_{1}),
\end{split}
\end{eqnarray*}
it holds that 
\begin{equation}\label{eq-A5}
\beta_{1}=\tilde{K}_{b} \alpha_{2},
\end{equation}
where $\tilde{K}_{b}$ is a $C^{2}$-function of $(\alpha_{2}, U_L)$ satisfying
\begin{equation}\label{eq-A6}
\tilde{K}_{b}|_{\beta_1=\alpha_{2}=0,\ U_{L}=\underline{U}}=1.
\end{equation}
\end{lemma}

\bigskip
\section*{Acknowledgements}
The research of Jie Kuang was supported in part by the NSFC Project 11801549, NSFC Project 11971024 and the Start-Up Research Grant from Wuhan Institute of Physics and Mathematics, Chinese Academy of Sciences, Project No.Y8S001104.
The research of Wei Xiang was supported in part by the Research Grants Council of the HKSAR, China (Project No.CityU 11303518, Project CityU 11304820 and Project CityU 11300021).
The research of Yongqian Zhang was supported in part by the NSFC Project 11421061, NSFC Project 11031001, NSFC Project 11121101,
the 111 Project B08018(China) and the Shanghai Natural Science Foundation 15ZR1403900.

\end{document}